\pdfoutput=1
\documentclass{amsart}

\usepackage{amssymb}
\usepackage{amsthm}
\usepackage{amsfonts}
\usepackage{amsmath}
\usepackage{natbib}

\newtheorem{theorem}{Theorem}
\newtheorem{lemma}[theorem]{Lemma}
\newtheorem{corollary}[theorem]{Corollary}

\newtheorem{proposition}[theorem]{Proposition}
\newtheorem{definition}[theorem]{Definition}
\newtheorem{notation}[theorem]{Notation}
\newtheorem{example}[theorem]{Example}
\newtheorem{remark}[theorem]{Remark}

\usepackage{multicol}
\usepackage{multirow}
\usepackage[utf8x]{inputenc}
\usepackage{synttree}
\usepackage{algorithm}
\usepackage[noend]{algorithmic}
\usepackage{comment}
\usepackage[pdfencoding=auto]{hyperref}
\usepackage{enumitem}
\usepackage{hyphenat}

\makeatletter
\newcommand{\eqnum}{\refstepcounter{equation}\textup{\tagform@{\theequation}}}
\makeatother

\DeclareMathOperator{\IntegralBasisForWeierstrassPolynomial}{\mathtt{MergingIntegralBases}}
\DeclareMathOperator{\MergeCoefficients}{\mathtt{CoefficientsForMerging}}

\DeclareMathOperator{\TruncatedFactor}{\mathtt{IntegralBasisElement}}
\DeclareMathOperator{\SeparateUnit}{\mathtt{SeparateUnit}}
\DeclareMathOperator{\HenselLift}{\mathtt{HenselLift}}
\DeclareMathOperator{\Splitting}{\mathtt{Splitting}}
\DeclareMathOperator{\BlockSplitting}{\mathtt{BlockSplitting}}
\DeclareMathOperator{\SegmentSplitting}{\mathtt{SegmentSplitting}}
\DeclareMathOperator{\Spec}{Spec}
\DeclareMathOperator{\Sing}{Sing}

\DeclareMathOperator{\Ann}{Ann}
\DeclareMathOperator{\Int}{Int}
\DeclareMathOperator{\Hom}{Hom}

\DeclareMathOperator{\Proj}{Proj}
\DeclareMathOperator{\sing}{sing}
\DeclareMathOperator{\TQR}{Q}

\newcommand{\Q}{{\mathbb Q}}
\newcommand{\N}{{\mathbb N}}
\newcommand{\CC}{{\mathbb C}}

\newcommand{\Z}{{\mathbb Z}}

\newcommand{\degbound}{{e}}

\begin{document}

\title{Computing integral bases via localization and Hensel lifting}

\author{Janko B\"ohm}
\address{Department of Mathematics,  University of Kaiserslautern\\
Erwin-Schr\"odinger-Str. - (67663) Kaiserslautern, Germany}
\email{boehm@mathematik.uni-kl.de}

\author{Wolfram Decker}
\address{Department of Mathematics,  University of Kaiserslautern\\
Erwin-Schr\"odinger-Str. - (67663) Kaiserslautern, Germany}
\email{decker@mathematik.uni-kl.de}

\author{Santiago Laplagne}
\address{Departamento de Matem\'atica, FCEyN, Universidad de Buenos Aires \\
Ciudad Universitaria Pabell\'on I - (C1428EGA) - Buenos Aires, Argentina}
\email{slaplagn@dm.uba.ar}

\author{Gerhard Pfister}
\address{Department of Mathematics,  University of Kaiserslautern\\
Erwin-Schr\"odinger-Str. - (67663) Kaiserslautern, Germany}
\email{pfister@mathematik.uni-kl.de}

\begin{abstract}
We present a new algorithm for computing integral bases in algebraic function
fields of one variable, or equivalently for constructing the normalization of a plane curve.
Our basic strategy makes use of the concepts of localization and completion,
together with the Chinese remainder theorem, to reduce the problem to the
task of finding integral bases for the branches of each singularity of the curve.
To solve the latter task, in turn, we work
with suitably truncated Puiseux expansions. In contrast to van Hoeij's
algorithm \citep{vanHoeij94}, which also relies on Puiseux expansions (but
pursues a different strategy), we use Hensel's lemma as a key ingredient. This
allows us at some steps of the algorithm to compute factors corresponding to
conjugacy classes of Puiseux expansions, without actually computing the
individual expansions. In this way, we make substantially less use of the
Newton-Puiseux algorithm. In addition, our algorithm is inherently parallel.
As a result, it outperforms in most cases any other algorithm known to us by
far. Typical applications are the computation of adjoint ideals \citep{BDLS2}
and, based on this, the computation of Riemann-Roch spaces and the
parametrization of rational curves.
\end{abstract}

\keywords{Normalization, integral closure, integral basis, curve singularity, Puiseux series}

\maketitle

\section{Introduction}

\label{sect:intro} Let $A$ be a reduced Noetherian ring, and let $\TQR(A)$ be
its total ring of fractions. The {\emph{normalization}} of $A$ is the integral
closure of $A$ in $\TQR(A)$. We denote the normalization by $\overline{A}$ and
call $A$ {\emph{normal}} if $A=\overline{A}$. Recall that if $A$ is a reduced
\emph{affine} (that is, finitely generated) algebra over a field $K$, then $\overline{A}$
is a finite $A$-module by the splitting of normalization  (see \citet[Theorem 1.5.20]{JP})
and Emmy Noether's finiteness result below (see \citet[Chapter 13]{Eis} for a proof):

\begin{theorem}[Emmy Noether] Let $R$ be a Noetherian domain
and let $L$ be a finite extension field of $F=\TQR(R)$. Suppose that
$R$ is an affine domain over a field $K$ or that $R$ is normal and $L$ is separable over $F$.
Then the integral closure $S$ of $R$ in $L$ is a finitely generated $R$-module.
\end{theorem}

\begin{remark}
\label{rem:free-ness-over-pid}
With notation and assumptions as in the theorem, suppose that $R$ is a PID.
Then $S$, as a finitely generated torsion-free module over a PID, is free. In fact, it is free of rank
$[L:F]$ since a set of free generators for $S$ over $R$ is also a basis for
$L=SF$ over $F$.
\end{remark}

\begin{definition}
If $R\subset T$ is a ring extension such that the integral closure $S$
of $R$ in $T$ is a free $R$-module, then we call any set of free generators for $S$
over $R$ an {\emph{integral basis}} for $S$ over $R$.
\end{definition}

In this paper, we focus on  the case where $A$ is the coordinate ring
of an algebraic curve defined over a field $K$ of characteristic zero.
More precisely, let $f\in K[X,Y]$ be an irreducible polynomial
in two variables, let $C\subset\mathbb{A}^{2}(K)$ be the affine plane curve
defined by $f$, and let
\[
A=K[C]=K[X,Y]/\langle f(X,Y)\rangle
\]
be the {\emph{coordinate ring}} of $C$. We write $x$ and $y$ for the residue
classes of $X$ and $Y$ modulo $f$, respectively. Throughout the paper, we
suppose that $f$ is monic in $Y$ (due to Noether normalization, this can
always be achieved by a linear change of coordinates). Then the
{\emph{function field}} of $C$ is
\[
K(C)=\TQR(A)=K(x)[y]=K(X)[Y]/\langle f(X,Y)\rangle,
\]
where $x$ is a separating transcendence basis of $K(C)$ over $K$, and $y$ is
integral over $K[x]$, with integrality equation $f(x,y)=0$.
Indeed, we have the isomorphism $\TQR(K[x][y]) \rightarrow  K(x)[y]$ defined by
mapping $1/h(x,y) \mapsto b(x,y) / x^c$, where $X^c = a f + b h\in K[X][Y]$
is a representation which arises from a B\'ezout identity in $K(X)[Y]$ by clearing denominators.

In particular, $A$
is integral over $K[x]$, which implies that $\overline{A}$ coincides with the
integral closure of $K[x]$ in $K(C)$. We may, hence,
represent $\overline{A}$ either by generators over $A$ or by generators over
$K[x]$. Note that by Remark \ref{rem:free-ness-over-pid}, $\overline{A}$
is a free $K[x]$-module of rank
\[
n:=\deg_{Y}(f)=[K(C):K(x)].
\]

\begin{remark}
\label{rem:spec-int-basis}
In the context outlined above, there exist
polynomials $p_{i} \in K[X][Y]$ of degree $i$ in $Y$ and polynomials
$d_i\in K[X]$ such that
\[
\left\{1,\frac{p_{1}(x,y)}{d_1(x)},\dots,\frac{p_{n-1}(x,y)}{d_{n-1}(x)}\right\}
\]
is an integral basis 
 for $\overline{A}$ over $K[x]$. In fact, such a
basis is obtained from any given set of
$K[x]$-module generators for $\overline{A}$ by unimodular row operations over
the PID $K[X]$: Represent the given generators by polynomials of type
$c_{i}=\sum_{j=0}^{n-1}c_{ij}Y^{n-1-j}$, with coefficients $c_{ij}\in K(X)$.
Then take $d$ to be the least common denominator of the $c_{ij}$,
transform the matrix $(d\cdot c_{ij})$ into Hermite normal form $(p_{ij})$,
set $\widetilde{p}_{i}=\sum_{j=0}^{n-1}p_{n-1-i ,j}Y^{n-1-j}$ for each $i=0,\dots, n-1$,
and let the ${p}_{i}(X,Y)/d_i(X)$ be obtained by reducing the
$\widetilde{p}_{i}(X,Y)/d(X)$ to lowest terms.
\end{remark}

\begin{remark}
\label{rem:spec-int-basis-II} The general normalization algorithms presented
in \citep{GLS10}, \citep{BDLPSS} are designed to
return an ideal $U\subset A$ together with an element $d\in A$ such that
$\overline{A}=\frac{1}{d}U\subset\TQR(A)$. Here, as will become clear in Section
\ref{sect:Int-basis-via-norm}, we may take
a generator of the elimination ideal $\langle \frac{\partial f}{\partial X},
\frac{\partial f}{\partial Y}, f \rangle \cap K[X]$ to represent $d$
(note that this ideal defines the $X$-coordinates
of the singularities of the curve defined by $f$ over
the algebraic closure $\overline{K}$).
If $u_{0}=d(x),u_{1},\dots,u_{r}$ generate the ideal $U$, the
$y^{i}u_{j}(x,y)/d(x)$, $0\leq i\leq n-1$, $0\leq j\leq r$, generate
$\overline{A}$ over $K[x]$. An integral basis is then obtained by operations
as described in the remark above.
\end{remark}

\begin{remark}
\label{rem groeb lex} In practical terms, $u_{0},\dots, u_{r}$ are given as
polynomials in $K[X,Y]$ of $Y$-degree at most $n-1$. If these polynomials,
together with $f$, form a Gr\"obner basis with respect to the lexicographical
ordering, taking $Y>X$,
then already the elements $y^{i}u_{j}(x,y)/d(x)$, $0\leq i\leq n-1-\deg
(u_{j})$, $0\leq j\leq r$, generate $\overline{A}$ over $K[x]$.
\end{remark}

\begin{example}
\label{exampleCusp} Consider the standard cusp: Let
\[
A=K[x,y]=K[X,Y]/\langle Y^{3}-X^{2}\rangle.
\]
As a module over $A$, we may represent $\overline{A}$ as
\[
\overline{A}=A \cdot\frac{y^{2}}{x}+A \cdot1 = \frac{1}{x}\left\langle y^{2},
x\right\rangle _{A}%
\]
(see \citet[Example 2.5]{GLS10}). Considering $\overline{A}$ over $K[x]$, we
get
\[
\overline{A}=K[x]\cdot\frac{y^{2}}{x}+K[x]\cdot y\cdot\frac{y^{2}}%
{x}+K[x]\cdot y^{2}\cdot\frac{y^{2}}{x}+K[x]\cdot1+K[x]\cdot y+K[x]\cdot
y^{2}.
\]
Since $y^{3}=x^{2}$ and $K[x]\cdot y^{2}\subset K[x]\cdot y^{2}/x$, we have
\[
\overline{A}=K[x]\cdot\frac{y^{2}}{x} \oplus K[x]\cdot1\oplus K[x]\cdot y.
\]
Hence, $\{1,y,y^{2}/x\}$ is an integral basis as in Remark
\ref{rem:spec-int-basis}.
\end{example}

The algorithms in \citep{GLS10}, \citep{BDLPSS} work for any reduced affine
algebra $A$ over a perfect field. They rely on the Grauert and Remmert
normalization criterion which applies in global or local settings (see
\citet{GR}, \citet[Prop.~3.6.5]{GP}, \citet[Prop.~3.3]{BDLPSS}): whereas the
algorithm in \citep{GLS10} is of global nature, the idea in \citep{BDLPSS} is to
consider a finite stratification of the singular locus $\Sing(A)$, apply a
local version of the normalization algorithm at each stratum, and find
$\overline{A}$ by putting the resulting local contributions together.
If $A$ is the coordinate ring of a curve, then $\Sing(A)$ is finite, and
we may stratify it by considering each $P\in\Sing(A)$ separately.
Computing an integral basis for $\overline{A}$ over
$K[x]$ is then equivalent to computing a local contribution to
$\overline{A}$ at each $P$.

In this paper, we present a  new method for the latter task which is custom-made
for our case of interest here. From now on, let $A=K[C] = K[x][y]$ be the coordinate
ring of a plane curve $C$ with notation and assumptions as before. To describe the
main ideas of the new method, we suppose  for simplicity that the prime ideal
$P\in\Sing(A)$ under consideration defines a $K$-rational singularity of $C$,
and that this singularity is the origin, that is, $P=\langle x,y \rangle$.

Consider the completion of $A$ at ${\langle x \rangle}$:
$$
\widehat{A} = K[[x]][y] = K[[X]][Y] /\langle f \rangle.
$$
Since $A$ is a reduced excellent ring, $\widehat{A}$
is reduced as well (see \citet[Section 7.8]{EGA4}). 
The normalization $\overline{\widehat{A}}$ in turn is a free $K[[x]]$-module
of rank $n =\deg_{Y}(f)$. Indeed, consider the decomposition
\begin{equation}
f=f_{0}\widetilde{f}=f_{0} f_{1}\cdots f_{r}
\label{equation:branches}%
\end{equation}
given by the Weierstrass preparation theorem (see \citet{Abhyankar}, \citet{JP}).
Here, $f_{0}\in K[[X]][Y]$ is a unit in $K[[X,Y]]$, and $f_{1}, \dots, f_{r}\in K[[X]][Y]$ are irreducible
Weierstrass polynomials to which we refer as the \emph{branches} of $f$ (over $K$, centered at
the origin). It follows from Remark \ref{rem:free-ness-over-pid} that if $g$ is one of
the branches, then the normalization of the ring $B=K[[x]][y] =K[[X]][Y]/\langle g \rangle$ is a free
$K[[x]]$-module, a result which extends to the cases $g= \widetilde{f}$ and $g=f$
by the splitting of normalization. In each case, we refer to an integral basis for $B$
over $K[[x]]$ as an \emph{integral basis for $g$}. Our new algorithm
is designed so that it computes such bases $\mathcal{B}_{f_1}, \dots, \mathcal{B}_{f_r}$
for the branches of $f$, and so that it finds the desired local contribution to $\overline{A}$
at $P$ from the $\mathcal{B}_{f_i}$.

We present more details of the algorithm while outlining the structure of our paper.

To fix our ideas, in Sections \ref{sect:Int-basis-via-norm} and
\ref{sect:noem-via-loc}, we give a more detailed account of the Grauert and Remmert
type algorithms. Furthermore, we discuss an efficient criterion for detecting
whether a given point is the only singularity of the curve under
consideration. In Section \ref{sect:basic-puiseux}, we review Puiseux
expansions and their connection to integrality via valuations.

A crucial theoretical result proved in Section \ref{sect:norm-via-loc-and-compl-the-theory}  is
that $f_1,\dots, f_r$ admit integral bases of type
\[
\mathcal{B}_{f_i}= \left\{1=p^{(i)}_{0},\frac{p^{(i)}_{1}(x,y)}{x^{e^{(i)}_{1}}},\ldots,\frac{p^{(i)}_{m-1}(x,y)}{x^{e^{(i)}_{m-1}}}\right\} ,%
\]
with monic polynomials  $p^{(i)}_d\in K[X][Y]$ of degree $d$ in $Y$:  We then speak of integral bases of
\emph{monic triangular type}. Using an explicit version of the splitting of  normalization via the
Chinese remainder theorem, we show how such bases fit together to an integral basis
for $\widetilde{f} = f_{1}\cdots f_{r}$. Any such basis, in turn, can be transformed into
an integral basis $\mathcal{B}_{\widetilde{f}}$ for $\widetilde{f}$  of monic triangular type by computing
a Hermite normal form over the PID $K[[x]]$.
Going one step further, in Section \ref{sect:local-contribution}, we show how to turn
$\mathcal{B}_{\widetilde{f}}$ into  an integral basis $\mathcal{B}_{{f}}$ for ${f}$
which is of monic triangular type: Under the additional assumption
that the origin is the only singularity of $C$ with $X$-coordinate zero,
we prove that the $K[x]$-module generated by the elements of $\mathcal{B}_{{f}}$ is a ring
which is a local contribution to $A$ at $P$.

How to actually construct integral bases $\mathcal{B}_{f_i}$ for the branches is a topic of Section
\ref{sect:loc-contr-via-Hensel}, which is the algorithmic heart of the paper. Working
with suitably truncated Puiseux series, we describe an effective way to obtain the
$\mathcal{B}_{f_i}$. Though inspired by van Hoeij's paper
\citep{vanHoeij94}, our strategy is different, with Hensel lifting providing a crucial new ingredient.
In summarizing the whole algorithm, we also show that we can actually avoid the use of the
Chinese remainder algorithm when computing $\mathcal{B}_{\widetilde{f}}$ from
the $\mathcal{B}_{f_i}$. This improves the  performance of the algorithm considerably.

We have implemented our algorithm in the computer algebra system
\textsc{Singular} \citep{GPS}. In Section \ref{sec timings}, we compare its
 performance with that of the local to global Grauert and Remmert type
algorithm. We also give timings for the implementation of van Hoeij's
algorithm in \textsc{Maple} and for the variant of the Round 2 algorithm
implemented in \textsc{Magma}.

\section{The Global Normalization Algorithm}

\label{sect:Int-basis-via-norm}

In this section, we review the global version of the normalization algorithm.
To begin with, we fix our notation and present some general facts on
normalization. For this, $A$ may be any reduced Noetherian ring. We write%

\[
\Spec(A)=\{P\subset A\mid P{\text{ prime ideal}}\}
\]
for the {\emph{spectrum}} of $A$. The {\emph{vanishing locus}} of an ideal $J$
of $A$ in $\Spec(A)$ is the set $V(J)=\{P\in\Spec(A)\mid P\supset J\}$. We
denote by
\[
N(A)=\{P\in\operatorname{Spec}(A)\mid A_{P}\text{ is not normal}\}
\]
the {\emph{non-normal locus}} of $A$, and by
\[
\operatorname{Sing}(A)=\{P\in\operatorname{Spec}(A)\mid A_{P}\text{ is not
regular}\}
\]
the {\emph{singular locus}} of $A$. Then $N(A)\subset\Sing(A)$, with equality
holding if $A$ is the coordinate ring of a curve (see \citet[Theorem 4.4.9]{JP}).

\begin{definition}
The \emph{conductor} of $A$ is
\[
\mathcal{C}_{A}=\Ann_{A}(\overline{A}/A)=\{a\in A\mid a\overline{A}\subset
A\}.
\]
\end{definition}

Note that $\mathcal{C}_{A}$ is the largest ideal of $A$ which is also an ideal
of $\overline{A}$.

To emphasize the role of the conductor, we note:

\begin{lemma}
\label{lemma:role-of-cond} Let $A$ be a reduced Noetherian ring. Then
$N(A)\subset V(\mathcal{C}_{A})$. Furthermore, $\overline{A}$ is a finite
$A$-module iff $\mathcal{C}_{A}$ contains a non-zerodivisor of $A$. In this
case, $N(A)=V(\mathcal{C}_{A})$.
\end{lemma}

Note, however, that $\mathcal{C}_{A}$ can only be computed a posteriori, once
$\overline{A}$ is already known.

\begin{definition}
\label{def:test-ideal} Let $A$ be a reduced Noetherian ring. A {\emph{test
ideal}} for $A$ is a radical ideal $J\subset A$ such that $V(\mathcal{C}%
_{A})\subset V(J)$. A {\emph{test pair}} for $A$ consists of a test ideal $J$
together with a non-zerodivisor $g\in J$ of $A$.
\end{definition}

Test pairs appear in the Grauert and Remmert normality criterion which is
fundamental to algorithmic normalization (see \citet{GR}, \citet[Prop.~3.6.5]{GP}).
The algorithm by de Jong (see \citet{deJong98},
\citet{DGPJ}) and its improvement, the algorithm by \cite{GLS10},
are based on this criterion, and apply to any reduced affine
algebra $A=L[X_{1},\dots,X_{n}]/I$ over a perfect field $L$. Initially, by means of
equidimensional decomposition, we may reduce to the case where $A$ is equidimensional.
In this case, since we work over a perfect field, the Jacobian
ideal\footnote{The {\emph{Jacobian ideal}} $M$ of $A=L[X_{1},\dots,X_{n}]/I$
is generated by the images of the $c\times c$ minors of the Jacobian matrix
$\big(\frac{\partial f_{i}}{\partial X_{j}}\big)$, where we suppose that $I$ is of
pure codimension $c$, and where $f_{1},\dots,f_{r}$ are generators for $I$. By the Jacobian
criterion, $V(M) = \Sing(A)$ (see \citet[Theorem 16.19]{Eis}).} $M$ of $A$ is
non-zero and contained in the conductor $\mathcal{C}_{A}$. This implies
that the radical $J=\sqrt{M}$ together with any non-zero divisor $g\in J$
of $A$ is a test pair (see \citet[Lemma 4.1]{GLS10}). Given such a pair (see
\citet[Remark 4.6]{GLS10} for how to find $g$), the idea of computing
$\overline{A}$ is to successively enlarge $A$ by finite ring extensions
$A_{i+1}\cong\Hom_{A_{i}}(J_{i},J_{i})\cong\frac{1}{g}(gJ_{i}:_{A_{i}}%
J_{i})\subset\overline{A}\subset\TQR(A)$, with $A_{0}=A$ and $J_{i}%
=\sqrt{JA_{i}}$, until the normality criterion by Grauert and Remmert allows
us to stop. The algorithm by Greuel et al. then returns an ideal $U\subset
A$ together with a power $d$ of $g$ such that $\overline{A}=\frac{1}
{d}U\subset\TQR(A)$.

\begin{remark}
\label{rem:choice-of denominator} If $M$ is non-zero and contained in
$\mathcal{C}_{A}$, then any non-zerodivisor $c\in M$ of $A$ is a valid
denominator: If  $\overline{A}=\frac{1} {d}U$ as above, then $c
\cdot\frac{1} {d}U=: U^{\prime}$ is an ideal of A, and $\frac{1} {d}U =
\frac{1} {c}U^{\prime}$.

\end{remark}

\begin{example}
\label{example:the global algorithm} Let $A$ be the coordinate ring of the
curve $C$ with defining polynomial $f(X,Y)=X^{5} -Y^{2}(Y-1)^{3}\in
{\mathbb{Q}}[X,Y]$. Then
\[
J:=\left\langle x, y\left(  y-1\right)  \right\rangle _{A}
\]
is the radical of the Jacobian ideal, so we can take $(J,x)$ as a test
pair. In its first step, the normalization algorithm yields
\[
A_{1}=\frac{1}{x}U_{1}=\frac{1}{x}\left\langle x,y(y-1)^{2}\right\rangle _{A}
\text{.}%
\]
In the next steps, we get
\[
A_{2}=\frac{1}{x^{2}}U_{2}=\frac{1}{x^{2}}\left\langle x^{2}%
,xy(y-1),y(y-1)^{2}\right\rangle _{A}%
\]
and%
\[
A_{3}=\frac{1}{x^{3}}U_{3}=\frac{1}{x^{3}}\left\langle x^{3},x^{2}%
y(y-1),xy(y-1)^{2},y^{2}(y-1)^{2}\right\rangle _{A}\text{.}%
\]
In the final step, we find that $A_{3}$ is normal and, hence, equal to
$\overline{A}$.
\end{example}

\section{Normalization of Curves via Localization}

\label{sect:noem-via-loc}

In this section, we discuss the local to global variant of the normalization
algorithm proposed by \cite{BDLPSS}. To simplify our
presentation, we focus on the case of a reduced Noetherian ring with a finite
singular locus (such as the coordinate ring of a curve). Our starting point
is Proposition \ref{prop:local-to-global} below which, as we will see
later on, is also fundamental to
our new algorithm. In formulating the proposition, if $P\in\Spec(A)$ and
$A\subset A^{\prime}\subset\overline{A}$ is an intermediate ring, we write
$A_{P}^{\prime}$ for the localization of $A^{\prime}$ at $A\setminus P\subset
A^{\prime}$.

\begin{proposition}
\label{prop:local-to-global} Let A be a reduced Noetherian ring with a finite
singular locus $\Sing(A)=\{P_{1},\dots,P_{s}\}$. For $i=1,\dots,s$, let an
intermediate ring $A\subset A^{(i)}\subset\overline{A}$ be given such that
$A^{(i)}_{P_{i}}=\overline{A _{P_{i}}}$. Then
\[
\sum_{i=1}^{s}A^{(i)}=\overline{A}.
\]

\end{proposition}

\begin{proof}
A more general result is proved in \citet[Proposition 3.2]{BDLPSS}.
\end{proof}

\begin{definition}
We call any ring $A^{(i)}$ as in the proposition a \emph{local contribution}
to $\overline{A}$ at $P_{i}$. If in addition $A^{(i)}_{P_{j}}={A _{P_{j}}}$
for $j\neq i$, we speak of a \emph{minimal local contribution} to
$\overline{A}$ at $P_{i}$.
\end{definition}

\begin{remark}
Note that  a minimal local contribution is uniquely determined since, by definition, its
localization at each $P\in\operatorname*{Spec}(A)$ is determined.
\end{remark}

Given a reduced affine algebra $A$ over a perfect field $L$ with a finite
singular locus, Proposition \ref{prop:local-to-global} allows us to split the
computation of $\overline{A}$ into local tasks at the primes $P_{i}%
\in\Sing(A)$. One way of finding the minimal local contributions $A^{(i)}$ is
to apply the local version of the normalization algorithm from \citet{BDLPSS},
which relies on a local variant of the Grauert and Remmert criterion. For each
$i$, the basic idea is to use $P_{i}$ together with a suitable element $g_{i}$
of the Jacobian ideal instead of a test pair as in Definition
\ref{def:test-ideal}.

\begin{example}
\label{example:localContribution} As in Example
\ref{example:the global algorithm}, let $A$ be the coordinate ring of the
curve $C$ with defining polynomial $f(X,Y)=X^{5} -Y^{2}(Y-1)^{3}\in
{\mathbb{Q}}[X,Y]$. Note that $C$ has a double point of type $A_{4}$ at
$(0,0)$ and a triple point of type $E_{8}$ at $(0,1)$. If we apply the
strategy above, taking $P_{1}=\left\langle x,y\right\rangle _{A}$,
$P_{2}=\left\langle y - 1, x\right\rangle _{A}$, and $g_{1} = g_{2} = x$, we
get local contributions $\frac{1}{d_{i}} U_{i}$, $i =1, 2$. Specifically,
\[%
\begin{tabular}
[c]{lll}%
$d_{1}=x^{2}$ & and & $U_{1}=\left\langle x^{2}, y(y-1)^{3}\right\rangle _{A}%
$,\\
$d_{2}=x^{3}$ & and & $U_{2}=\left\langle x^{3}, x^{2}y^{2}\left(  y-1\right)
, y^{2}\left(  y-1\right)  ^{2}\right\rangle _{A}$.
\end{tabular}
\]
Summing up the local contributions, we get $\overline{A} = \frac{1}{d}U$ with
$d=x^{3}$ and
\[
U=\left\langle x^{3},\text{ }y(y-1)^{3}x,\text{ }y^{2}\left(  y-1\right)
x^{2} ,\text{ }y^{2}\left(  y-1\right)  ^{2}\right\rangle _{A}\text{.}%
\]
Note that $U$ coincides with the ideal $U_{3}$ computed in Example
\ref{example:the global algorithm}.
\end{example}

\begin{remark}
In Example \ref{example:localContribution}, the normalization of the local
ring $A_{P_{2}}$ is $$\overline{A_{P_{2}}}= \frac{1}{x^{3}}\langle x^{3},
x^{2}(y-1), (y-1)^{2}\rangle_{A_{P_{2}}}.$$ Indeed, since $y^{2}$ is a unit in
$\overline{A_{P_{2}}}$, this follows by localizing $U_{2}$ at $P_{2}$. Note,
however, that $(y-1)/x$ and $(y-1)^{2}/x^{3}$ are not integral over $A$.
Hence, $\frac{1}{x^{3}}\langle x^{3}, x^{2}(y-1), (y-1)^{2}\rangle_{A}$ is not
a local contribution to $A$ at $P_{2}$.

\end{remark}

Relying on the Jacobian criterion, we may find the primes in $\Sing(A)$ by
means of primary decomposition. If there is precisely one such prime, this
requires (possibly expensive) computations which are only needed to detect
this fact. In the case of a plane curve $C$ considered here, supposing that
one singularity $P$ of $C$ is already known to us, we may check whether $P$ is
the only singularity of $C$ by comparing the local Tjurina number of $C$ at
$P$ with the total Tjurina number of $C$. Computing the total Tjurina number
via Gr\"obner bases over the rational numbers, however, can be expensive due to
coefficient swell. To overcome this problem, we provide an efficient modular
criterion. Note that though singularities at infinity do not matter for
obtaining integral bases, the criterion takes these singularities into
account. That is, it is formulated in the projective setting.

Let $L$ be any field, let $F \in L[X,Y,Z]$ be a square-free homogeneous
polynomial of positive degree, and let $\Gamma= \Proj(L[X,Y,Z]/\langle F
\rangle)$ be the projective curve defined by $F$. Moreover, write
\[
S=L[X,Y,Z]/\!\left\langle
F_{X}, F_{Y},F_{Z}\right\rangle ,
\]
where $F_{X}, F_{Y}, F_{Z}$ are the partial derivatives of $F$. Then, taking
Euler's rule into account, $\Gamma_{\sing}=\Proj(S)\subset\Gamma$ is the
singular locus of $\Gamma$. For any $Q\in\Gamma_{\sing}$, let $S_{(Q)}$ be the
homogeneous localization of $S$ at $Q$. Then
\[
\tau_{Q}(\Gamma)=\dim_{L} S_{(Q)}%
\]
is the Tjurina number of $\Gamma$ at $Q$. For example, if $P=\langle X,Y
\rangle$, then
\[
\tau_{P}(\Gamma)=\dim_{L}\left(L[X,Y]_{\langle X,Y \rangle}/\!\left\langle
f,\text{ } f_{X}, f_{Y}\right\rangle \right)  ,
\]
with $f=F(X,Y,1)$. The Tjurina number of $\Gamma$ in the chart $X\neq0$ is
\[
\tau_{X\neq0}(\Gamma)=\dim_{L}\left(L[X,Y]/\left\langle f,\text{ }
f_{X},f_{Y}\right\rangle \right)  ,
\]
and similarly for the other coordinate charts. Finally,
\[
\tau(\Gamma)=\deg\operatorname*{Proj}(S) =\sum_{Q\in\Gamma_{\sing}}\tau_{Q}
\]
is the total Tjurina number of $\Gamma$.

\begin{proposition}
\label{prop:one-singular-point}
Let $F \in{\mathbb{Q}}[X,Y,Z]$ be a square-free homogeneous polynomial of
positive degree with integer coefficients. Let $q$ be a prime number such that
the reduction $F_{q}$ of $F$ modulo $q$ is non-zero. Consider the curves
$\Gamma= \Proj({\mathbb{Q}}[X,Y,Z]/ \langle F \rangle)$ and $\Gamma_{q}=
\Proj({\mathbb{F}}_{q}[X,Y,Z]/\langle F_{q} \rangle)$, and let $P=\langle X,Y
\rangle$. Suppose that
\[%
\begin{tabular}
[c]{lllll}%
$\tau_{P}(\Gamma_{q})=\tau_{P}(\Gamma)>0$ & and & $\tau_{X\neq0}(\Gamma
_{q})=\tau_{Y\neq0}(\Gamma_{q})=0$\text{.} &  &
\end{tabular}
\]
Then $\Gamma_{\sing} =\{P\}$.
\end{proposition}

\begin{proof}
By \citet[Theorem 5.3]{arnold}, considering the Hilbert functions of
\begin{align*}
S  &  =\mathbb{Q}[X,Y,Z]/\left\langle F_{X}, F_{Y},F_{Z} \right\rangle
\;\text{and}\\
S_{q}  &  =\mathbb{F}_{q}[X,Y,Z]/\left\langle (F_{X})_{q}, (F_{Y})_{q},
(F_{Z})_{q} \right\rangle ,
\end{align*}
we have
\[
\operatorname*{HF}\nolimits_{S}(t)\leq\operatorname*{HF}\nolimits_{S_{q}}(t),
\text{ for all } t.
\]
Since the Tjurina numbers are the leading coefficients of the respective
Hilbert polynomials, this implies that
\[
\tau(\Gamma)\leq\tau(\Gamma_{q})\text{.}%
\]

On the other hand, if $\tau_{X\neq0}(\Gamma_{q})=\tau_{Y\neq0}(\Gamma_{q})=0$,
then $(\Gamma_{q})_{\sing} =\{P\}$, so that%
\[
\tau(\Gamma_{q})=\tau_{P}(\Gamma_{q})=\tau_{P}(\Gamma)\leq\tau(\Gamma)\text{.}%
\]
Combining both inequalities yields%
\[
\tau_{P}(\Gamma)=\tau(\Gamma)\text{}%
\]
and, thus, $\Gamma_{\sing} =\{P\}$.
\end{proof}

\begin{remark}
The invariants in the criterion can be obtained efficiently by a standard
basis computation over $\mathbb{Q}$ with respect to a local monomial ordering and by
standard basis computations over $\mathbb{F}_{p}$ with respect to a global and
a local monomial ordering, respectively.

\end{remark}

\section{Puiseux Series, Valuations,  and Integrality}
\label{sect:basic-puiseux}

In this section, we discuss some basic facts about Puiseux series and their connection to integrality.
As in the introduction, $K$ will denote a field of characteristic zero. Moreover, $K\subset L$ will
be a field extension, with $L$ algebraically closed.

\subsection{Puiseux Series}

The \emph{field of Puiseux series} over $L$ is the field
\[
L\{\{X\}\}=\bigcup_{k=1}^{\infty}L((X^{1/k})).
\]
The Newton-Puiseux theorem, which is closely related to the aforementioned
finiteness theorem of Emmy Noether, says that $L\{\{X\}\}$ is the algebraic
closure of $L((X))$. In particular, $L[[X^{1/k}]]$ is the integral closure of
$L[[X]]$ in $L((X^{1/k}))$. See \citet[Chapter 13]{Eis}, \citet[Lecture
12]{Abhyankar}.

We have a canonical {\emph{valuation map}}
\[
\upsilon:L\{\{X\}\}\setminus\{0\}\rightarrow{\mathbb{Q}},\;\gamma\mapsto \upsilon(\gamma),
\]
where $\upsilon(\gamma)= \operatorname{ord}_{X}(\gamma)$ is the smallest exponent appearing
in a term of $\gamma$. By convention, $\upsilon(0)=\infty$. The corresponding \emph{valuation ring}
$L\{\{X\}\}_{\upsilon\geq0}=\bigcup_{k=1}^{\infty}L[[X^{1/k}]]$ consists of all Puiseux series
with non-negative exponents only. Henceforth it will be denoted by ${\mathcal{P}_{X}}$.

If $q\in L\{\{X\}\}[Y]$ is any polynomial in $Y$ with coefficients in
$L\{\{X\}\}$, the {\emph{valuation}} of $q$ at $\gamma\in L\{\{X\}\}$ is
defined to be $\upsilon_{\gamma}(q)=\upsilon(q(X,\gamma))$.

\subsection{Conjugate Puiseux Series}

Two Puiseux series in $L\{\{X\}\}$ are called \emph{conjugate} if they are
conjugate as field elements over $K((X))$.

\subsection{Rational Part}

Let $\gamma= a_{1}X^{t_{1}}+\ldots+a_{k}X^{t_{k}}+
a_{k+1}X^{t_{k+1}}+ \ldots\in{\mathcal{P}_{X}}$, with $0 \le t_{1}< \ldots < t_{k}<
\ldots$. Let $k \geq0$ be such that $a_{i} X^{t_{i}}\in K[X]$ for $1 \leq i
\leq k$ and $a_{k+1} X^{t_{k+1}} \not \in K[X]$. Then we call $a_{1}X^{t_{1}%
}+\ldots+a_{k}X^{t_{k}}$ the \emph{rational part} of $\gamma$, and
$a_{k+1}X^{t_{k+1}}$ its \emph{first non-rational term}.

\subsection{Notation}
\label{subsect: general-notation}

In what follows, $g\in K[[X]][Y]$ will be a square-free monic polynomial of degree $m\geq 1$
in $Y$. In subsequent sections, we will consider the defining polynomial $f\in K[X,Y]$ of
an affine plane curve as in the introduction and its factorization
$$
f=f_{0}\widetilde{f}=f_{0} f_{1}\cdots f_{r}
$$
given by the Weierstrass preparation theorem as in  Equation \eqref{equation:branches} of the introduction.
The polynomial $g$ will then be either one of the branches $f_1, \dots , f_r$ or their product
$\widetilde{f}$ or $f$ itself.

\subsection{Puiseux Expansions}
\label{subsect: Puiseux-expansions}

By the Newton-Puiseux theorem, the polynomial $g \in K[[X]][Y]$ has $m$ roots $\gamma_{1}, \dots,
\gamma_{m} \in L\{\{X\}\}$:
\[
g=(Y-\gamma_{1})\cdots(Y-\gamma_{m}).
\]
The  $\gamma_{i}$ are called the \emph{Puiseux expansions} of $g$. The monic
assumption guarantees that these expansions are integral over $K[[X]]$. In particular,
they are all contained in some $L[[X^{1/k}]]\subset {\mathcal{P}_{X}}$, so that  their terms
have non-negative exponents only. Furthermore, each $\gamma_{i}$ is
algebraic over $K((X))$, and its minimal polynomial over $K((X))$ has all
its coefficients in $K[[X]]$ (see, for example, \citet[Theorem 2.1.17]{SH}). We
conclude that the $\gamma_{i}$ can be grouped into conjugacy classes which
correspond to the irreducible factors of $g$ in $K[[X]][Y]$. If $g$ is
\emph{absolutely irreducible}, that is, $g$ is irreducible in $L[[X]][Y]$,
then $m$ is the least positive integer such that all  $\gamma_i$ are contained in
$L[[X^{1/m}]]$. 
In fact, in this case, we have a representation of type
$$
g(T^m,Y)=\prod (Y-\zeta (\omega^\ell T)),
$$
where $\zeta\in L[[T]]$, and $\omega\in L$ is a primitive $m$th root of unity.
That is, the Puiseux expansions of $g$ are of type $\gamma_\ell(X)=\zeta(\omega^\ell X^{1/m})$.
See \citet[Lecture 12]{Abhyankar}, \citet[Section 5.1]{JP}.

\subsection{Regularity Index and Singular Part}

\noindent If $\gamma= a_{1}X^{t_{1}}+a_{2}X^{t_{2}}+\dots$ is a Puiseux
expansion of $g$, with $0 \le t_{1}<t_{2}<\dots$ and no $a_{i}$ zero, we
define the \emph{regularity index} of $\gamma$ (with respect to $g$) to be the
least exponent $t_{k}$ such that no other Puiseux expansion of $g$ has the
same initial part $a_{1}X^{t_{1}}+\dots+a_{k}X^{t_{k}}$. This initial part is
called the \emph{singular part} of $\gamma$ (with respect to $g$).

\subsection{The Newton-Puiseux Algorithm}

The Puiseux expansions of $g$ can be computed recursively up to any given
$X$-degree using the Newton-Puiseux algorithm (see, for example, \citet{JP}).
Essentially, to get a solution $a_{1}X^{t_{1}}+a_{2}X^{t_{2}}+\dots$ of
$g(X,\gamma(X))=0$, with $t_{1}<t_{2}<\dots$, the algorithm proceeds as
follows: Starting from $g^{(0)}=g$ and $K^{(0)}=K((X))$, we commence the $i$th
step of the algorithm by looking at a polynomial $g^{(i-1)}\in K^{(i-1)}[Y]$.
We then choose one face $\Delta$ of the Newton polygon of $g^{(i-1)}$ such
that all the other points of the polygon lie on or above the line containing
the face. Let $g_{\Delta}^{(i-1)}$ be the sum of terms of $g^{(i-1)}$
involving the monomials of $g^{(i-1)}$ on $\Delta$. That is, if $-\frac{w_{1}%
}{w_{2}}$ is the slope of $\Delta$, then $g_{\Delta}^{(i-1)}$ is the sum of
terms of $g^{(i-1)}$ of lowest $(1,\frac{w_{2}}{w_{1}})$-weighted degree. We
write $d_{i}$ for this degree. Choose an irreducible factor of $g_{\Delta
}^{(i-1)}$ over $K^{(i-1)}$ and a root $q_{i}$ of that factor. Note that
$q_{i}$ is of type $q_{i}=c_{i}X^{\frac{w_{2}}{w_{1}}}$, where $c_{i}$ is a
root of the polynomial $g_{\Delta}^{(i-1)}(1,Y)$. Now, let $K^{(i)}%
=K^{(i-1)}(q_{i})$ and set $g^{(i)}=\frac{1}{X^{d_{i}}}g^{(i-1)}(X,q_{i}%
\cdot(1+Y))$. Then the $i$th term of the expansion to be constructed is
$a_{i}X^{t_{i}}=q_{1}\cdots q_{i}$. It is clear from this construction that
different conjugacy classes of expansions arise from different choices for the
faces and irreducible factors of $g_{\Delta}^{(i-1)}$ over $K^{(i-1)}$, respectively.

\begin{example}
\label{examplePuiseux} The eight Puiseux expansions of the polynomial
\begin{align*}
g  &  =Y^{8}+(-4X^{3}+4X^{5})Y^{7}+(4X^{3}-4X^{5}-10X^{6})Y^{6}\\
& +(4X^{5} -6X^{6})Y^{5} +(6X^{6}-8X^{8})Y^{4}+(8X^{8}-4X^{9})Y^{3}\\
& +(4X^{9}+4X^{10})Y^{2}+4X^{11}Y+X^{12}\in{\mathbb{Q}}[X,Y]
\end{align*}
are conjugate over ${\mathbb{Q}}((X))$; their singular parts are of type
\[
q_{1}+q_{1}q_{2}+q_{1}q_{2}q_{3},
\]
where the $q_{i}$ satisfy
\[%
q_{1}^{2}+X^{3}=0, \ \ \ q_{2}^{2}+\frac{1}{2X}q_{1}=0, \  \ \ \text{and } \ \ \ q_{3}^{2}-\frac{1}{8X}q_{1}=0.
\]
To see this, note that the Newton polygon of $g^{\left(  0\right)  }=g$ has
only one face $\Delta_{0}$, leading to $g_{\Delta_{0}}^{\left(  0\right)
}=\left(  X^{3}+Y^{2}\right)  ^{4}$ and the extension%
\[
K_{0}=\mathbb{Q}((X))\subset K_{1}=K_{0}[iX^{\frac{3}{2}}].
\]
In the next step, $g^{\left(  1\right)  }$ has only one face $\Delta_{1}$,
yielding
\[
g_{\Delta_{1}}^{\left(  1\right)  }=4\left(  2Y^{2}+\frac{q_{1}}{X}\right)
^{2}%
\]
and%
\[
K_{1}\subset K_{2}=K_{0}[iX^{\frac{3}{2}},(1-i)X^{\frac{1}{4}}].
\]
Finally, also $g^{\left(  2\right)  }$ has only one face $\Delta_{2}$, which
corresponds to
\[
g_{\Delta_{2}}^{\left(  2\right)  }=-2\cdot\left(  8Y^{2}-\frac{q_{1}}%
{X}\right)
\]
and the extension%
\[
K_{2}\subset K_{3}=K_{0}[iX^{\frac{3}{2}},(1-i)X^{\frac{1}{4}},(1+i)X^{\frac
{1}{4}}]=K_{0}[i,X^{\frac{1}{4}}].
\]

\end{example}

\subsection{Maximal Integrality Exponents}
\label{sect:max-expo}

Let $\Gamma=\{\gamma_{1},\dots,\gamma_{m}\}$ be the set of Puiseux
expansions of $g$. The {\emph{valuation}} of a polynomial $q\in L\{\{X\}\}[Y]$
at $g$ is defined to be $\upsilon_{g}(q)=\min_{1 \leq i \leq m} \upsilon_{\gamma_{i}}(q)$.
Note that if $q$ is monic of degree $d\geq 1$ in $Y$, 
and
\[
q=(Y-\eta_{1}(X))\cdots(Y-\eta_{d}(X))
\]
is the factorization of $q$ in $L\{\{X\}\}[Y]$, then
\[
\upsilon_{g}(q)=\min_{1 \leq i \leq m}\sum_{j=1}^{d}\upsilon(\gamma_{i}-\eta_{j})\text{.}%
\]

\begin{lemma}
\label{lemma:intA} Let $g\in K[[X]][Y]$ be a square-free monic polynomial of degree
$m\geq 1$ in $Y$, with Puiseux expansions $\gamma_{1}, \dots, \gamma_{m}$. Fix
an integer $d$ with $1\leq d \leq m-1$. If $\mathcal{A}\subset\{1,\dots,m\}$ is
a subset of cardinality $d$, set
\[
\Int({\mathcal{A}})=\min_{i\not \in \mathcal{A}}\left(  \sum_{j\in\mathcal{A}%
}\upsilon(\gamma_{i}-\gamma_{j})\right).
\]
Choose a subset $\widetilde{\mathcal{A}}\subset\{1,\dots,m\}$ of cardinality
$d$ such that $\Int({\widetilde{\mathcal{A}}})$ is maximal among all
$\Int({\mathcal{A}})$ as above, and set $\widetilde{p}_d=\prod_{j\in
\widetilde{\mathcal{A}}}(Y-\gamma_{j})\in{\mathcal{P}_{X}}[Y]$. Then
$\upsilon_{g}(\widetilde{p}_d)=\Int({\widetilde{\mathcal{A}}})$, and this number is the
maximal valuation $\upsilon_{g}(q)$, for
$q\in L\{\{X\}\}[Y]$ monic of degree $d$ in $Y$.

\end{lemma}

\begin{proof}
That $\upsilon_{g}(\widetilde{p}_d)=\Int({\widetilde{\mathcal{A}}})$ is clear from the
definitions. That this number is the maximum valuation $\upsilon_{g}(q)$ as claimed
follows as in the proof of \citet[Theorem 5.1]{vanHoeij94}, where the case
$d=m-1$ is treated.
\end{proof}

In the situation of the lemma, we write
\[
o(\Gamma,d)=\upsilon_{g}(\widetilde{p}_d).
\]

In case $d=m-1$, we abbreviate
\[
\Int_{i}=\Int(\{1,\dots,i-1,i+1,\dots,n\})=\sum_{j\neq i}\upsilon(\gamma_{i}%
-\gamma_{j}).
\]

\begin{remark}\label{rem:max-int-coeff}
By construction, we have $$o(\Gamma,1)\leq\ldots\leq o(\Gamma,m-1).$$
\end{remark}

\begin{example}
\label{exampleOneBranch} Let $g=(Y^{2}+2X^{3})+Y^{3}\in{\mathbb{Q}}[X,Y]$. The
Puiseux expansions of $g$ are
\begin{align*}
\gamma_{1} &  =a_{1}X^{3/2}+X^{3}+\dots\;,\\
\gamma_{2} &  =a_{2}X^{3/2}+X^{3}+\dots\;,\\
\gamma_{3} &  =-1-2X^{3}+\dots\;,
\end{align*}
where $a_{1},a_{2}$ are the roots of $Z^{2}+2$. Then $\Int_{1}=3/2+0=3/2$,
$\Int_{2}=3/2+0=3/2$, and $\Int_{3}=0+0=0$, so that both $i=1$ and $i=2$
maximize the valuation. Taking $i=1$, we get $\widetilde{p}_2=(Y-\gamma
_{2})(Y-\gamma_{3})$ and $o(\Gamma,2)=3/2$.
\end{example}

\begin{example}
\label{exampleTwoBranches-int} Let $g=(Y^{3}+X^{2})(Y^{2}-X^{3})+Y^{6}%
\in{\mathbb{Q}}[X,Y]$. The Puiseux expansions of $g$ are

\begin{equation*}
\begin{aligned}[c]
\gamma_1 &= a_1 X^{2/3} +\dots, \\
\gamma_2 &= a_2 X^{2/3} +  \dots, \\
\gamma_3 &= a_3 X^{2/3} + \dots,
\end{aligned}
\quad\quad
\begin{aligned}
\gamma_4 &=  X^{3/2} -1/2x^{11/2} + \dots, \\
\gamma_5 &= -X^{3/2} + \dots,  \\
\gamma_6 &= 1 + \dots,
\end{aligned}
\end{equation*}

\noindent where the $a_{i}$ are the roots of $Z^{3}+1$. Then
$\Int_{1}=\Int_{2}=\Int_{3}=2/3+2/3+2/3+2/3+0=8/3$, $\Int_{4}=\Int_{5}%
=3/2+2/3+2/3+2/3+0=7/2$, and $\Int_{6}=0$. We conclude that
$o(\Gamma,5)=7/2$.
\end{example}

\begin{lemma}
\label{same exp} Let $g\in K[[X]][Y]$ be a square-free monic polynomial of degree
$m\geq 1$ in $Y$, let $1\leq d \leq m-1$, and let $R$ be one of the rings $K[X]$,
$K[X]_{\left\langle X\right\rangle }$, $K[[X]]$, $K((X))$, ${\mathcal{P}_{X}}$, or $L\{\{X\}\}$.
The maximal valuation $v_{g}(q)$, $q\in R[Y]$ monic of degree $d$ in $Y$, is
independent of the choice of $R$ from among this list.
\end{lemma}
\begin{proof}
For any ring $R$ as in the assertion, we have natural inclusions $K[X]\subset R \subset L\{\{X\}\}$.
Hence, the value $\upsilon_{g}(q)$  is defined for any polynomial $q\in R[Y]$ and it suffices to show that
there is a polynomial $p_d\in K[X][Y]$ such that $\upsilon_{g}(p_d)$ maximizes the valuation over $L\{\{X\}\}$
in degree $d$. For this, we recall from Lemma \ref{lemma:intA} that there is a polynomial
$\widetilde{p}_d=\prod_{j\in\widetilde{\mathcal{A}}}(Y-\gamma_{j}) \in \mathcal{P}_{X}[Y]$
which maximizes the valuation over $L\{\{X\}\}$ in degree $d$. We may choose an
integer $k$ such that $\widetilde{p}_d\in L[[X^{1/k}]][Y]$. By truncating each
$\gamma_{j}$ to degree $\upsilon_{g}(\widetilde{p}_d)$, we get a polynomial $\overline
{p}_d = \prod_{j\in\widetilde{\mathcal{A}}}(Y-\overline{\gamma}_{j}) \in L[X^{1/k}][Y]$ with $\upsilon_{g}(\overline{p}_d)
= \upsilon_{g}(\widetilde{p}_d)$. Since $\overline{p}_d$ is monic in $Y$, by applying the trace map for
$L(X^{1/k})$ over $L(X)$ to $\overline{p}_d$ and dividing by the integer leading coefficient of the resulting
polynomial, we get a monic polynomial $p_d^\prime\in L[X][Y]$ of degree $d$ in $Y$ with $\upsilon_{g}({p_{d}^\prime})
\geq  \upsilon_{g}(\widetilde{p}_d)$ (note that the trace map sends $X^{1/k}$ to zero).
Next, considering $p_d^\prime$ as a polynomial in $X, Y$ with coefficients in $L$ and
adjoining these coefficients to $K$, we get a finite field extension $K \subset K^\prime$ such that
$p_d^\prime \in K^\prime[X][Y]$. Applying the trace map of this extension to $p_d^\prime$ and
dividing by the integer leading coefficient of the resulting polynomial, we get a monic polynomial
$p_d\in K[X][Y]$ of degree $d$ in $Y$ with $\upsilon_{g}({p_d}) \geq  \upsilon_{g}(\widetilde{p}_d)$.
In fact, by Lemma \ref{lemma:intA} and the choice of $\widetilde{p}_d$, equality holds since
$\widetilde{p}_d$ maximizes the valuation over $L\{\{X\}\}$.
\end{proof}

\begin{example}
\label{exampleTwoBranches-int-2} In Example \ref{exampleTwoBranches-int}, choosing
$\widetilde{p}_2=(Y-\gamma_{2})(Y-\gamma_{3})$, we get $\overline {p}_2 = (Y-a_{2}X^{3/2})(Y+1)$
 and, thus,  $p_2=p_2^\prime=Y(Y+1)$.
\end{example}

The reason for considering the valuations $v_{g}(q)$ is that they are directly
related to integrality: If $q\in  K[[X]][Y]$ is monic of degree $0\leq d\leq m-1$ in $Y$,
then $\lfloor v_g(q) \rfloor$ is the maximum integer $e$ such that $q(x,y) / x^e$ is integral
over
\[
B = K[[x]][y] =K[[X]][Y]/\!\left\langle g\right\rangle.
\]
For the lack of reference in this generality,  we show this in Theorem \ref{theorem:integrality local}
below (see also \citet[Chapter 5, \S 1]{JP} and  \citet[Chapter 5, \S 10]{walker}). We refer to \citet[Chapter 6]{SH}
for a general introduction to valuations and their connection to integrality.

We will use the following result to reduce problems
concerning a possibly reducible polynomial $g\in K[[X]][Y]$ to the irreducible case:

\begin{proposition}
[Splitting of Normalization]\label{prop split} Let $g\in K[[X]][Y]$ be a square-free
monic polynomial  of degree $\geq 1$ in $Y$, and let $g = g_{1}\cdots g_{s}$ be its
decomposition into irreducible factors $g_i\in K[[X]][Y]$.  For each $i$, $1 \le i \le s$,
set $h_{i}=\prod_{j=1\text{, }j\neq i}^{s}g_{j}$. Then the $g_{i}$ and $h_{i}$
are coprime in $K((X))[Y]$, so that there are polynomials $a_{i},b_{i}\in K[[X]][Y]$
and integers $c_{i}\in\mathbb{N}$ fitting into B\'ezout identities of type
\[
a_{i}g_{i}+b_{i}h_{i}=X^{c_{i}}, \;\text{ for }\; i=1, \dots, s\text{.}%
\]
Furthermore, the normalization of \ \negthinspace\negthinspace
\ $K[[X]][Y]/\!\left\langle g_{1}\cdots g_{s}\right\rangle$ splits as%
\begin{equation}
\label{equation:splitting}
\overline{K[[X]][Y]/\!\left\langle g_{1}\cdots g_{s}\right\rangle }%
\cong\bigoplus_{i=1}^{s}\overline{K[[X]][Y]/\!\left\langle g_{i}\right\rangle
},
\end{equation}
where the splitting is given by%
\[
(t_{1}\!\!\!\mod g_{1},\dots,t_{s}\!\!\!\mod g_{s})\mapsto\sum_{i=1}^{r}%
\frac{b_{i}h_{i}t_{i}}{X^{c_{i}}}\!\!\!\mod g_{1}\cdots g_{s}\text{.}%
\]
\end{proposition}

\begin{proof}
Clear by the Chinese remainder theorem and its proof. See \citet[Theorem
1.5.20]{JP}.
\end{proof}

We are now ready to clarify the relation between  valuations and integrality.
To prove the theorem which is relevant to us here, we need the result below:
\begin{proposition}
\label{rem:val-int-gen}
Let $R$ be an integral domain with quotient field $\TQR(R)$. Then
$\phi \in \TQR(R)$ is integral over $R$ iff $\upsilon(\phi) \ge 0$ for every valuation
$\upsilon$ on $\TQR(R)$ whose valuation ring contains $R$.
If $R$ is Noetherian, it is enough to consider discrete valuations.
\end{proposition}
\begin{proof}
 See, for example, \citet[Proposition 6.8.14]{SH}.
\end{proof}

\begin{theorem}
\label{theorem:integrality local}
Let $g \in K[[X]][Y]$ be a square-free monic polynomial of degree
$m\geq 1$ in $Y$. Let $q\in  K[[X]][Y]$ be monic of degree
$0\leq d\leq m-1$ in $Y$, and let $e$ be the maximal integer such that $\frac{q(x,y)}{x^{e}}$
is integral over $K[[x]][y]=K[[X]][Y]/\langle g\rangle$. Then
\[
e=\lfloor \upsilon_{g}(q)\rfloor\text{.}%
\]
\end{theorem}

\begin{proof}
Let  $\Gamma$ be the set of Puiseux expansions of $g$. Then
$\upsilon_{g}(q) = \min_{{\gamma}\in \Gamma}\upsilon_{\gamma}(q)$
by the very definition of $\upsilon_{g}(q)$ in Section \ref{sect:max-expo}. Hence, the assertion
of the theorem is equivalent to the following statement:
\begin{equation}
\begin{aligned}
\label{equation:equiv-statem}
\text{If $\widetilde{e}$ is a non-negative integer, then ${q(x,y)}/{x^{\widetilde{e}}}$ is integral over $K[[x]][y]$} \\
\text{$\Longleftrightarrow \upsilon_{\gamma}(q) \ge \widetilde{e}$ for each $\gamma\in\Gamma$.}
\end{aligned}
\end{equation}
\noindent
To show this statement, we proceed in four steps:

\emph{Step 1}: We may assume that $K$ is algebraically closed. To see this, we consider a field extension $K\subset L$
with $L$ algebraically closed as before, and show that an element $w \in K((x))[y] = \TQR(K[[X]][y])
\subset \TQR(L[[X]][y]) = L((x))[y]$ is integral over $K[[x]][y]$ iff it is integral over $L[[x]][y]$.  One direction is immediate:
any integral equation for $w$ over $K[[x]][y]$ is also an integral equation for $w$ over $L[[x]][y]$. Conversely, if
$p(w) = 0$ is an integral equation for $w$ over $L[[x]][y]$, then adjoining the coefficients of $p$ to $K$ yields
a finite field extension of  $K$.  Applying the trace map of this extension to  $p$ and dividing by the integer
leading coefficient of the resulting polynomial, we get an integral equation for $w$ over $K[[x]][y]$.

\emph{Step 2}: We may assume that $g$ is irreducible. Indeed,  if $g=g_1 \cdots g_s$ is the  factorization of $g$
as in Proposition \ref{prop split}, each $g_{i}$, $1 \le i \le s$,  corresponds to a conjugacy class of the Puiseux expansions
of $g$ (see Section \ref{subsect: Puiseux-expansions}). Hence, by Proposition \ref{prop split}, statement
\eqref{equation:equiv-statem} holds for $g$ iff it holds for each $g_{i}$.

\emph{Step 3}: We may assume that $g$ is a Weierstrass polynomial. Indeed, since we already assume that
$g$ is irreducible, $g$ is either a Weierstrass polynomial or a unit in $K[[x, y]]$. In the latter case,
$g(0,0)\neq 0$, and it follows from  Hensel's lemma that $g(0,Y)=(Y-c)^m$ for some $c\in K$
(see Section \ref{sec Hensel}  below for Hensel's lemma). Hence, the translation $Y\rightarrow Y+c$ turns $g$
into a Weierstrass polynomial.

\emph{Step 4}:
From the Newton-Puiseux theorem, we know that under the assumptions above,
the  Puiseux expansions of $g$ are of type $\eta (\omega^\ell X^{1/m})$,
$l=1,\dots,m$, where $\eta\in T K[[T]]$, and $\omega$ is a primitive  $m$th
root of unity (see Section \ref{subsect: Puiseux-expansions}). In particular,
the value $\upsilon_{\gamma}(q)$ on  the right hand hand side of \eqref{equation:equiv-statem}
is independent of the choice of $\gamma\in\Gamma$. To show
that \eqref{equation:equiv-statem} holds, we may, hence, work with the
fixed expansion $\gamma(X)=\eta (X^{1/m})$.

For the implication from left to right in \eqref{equation:equiv-statem},  we then note that
$$
\upsilon_{\gamma} : K((x))[y] \rightarrow \Z \cup \{\infty\}, \ \phi \mapsto m\upsilon_{\gamma}(\phi),
$$
is a well-defined valuation on $K((x))[y]$ whose valuation ring contains $K[[x]][y]$.
Thus, Remark~\ref{rem:val-int-gen} implies that if $\upsilon_\gamma(q) < \widetilde{e}$,
then $q/x^{\widetilde{e}}$ is not integral.

For the converse implication,  we conclude from \citet[Theorem 5.1.3]{JP} that the map
$K[[x,y]] \rightarrow K[[T]]$ defined by $(x, y) \mapsto (T^m, \eta (T))$ allows us to regard $K[[T]]$ as the normalization of $K[[x,y]]$.
Hence, if $q(x,y) / x^{\widetilde{e}}$ is not integral, then the fraction $q(T^m, \eta (T)) / T^{m \widetilde{e}} \not\in K[[T]]$,
that is, the fraction has negative order in T. But
\[
\operatorname{ord}_{\;\! T}(q(T^m, \eta (T)) / T^{m \widetilde{e}}) = m \cdot \operatorname{ord}_X(q(X, \gamma(X)) / X^{\widetilde{e}}),
\]
 which shows that $\upsilon_\gamma(q) < \widetilde{e}$. This concludes the proof.
\end{proof}

\begin{definition}
\label{defn integrality exponent2}
Let $g \in K[[X]][Y]$ be as above.
If $q\in K[[X]][Y]$ is monic of degree $0\leq d\leq m-1$ in $Y$, then we call
$$e_g(q):=\lfloor v_{g}(q)\rfloor$$
the \emph{integrality exponent of $q$ with respect to }$g$.
Furthermore, we call
\[
e_{g,d}:=\max\left\{  e_g(q)\mid q\in K[[X]][Y]\text{ monic in $Y$, }\deg q=d\right\}=\lfloor o(\Gamma,d)\rfloor
\]
the \emph{maximal integrality exponent with respect to g in degree $d$}.
\end{definition}

\begin{remark}
By Remark \ref{rem:max-int-coeff}, we have $$0 = e_{g,0} \le e_{g,1}\leq\ldots\leq e_{g,{m-1}}.$$
\end{remark}

\begin{definition}
With notation as above, we call%
\[
E(g)=e_{g,m-1}
\]
the {\emph{maximal integrality exponent with respect to $g$}}.
\end{definition}

For the defining equation of an affine plane curve as in the introduction,
the analogue to  Theorem \ref{theorem:integrality local} is well-known:
\begin{proposition}
\label{prop:integrality-local-II}
Let $f \in K[X][Y]$ be an irreducible monic polynomial of degree
$n\geq 1$ in $Y$. Let $q\in  K[X][Y]$ be monic of degree
$0\leq d\leq n-1$ in $Y$, and let $e$ be the maximal integer such that $\frac{q(x,y)}{x^{e}}$
is integral over $A=K[x][y]=K[X][Y]/\langle f\rangle$. Then
\[
e=\lfloor \upsilon_{f}(q)\rfloor\text{.}%
\]
\end{proposition}
\begin{proof}
We know that an element $\phi\in\TQR(A)=K(x)[y]$ is integral over $A$ iff it is
integral over $K[x]$. By \citet[Theorem 3.2.6]{Stichtenoth08}, the latter is equivalent
to the condition that $\upsilon(\phi) \ge 0$ for every (discrete) valuation
$\upsilon$ on $\TQR(A)$ whose valuation ring contains $K[x]$. Similar to the proof
of Theorem \ref{theorem:integrality local}, this in turn means that  $v_{\gamma}(\phi)\geq 0$
for all Puiseux expansions $\gamma$ of $f$ (see \citet[Section 2.4]{vanHoeij94}).
The result follows.
\end{proof}

\section{Integral Bases and Integrality Exponents}

\label{sect:norm-via-loc-and-compl-the-theory}

As in the previous section, let $g\in K[[X]][Y]$ be a
square-free monic polynomial of degree $m\geq 1$ in $Y$, and write
\[
B=K[[x]][y] =K[[X]][Y]/\!\left\langle g\right\rangle \text{.}
\]
Then $\overline{B}$ coincides with the integral closure of
$K[[x]]$ in $K((x))[y]$. In fact, $\overline{B}$ is a free $K[[x]]$-module of rank $m$:
If $g$ is irreducible, this follows from Remark \ref{rem:free-ness-over-pid};
if $g$ is arbitrary, apply the splitting of normalization  as in Proposition \ref{prop split}
to reduce to the irreducible case.

\begin{definition}
With notation as above, we refer to any integral basis for $\overline{B}$ over
$K[[x]]$ as an {\emph{integral basis for}} $g$. 
\end{definition}

In this section, we study the shape of integral bases in terms of integrality exponents,
and we show how to construct an integral basis for $g$ from given integral bases
for its irreducible factors. With regard to integrality exponents, we use the terminology from
Definition \ref{defn integrality exponent2}. Specifically, $e_{g,d}$ denotes the
maximal integrality exponent with respect to $g$ in degree $d$, for $d=0,\dots, m-1$.

\begin{proposition}
\label{prop loc int bas shape}
With notation as above, for each $d=1, \dots, m-1$, let a monic polynomial
$p_{d}\in K[[X]][Y]$ of degree $d$ in $Y$ and an integer $e_d$ be given. Then
\[
\overline{\mathcal{B}} = \left\{1=p_{0},\frac{p_{1}(x,y)}{x^{e_{1}}},\ldots,\frac{p_{m-1}(x,y)}{x^{e_{m-1}}}\right\}%
\]
is an integral basis for $g$ iff $e_{d}= e_g(p_d)= e_{g,d}$, for $d=1, \dots, m-1$.
\end{proposition}

\begin{proof}
Write ${B}_{d}^{\prime}=\left\langle 1,\frac{p_{1}(x,y)}{x^{e_{1}}},\ldots,\frac{p_{d}(x,y)}{x^{e_{d}}}\right\rangle _{K[[x]]}$
for each $d$, and ${B}^{\prime}=\left\langle \overline{\mathcal{B}} \right\rangle={B}_{m-1}^{\prime}$.

First suppose that $\overline{\mathcal{B}}$ is an integral basis for $g$, and consider a fixed $d$.
Then $e_{d}\leq e_g(p_d) \leq e_{g,d}$.  To show that these numbers are equal, we choose an element
$q\in K[[X]][Y]$ which is monic in $Y$ of degree~$d$ and satisfies $e_g(q)=e_{g,d}$. Then
$\frac{q(x,y)}{x^{e_{g,d}}}$ is integral over $K[[x]]$ by Theorem \ref{theorem:integrality local}, so
$\frac{q(x,y)}{x^{e_{g,d}}}\in {B}_{d}^{\prime}$ since $\overline{\mathcal{B}}$ is an integral basis for $g$.
Writing $\frac{q(x,y)}{x^{e_{g,d}}}$ as a $K[[x]]$-linear combination of the generators
of ${B}_{d}^{\prime}$, and comparing the coefficients of $y^d$, we get
$e_{d}=e_{g,d}$, as desired.

Conversely, suppose that the equalities $e_{d}= e_g(p_d)= e_{g,d}$ hold. Then the
$p_{d}(x,y)/x^{e_d}$ are integral over ${B}$. In particular, ${B}\subset {B}^{\prime
}\subset\overline{{B}}$. Hence, to show that ${B}^{\prime}=\overline{{B}}$ (and,
thus, that $\overline{\mathcal{B}}$ is an integral basis for $g$), it is
enough to prove the following: Given a polynomial $q\in K[[X]][Y]$ of degree
$0\leq d\leq m-1$ in $Y$ such that $\frac{q(x,y)}{x^{e}}\in\overline{{{{B}}}}$
for some $e\in\N$, we must have $\frac{q(x,y)}{x^{e}}\in {B}_{d}^{\prime}$.

We do induction on $d$. There is nothing to show in case $d=0$. If $d\geq1$, let $c \in K[[X]]$ be the leading
coefficient of $q \in K[[X]][Y]$. After factoring out a unit in $K[[X]]$, we can assume that
$c = X^t$, for some $t \in \N$. Write $q$ as a product $q=X^t\widetilde{q}$, with
$\widetilde{q}\in K((X))[Y]$ monic in $Y$. Then, by Lemma \ref{same exp} and  Theorem
\ref{theorem:integrality local}, we have $\lfloor \upsilon_{g}(\widetilde{q})\rfloor \leq e_{g,d} = e_d$, hence
\[
e \leq t + e_g(\widetilde{q}) = t + \lfloor \upsilon_{g}(\widetilde{q})\rfloor \leq t+e_d\text{.}%
\]
This implies $\frac{x^t p_{d}(x,y)}{x^{e}} = \frac{p_{d}(x,y)}{x^{e-t}}\in {B}_{d}^{\prime} \subset \overline{{B}}$.
Since $\deg_{Y}(q-X^t p_{d})<d$ and $\frac{q(x,y)}{x^{e}}-\frac
{x^tp_{d}(x,y)}{x^{e}}\in\overline{{B}}$, the induction hypothesis gives
$
\frac{q(x,y)}{x^{e}}-\frac{x^t p_{d}(x,y)}{x^{e}}\in {B}_{d-1}^{\prime}\subset
{B}_{d}^{\prime}$.
Therefore $\frac{q(x,y)}{x^{e}}\in {B}_{d}^{\prime}$, as claimed.
\end{proof}

\begin{remark}
\label{remark:exist1}
We say that an integral  basis as in  Proposition \ref{prop loc int bas shape} is of \emph{monic triangular type}.
Together with Lemmas \ref{lemma:intA} and \ref{same exp}, the proposition shows the
existence of such bases, where the $p_d$ can even be chosen to be polynomials in $K[X][Y]$.
Henceforth, whenever we speak of an integral basis of monic triangular type, we tacitly assume that
the $p_d$ are chosen that way.
How to actually compute bases of this type in the case where $g=f_i$ is a branch of our
given polynomial $f$ is a topic of Section \ref{section:lowdegree}.
\end{remark}

\begin{proposition}
\label{coro complete split}
Let $g = g_{1}\cdots g_{s}\in K[[X]][Y]$ be the decomposition of a square-free monic polynomial
$g$ of degree $\geq 1$ in $Y$ into its irreducible factors. For $i=1,\ldots,s$, let
\[
\mathcal{B}^{(i)}=\left\{1=p_{0}^{(i)},\frac{p_{1}^{(i)}}{X^{e_{1}^{(i)}}},\ldots,\frac{p_{m_{i}%
-1}^{(i)}}{X^{e_{m_{i}-1}^{(i)}}}\right\}%
\]
represent an integral basis for  $g_{i}$ as in Proposition \ref{prop loc int bas shape}.
With notation as in Proposition \ref{prop split}, for each $i$, set
\[
\widetilde{\mathcal{B}}^{(i)}=\left\{  \frac{b_{i}h_{i}}{X^{c_{i}}},\frac{b_{i}h_{i}p_{1}^{(i)}%
}{X^{c_{i}+e_{1}^{(i)}}},\ldots,\frac{b_{i}h_{i}p_{m_{i}-1}^{(i)}}%
{X^{c_{i}+e_{m_{i}-1}^{(i)}}}\right\}.
\]
Then $\widetilde{\mathcal{B}}^{(1)}\cup\ldots\cup \widetilde{\mathcal{B}}^{(s)}$ represents an integral basis for $g$.
\end{proposition}

\begin{proof}
Immediate from Proposition \ref{prop split}.
\end{proof}

There is an algorithmic way of transferring an integral basis for $g$
as in the proposition to an integral basis for $g$ of monic triangular type:  

\begin{remark}
\label{rem: HN-mts}
Each matrix $M$ with entries in the PID $K[[X]]$ of maximal column rank has a uniquely
determined upper triangular Hermite normal form $H=(p_{ij})$, where the diagonal elements
are of type $p_{ii}=X^{\nu_{i}}$, and the $p_{ij}$, $j>i$, are polynomials in $K[X]$ of degree
$<\nu_{i}$. So the entries of $H$ are polynomials in $X$, while the entries of $M$ are
power series in $X$. To compute $H$ from $M$ via unimodular row operations, we have to
consider suitably truncated power series, that is, we work over
$K[[X]]/\langle X^{t+1}\rangle$, where $t$ is a precision which guarantees a correct result $H$
(see \citet{Durvye84}). Applying this in the situation of Proposition \ref{coro complete split},
starting from a set $\widetilde{\mathcal{B}}^{(1)}\cup\ldots\cup \widetilde{\mathcal{B}}^{(s)}$
as in the proposition and proceeding as in Remark \ref{rem:spec-int-basis}, we get
an integral basis for $g$ of type
\[
\overline{\mathcal{B}} = \left\{{p_{0}(x,y)},
\frac{p_{1}(x,y)}{x^{e_{1}}},\ldots,\frac{p_{m-1}(x,y)}{x^{e_{m-1}}}\right\},%
\]
with polynomials $p_{d}\in K[X][Y]$ of degree $d$ in $Y$.
\end{remark}

\begin{corollary}
\label{cor:mon-triang-basis}
With notation as in Remark \ref{rem: HN-mts} above, let
$\overline{\mathcal{B}}$ be obtained by the recipe given in
that remark. Then $\overline{\mathcal{B}}$ is
an integral basis for $g$ of monic triangular type.
\end{corollary}

\begin{proof}
Fix a degree $d$, and write $\overline{\mathcal{B}}_d$ for the set of elements in $\overline{\mathcal{B}}$
of $y$-degree \ $\leq d$. By construction, the leading coefficient of $p_{d}$ is a
power of $x$, say $x^{\widetilde{e}_d}$. 
Now consider an integral basis $\overline{\mathcal{B}}'$ for $g$ of monic
triangular type (according to Remark \ref{remark:exist1}, such bases exist). Expressing
the $d$th element of $\overline{\mathcal{B}} '$ as a $K[[x]]$-linear combination of the elements
in $\overline{\mathcal{B}}_d$ and comparing the coefficients of $y^d$, we see that $\widetilde{e}_d=0$
and $e_d = e_{g,d}$.
\end{proof}

\section{Normalization of Plane Curves via Localization and Completion:
Local Contributions From Integral Bases for the Branches}

\label{sect:local-contribution}

In this section, $f\in K[X,Y]$ denotes the defining polynomial of an irreducible
plane curve $C\subset\mathbb{A}^{2}(K)$
with assumptions as in the introduction. For simplicity of the presentation,
we focus on the case of a $K$-rational singularity, supposing that this singularity is the origin:
let $P=\langle x,y\rangle\in \Sing(A)$. How to reduce to the case of such
a singularity is a topic of Section  \ref{sect:loc-contr-via-Hensel}.

As in Equation \eqref{equation:branches} of the introduction, factorize $f$ as
\begin{equation*}
f=f_{0}\widetilde{f}=f_{0} f_{1}\cdots f_{r}\text{,}
\end{equation*}
where $f_{0}\in K[[X]][Y]$ is a unit in $K[[X,Y]]$, and $f_{1}, \dots, f_{r}$ are
irreducible Weierstrass polynomials in $K[[X]][Y]$, the branches of $f$
(over $K$, centered at the origin).

\begin{remark}
\label{rem:branches-and-PS}
Recall from Section \ref{subsect: Puiseux-expansions} that the irreducible factors of $f$
in $K[[X]][Y]$ correspond to the conjugacy classes of the Puiseux expansions of $f$. Developed up to a
given degree, the $f_i$, $0 \le i \le r$,  may hence be found by computing all expansions via the
Newton-Puiseux algorithm. There is, however, a more effective approach:
in Section \ref{sect:PuiseuxBlock}, we will present a method  which, based
on Hensel's lemma, makes considerably less use of
the Newton-Puiseux algorithm.
\end{remark}
Proposition \ref{coro complete split} shows us how to obtain an integral basis for $f$
from integral bases for its irreducible factors. It turns out, however, that factors
whose zeros on the line $X=0$ are non-singular points of $C$
can be treated simultaneously, in a more efficient way.
To further simplify our presentation, we assume in addition that $f_0$ is the product of these
factors. Suppose from now on that the origin is the only singularity of $C$ with $X$-coordinate
zero. That is, if $\langle x\rangle\subset Q\in \Sing(A)$,  then $Q=P=\langle x,y\rangle$.
We then also say, that $P$ is the only singularity (of $A$) at $X=0$. How to reduce to this
case is another topic of Section  \ref{sect:loc-contr-via-Hensel}.

Using  Puiseux series, we show in Proposition~\ref{completion basis} that under
the additional assumption above, we can read off an  integral basis for $f$
from such a basis for $\widetilde{f}$.
More precisely, taking Remark~\ref{remark:exist1} into account and starting from a basis for $\widetilde{f}$ of monic
triangular type, we will specify a basis for $f$ of the same type.
This will allow us to  prove in Proposition \ref{prop completion to localization} that
the set $\mathcal{B}_f$  representing the latter basis also represents a set of $K[x]$-module generators for the local contribution to
$\overline{A}$ at $P$. This is a crucial result on our way to computing an integral basis for $\overline{A}$
over $K[x]$ via our local approach. An interesting side remark is that $\mathcal{B}_f$ represents, in
addition, a set of free generators for $\overline{K[X]_{\langle X \rangle}[Y]/\langle f \rangle}$ over ${K[X]_{\langle X \rangle}}$.

\begin{proposition}
\label{completion basis} Write $f$ as a product $f=f_{0}\widetilde{f}$ as in Equation \eqref{equation:branches}
of the introduction, where $f_{0}\in K[[X]][Y]$ is a unit in $K[[X,Y]]$ of degree $m_0$, and
$\widetilde{f}\in K[[X]][Y]$ is a Weierstrass polynomial of degree $m$. Suppose that $P=\left\langle x,y\right\rangle$
is the only singularity at $X=0$. Let
\[
\mathcal{B}_{\widetilde{f}} =\left\{1=p_{0},\frac{p_{1}}{X^{e_{1}}},\ldots,\frac{p_{m-1}}{X^{e_{m-1}}}\right\}%
\]
represent an integral basis for $\widetilde{f}$ of monic triangular type.
Let $\overline{f}_{0} \in K[X][Y]$ be a (monic) polynomial such that
\[
\overline{f}_{0}\equiv f_{0}\operatorname*{mod}X^{e_{m-1}}.
\]
Then
\[
\mathcal{B}_f =\left\{1,Y,Y^{2},\dots,Y^{m_0-1},\overline{f}_{0}p_{0},\frac{\overline{f}_{0}p_{1}%
}{X^{e_{1}}},\ldots,\frac{\overline{f}_{0}p_{m-1}}{X^{e_{m-1}}}\right\}%
\]
represents an integral basis for $f$ of monic triangular type.
\end{proposition}

\begin{proof}
Set  $e_0=0$. Since $\mathcal{B}_{\widetilde{f}}$ represents an integral basis for $\widetilde{f}$,
we obtain from Proposition \ref{prop loc int bas shape}  that the maximal integrality coefficients
with respect to $\widetilde{f}$ satisfy $e_{\widetilde{f},d}=e_d$, for $d=0,\ldots,m-1$. Our assertion, in turn, will follow
by applying  Proposition \ref{prop loc int bas shape} to $f$, provided we show that
$e_{f,d}=0$, for $d=1,\ldots,m_{0}-1$, and $e_{f,d}=e_{d-m_0}$, for $d=m_0,\ldots,n-1$.
For this, let $q\in K[[X]][Y]$ be any monic polynomial of degree $1\leq d \leq n-1$
in $Y$, and let $\eta_{1},\dots,\eta_{d}$ be the Puiseux expansions of $q$. Factorize $q$ as
\[
q=q_{0} \widetilde{q},
\]
where $q_{0}\in K[[X]][Y]$ is a unit in $K[[X,Y]]$ and $\widetilde{q} \in K[[X]][Y]$ is a Weierstrass polynomial.

Let $\gamma_1,\dots,\gamma_{m_0}$ be the Puiseux expansions of $f_{0}$,
and let $\gamma_{m_0+1},\dots,\gamma_n$ be those of $\widetilde{f}$. Note that for the latter
$\gamma_i$, the constant terms are zero. For the former $\gamma_i$, the constant terms,
say $a_0^{(i)}$, are non-zero and,  since we suppose that the origin is the only singularity
at $X=0$, pairwise different. Further note that if for some $1\leq i\leq m_0$ there is no
expansion $\eta_{j}$, $1\leq j\leq d$, with initial term $a_{0}^{(i)}$, then
\[
\upsilon_{f}(q)=\min_{1\leq i\leq n}\upsilon_{\gamma_i}(q)=\min_{1\leq i\leq n}%
\sum_{j=1}^{d}\upsilon(\gamma_{i}-\eta_{j})=0\text{.}%
\]

If $d<m_0$, then we can always find an initial term $a_{0}^{(i)}$ as above
since there are $m_0$ pairwise different initial terms. Since $q$ was chosen
arbitrarily, this implies that $e_{f,d}=0$ for $d<m_0$.

Now suppose that $d\geq m_0$. Then, if $e_{f}(q)>0$, any
$a_{0}^{(i)}$ must appear as the initial term of some $\eta_{j}$,
$1\leq j\leq d$. In particular,
\[
m_0\leq\deg_{Y}(q_{0}).
\]

We claim that $e_{f}(q) \leq e_{d-m_0}$.
This is clear if $e_{f}(q)=0$. To prove the claim if $e_{f}(q)>0$, note
that for any $i$, we have
\[
\upsilon_{\gamma_i}(q)=\upsilon_{\gamma_i}(q_{0})+\upsilon_{\gamma_i}(\widetilde{q})=\left\{
\begin{tabular}
[c]{ll}%
$\upsilon_{\gamma_i}(\widetilde{q})$ & if $\gamma_i(0)=0$\\
$\upsilon_{\gamma_i}(q_{0})$ & if $\gamma_i(0)\neq0$.
\end{tabular}
\ \right.
\]
Hence,
\begin{align*}
\upsilon_{f}(q)  &  =\min_{1\leq i\leq n}\upsilon_{\gamma_i}(q)=\min\left\{  \min_{1\leq
i\leq m_0}\upsilon_{\gamma_i}(q_{0}),\min_{m_0<i\leq n}\upsilon_{\gamma_i}(\widetilde{q})\right\} \\
&  =\min\{\upsilon_{f_{0}}(q_{0}),\upsilon_{\widetilde{f}}(\widetilde{q})\}\leq \upsilon_{\widetilde{f}}(\widetilde{q})\text{.}%
\end{align*}
Since $m_0\leq\deg_{Y}(q_{0})$, we conclude that
\[
e_f(q)=\left\lfloor \upsilon_{f}(q)\right\rfloor \leq\left\lfloor \upsilon_{\widetilde{f}}(\widetilde{q})\right\rfloor
\leq e_{\widetilde{f}, d-\deg_{Y}(q_{0})}\leq e_{\widetilde{f},d-m_{0}}=e_{d-m_0},
\]
which shows our claim.  Since $q$ was chosen arbitrarily, it follows that
$e_{f,d} \leq e_{d-m_0}$. To prove that these numbers are equal, we set $k=d-m_0$
and show that $\overline{f}_{0}p_{k}/x^{e_{k}}$ is integral over $A$.
For this, let $\gamma=\gamma_i$ be any Puiseux expansion of
$f$. If $\gamma(0)=0$, then $\upsilon_{\gamma}(\overline{f}_{0}p_{k})\geq \upsilon_{\gamma}(p_{k})
\geq \upsilon_{\widetilde{f}}(p_{k})\geq e_{k}$. If $\gamma(0)\neq0$, then $\upsilon_{\gamma}(\overline{f}_{0}p_{k})
\geq \upsilon_{\gamma}(\overline{f}_{0})\geq e_{m-1}\geq e_{k}$ by the very definition of $\overline{f}_{0}$.
\end{proof}

\begin{remark}
\label{rem:int-basis-loc}
It is clear from Proposition \ref{prop loc int bas shape} that Proposition \ref{completion basis}
holds more generally for polynomials $p_d$ in $K[[X]][Y]$ instead of just $K[X][Y]$; in addition,
it is not necessary to truncate $f_0$ to $\overline{f}_0$. However, in
its above form,  $\mathcal B_f$ also represents both a set of free generators for
$\overline{K[X]_{\langle X \rangle}[Y]/\langle f \rangle}$ over $K[X]_{\langle X \rangle}$ (by faithful flatness)
and (as we will show next)  a set of $K[x]$-module generators for the
minimal local contribution to $\overline{A}$ at $P$.
\end{remark}

\begin{proposition}
\label{prop completion to localization}
Let
\[
\mathcal{B} =\left\{1=p_{0},\frac{p_{1}}{X^{e_{1}}},\ldots,\frac{p_{n-1}}{X^{e_{n-1}}}\right\}%
\]
represent an integral basis  for $f$ of monic triangular type.
Suppose that $P=\left\langle x,y\right\rangle$ is the only singularity  at $X=0$. Then $\mathcal{B}$ also represents
a set of $K[x]$-module generators for the
minimal local contribution to $\overline{A}$ at $P$.
\end{proposition}

\begin{proof}
By the assumption and Proposition \ref{prop loc int bas shape}, we have
$e_d=e_g(p_d)=e_{g,d}$ for all $d$. Write $A_{d}^{\prime}=\left\langle 1,
\frac{p_{1}(x,y)}{x^{e_{1}}},\ldots,\frac{p_{d}(x,y)}{x^{e_{d}}}\right\rangle _{K[x]}$
for each $d$, and $A^{\prime}= A_{n-1}^{\prime}$. Then
$A\subset A^{\prime}\subset\overline{A}$ by Proposition \ref{prop:integrality-local-II}. 
To show that $A^{\prime}$ is the minimal local contribution to $A$ at $P$,
we proceed in three steps.

\emph{Step 1}: We show: If $q\in K[X][Y]$ is a polynomial of degree $0\leq d\leq n-1$
in $Y$ such that $\frac{q(x,y)}{x^{e}}\in\overline{A}$ for some $e\in\N$, then
$\frac{q(x,y)}{x^{e}}\in A_{d}^{\prime}$. 

As in the proof of Proposition \ref{prop loc int bas shape}, we do induction on $d$.
There is nothing to show in case $d=0$. If $d\geq 1$, let $c\in K[X]$ be the leading
coefficient of $q\in K[X][Y]$ and write   $q=c\widetilde{q}=X^t \tilde{c}\widetilde{q}$,
where $\widetilde{q}\in K((X))[Y]$ is monic in $Y$ of degree $d$, and $\tilde{c}\in K[X]$
is not a multiple of $X$. Then, by  Lemma \ref{same exp} and Theorem
\ref{theorem:integrality local},  $\lfloor v_{f}(\widetilde{q})\rfloor \leq e_{g,d}=e_d$,
hence
\[
e\leq t+ e_g(\widetilde{q}) = t+\lfloor v_{f}(\widetilde{q})\rfloor\leq t+e_d.
\]
This implies
$
\frac{c(x)p_{d}(x,y)}{x^{e}}=\tilde{c}(x)\frac{p_{d}(x,y)}{x^{e-t}}\in
A_{d}^{\prime}\subset\overline{A}\text{.}%
$
Since $\deg_{Y}(q-c p_{d})<d$ and $\frac{q(x,y)}{x^{e}}-\frac
{c(x)p_{d}(x,y)}{x^{e}}\in\overline{A}$, the induction hypothesis gives
$
\frac{q(x,y)}{x^{e}}-\frac{c(x)p_{d}(x,y)}{x^{e}}\in A_{d-1}^{\prime}\subset
A_{d}^{\prime}.%
$
Therefore $\frac{q(x,y)}{x^{e}}\in A_{d}^{\prime}$, as claimed.

\emph{Step 2}: Having defined $A^{\prime}$ as an intermediate $K[x]$-module
$A\subset A^{\prime}\subset\overline{A}$, we now show that $A^{\prime}$ is, in fact, an intermediate
ring and, thus, an $A$-module. That is, we show that $A^{\prime}$ is closed under multiplication.
For this, note that any product of two elements of $A^{\prime}$ takes the form
\[
\frac{q(x,y)}{x^{e}}\cdot\frac{q^{\prime}(x,y)}{x^{e^{\prime}}}=\frac
{q^{\prime\prime}(x,y)}{x^{e+e^{\prime}}}\in \overline{A},
\]
where $q^{\prime\prime}\in K[X][Y]$ satisfies $\deg_{Y}(q^{\prime\prime})<n$.
But then, by step 1, the product is in $A^{\prime}$.

\emph{Step 3}: We finally localize: Set
\[
D=K[x]_{\left\langle x\right\rangle }[y] = K[X]_{\left\langle X\right\rangle }[Y]/\langle f\rangle\text{ and }%
D^{\prime}=\left\langle 1,\frac{p_{1}(x,y)}{x^{e_{1}}},\ldots,\frac
{p_{n-1}(x,y)}{x^{e_{n-1}}}\right\rangle _{K[x]_{\left\langle x\right\rangle }}.%
\]
Then $D\subset D^{\prime}\subset\overline{D}$. In fact, it follows by faithful flatness
that $D^{\prime}=\overline{D}$ (see Remark \ref{rem:int-basis-loc} above). We also
give a direct argument for this equality: Let $\frac{q(x,y)}{h(x)}%
\in\overline{D}$ be an arbitrary element of $\overline{D}$, with polynomials
$q\in K[X][Y]$ of $Y$-degree $<n$ and $h\in K[X]$. Write $h$ as a product
$h=X^{e}\cdot \widetilde{h}$, where $\widetilde{h}(x)$ is a unit in $K[x]_{\left\langle x\right\rangle }$.
Then also $\frac{q(x,y)}{x^{e}}\in\overline{D}$, so that there exists a polynomial $p\in K[X]$
such that $p(x)$ is a unit in $K[x]_{\left\langle x\right\rangle }$ and
$p(x)\frac{q(x,y)}{x^{e}}\in\overline{A}$. It follows from step 1 that
$p(x)\frac{q(x,y)}{x^{e}}\in A^{\prime}\subset D^{\prime}$, so that $\frac
{q(x,y)}{h(x)}\in D^{\prime}$ and, hence, $D^{\prime}=\overline{D}$.
This implies that
\[
A_{Q}^{\prime}=\overline{D}_{Q}=\overline{D_{Q}}=\overline{{A_{Q}}}%
\]
for all $Q\in\operatorname{Spec}(A$) with $\left\langle x\right\rangle \subset Q$.
On the other hand, since we suppose that $P = \langle x, y\rangle$ is the only singularity
at $X = 0$,  we have $\overline{{A_{Q}}} = A_Q$ for all $Q\in\operatorname{Spec}A$
with $\left\langle x\right\rangle \subset Q$ and $Q \neq P$. Moreover,
$A'_{Q}=A_{Q}$ for all $Q$ with $\left\langle x\right\rangle
\not \subset Q$ since the denominators of the generators of $A^{\prime}$ are contained
in $\left\langle x\right\rangle $. We conclude that $A^{\prime} $ is the minimal contribution
to $A$ at $P$.
\end{proof}

\begin{remark}
\label{ypowers}
Computing a set $\mathcal B$ as in Proposition \ref{prop completion to localization} requires
that we know the precision $t$ up to which all power series in $X$ appearing in the process must be developed.
Of course, the maximum power of $X$ appearing in the denominators of the elements of $\mathcal B$
will do. However, this number is known to us only a posteriori, once $\mathcal{B}$ has already been computed.
We will address this problem in Section \ref{sect:loc-contr-via-Hensel}.
\end{remark}

In the setting of the example below, it turns out that the desired precision is $t=3$.

\begin{example}
\label{exampleTwoBranches} Factorizing the polynomial $f = (Y^{3}+X^{2})(Y^{2}-X^{3})+Y^{6}$
from Example \ref{exampleTwoBranches-int} as in Equation
\ref{equation:branches} of the introduction, we get $f = f_0\cdot f_{1} \cdot f_{2}$, where
$f_0 \equiv Y + (-X^{3} - X^{2} + 1)$, $f_{1} \equiv Y^{3} + (X^{3} + X^{2})
Y^{2} + (- X^{2}) Y + X^{2}$, $f_{2} \equiv Y^{2} - X^{3} \mod X^4$.

Applying Proposition \ref{prop split} to the product $f_1\cdot f_2$, we set $h_{1}=f_{2}$ and use the
extended Euclidean algorithm to compute the B\'ezout identity
$a_{1}f_{1}+b_{1}h_{1}=X^{2}$, where
$a_{1}\equiv-4X^{3}Y-2X^{3}
-2X^{2}Y-X^{2}+XY-Y-1$, $b_{1}\equiv-4X^{3}Y^{2}-2X^{3}Y-2X^{2}Y^{2}%
-3X^{3}-2X^{2}Y+XY^{2}-Y^{2}-Y \mod X^4$.

Computing integral bases for $f_1$ and $f_2$, we get
$\left\{1, y, \frac{y^{2}}{x}\right\}$ and $\left\{1, \frac{y}{x}\right\}$, respectively
(proceed as in Section \ref{section:lowdegree} below). Hence,
by Proposition \ref{coro complete split}, the union of the sets
\[
\widetilde{\mathcal{B}}^{(1)} = \left\{  \frac{b_{1} f_{2}}{X^{2}}, \frac{b_{1}
f_{2} Y}{X^{2}}, \frac{b_{1} f_{2} Y^{2}}{X^{3}}
\right\}
\quad \quad \text{and} \quad \quad
\widetilde{\mathcal{B}}^{(2)} = \left\{  \frac{a_{1} f_{1}}{X^{2}}, \frac{a_{1}
f_{1} Y}{X^{3}} \right\}
\]
represents an integral basis for  $f_1 \cdot f_2$. Proceeding as in Remark \ref{rem: HN-mts}, we get the set
\[
\left\{1,Y,\frac{Y^{2}}{X},\frac{Y^{3}}{X^{2}},\frac{Y^{4}+X^{2}Y}{X^{3}}\right\}%
\]
which represents an integral basis for $f_1\cdot f_2$ of monic triangular type.

Finally, applying Proposition \ref{completion basis}, with $\overline{f}_0 = Y + (-X^{3} - X^{2} + 1)$, we conclude that
\[
\left\{1, \overline{f}_0, \overline{f}_0 Y,\frac{\overline{f}_0 Y^{2}}{X},\frac{\overline{f}_0 Y^{3}}{X^{2}},\frac{\overline{f}_0 (Y^{4}+X^{2}Y)}{X^{3}}\right\}%
\]
represents an integral basis for $f$ of monic triangular type.

\end{example}

\section{Normalization of Plane Curves via Localization and Completion: The
Algorithmic Point of View}

\label{sect:loc-contr-via-Hensel}

Let $A=K[C]=K[x,y]=K[X,Y]/\langle f(X,Y)\rangle$ be as before. In this section, we
present our complete algorithm for computing the minimal local contributions
to $\overline{A}$ at the primes $P\in\Sing(A)$. In particular, we discuss effective
ways of finding the branches of $f$, and of computing integral bases for these.
The normalization $\overline{A}$ itself and an integral basis for $\overline{A}$
over $K[x]$, respectively, are then obtained along the lines of Proposition
\ref{prop:local-to-global} and Remark \ref{rem:spec-int-basis-II}.

We start with a sketch of the algorithm. Here, for simplicity of the presentation,
we assume that there are no two singular points of $C$ with the same $X$-coordinate.
This allows us to apply the results of Section \ref{sect:local-contribution} to
all singularities. How to deal with a curve having singular points with the same
$X$-coordinate will be addressed in Remarks \ref{remark rotation} and
\ref{rem no rotation}. See  Examples \ref{example rotation} and
\ref{ex no rotation} for illustrations.

\subsection{Summary of the Algorithm}
\label{section:summary}

From a \textbf{theoretical point of view}, the algorithm involves the
following steps which will be applied to each prime $P\in\Sing(A)$:

\begin{enumerate}
[leftmargin=0.9cm]

\item If $P$ corresponds to a $K$-rational
singularity, \textbf{translate the singularity to the origin}. If $P$
corresponds to a set of conjugate singularities over $K$, extend the base
field $K$ as needed, and \textbf{translate one of the singularities to the
origin}.
\end{enumerate}

\noindent For the singularity at the origin, do (to simplify the presentation,
we will still write $K$ and $f$ for the extended field and the
transformed equation of our curve, respectively):

\begin{enumerate}
[resume,leftmargin=0.9cm]

\item Taking Lemma \ref{same exp} and Theorem \ref{theorem:integrality
    local} into account, use the recipe from Lemma \ref{lemma:intA} to determine
the \textbf{maximum integrality exponent} $E(f)$.

\item \label{step ci}Set $c_0=0$. For $i=1,\dots, r$, determine integers $c_{i}$ as in Proposition
\ref{prop split}. Then \textbf{factorize} $f=\prod_{i=0}%
^{r}f_{i}$ as in Equation \eqref{equation:branches} of the introduction,
developing each $f_{i}$ up to $X$-degree $E(f)+c_{i}$. For this, make use of Hensel's lemma
and the Newton-Puiseux algorithm as described in Sections
\ref{sec Hensel} and \ref{sec Hensel loc} below.

\item \label{step bi} For $i=1,\dots, r$, compute the B\'ezout coefficients $b_{i}$ from Proposition
\ref{prop split} up to  $X$-degree $E(f)+c_{i}$.

\item Use Algorithm \ref{alg:TruncateGeneral} in Section \ref{section:lowdegree} below to compute for each $1\leq i\leq r$
an \textbf{integral basis for the branch} $f_{i}$ of monic triangular type.

\item \label{step int-branches} Based on the results of the previous steps,  use the recipe from Proposition \ref{coro complete split}
to construct for each $1\leq i\leq r$ a set $\mathcal{B}^{(i)}$ as in the proposition. Then convert
$\mathcal{B}^{(1)}\cup\ldots\cup \mathcal{B}^{(r)}$ into an integral basis for $f_1 \cdots f_r$ of monic
triangular type  by using unimodular row operations as in Remark \ref{rem: HN-mts} (take
Corollary \ref{cor:mon-triang-basis} into account).

\item \label{step int-basis-f} Starting from the integral basis for $f_1 \cdots f_r$ obtained in step \eqref{step int-branches}, compute an integral basis
$\overline{\mathcal B}$ for $f$ of monic triangular type. Here, use Algorithm \ref{alg:HenselBasis}
in Section \ref{subsect:local-contr} below which is based on Proposition \ref{completion basis}. It is then
clear from the proposition that $\overline{\mathcal B}$ also represents a set of $K[x]$-module generators for the
minimal local contribution to $\overline{A}$ at $\langle x,y \rangle$.

\item \label{step inv-transl} If necessary, apply the \textbf{inverse translation} to the elements of the local
contribution to restore the singularity to the original position.

\item If $P$ corresponds to a set of conjugate singularities, then use Remark
\ref{rmk conjugated} below to modify the numerators and denominators of the local
contribution obtained in steps \eqref{step int-basis-f} and \eqref{step inv-transl}  for one of the singularities over the extended field
in order to obtain the \textbf{minimal local contribution to
$\boldsymbol{\overline{A}}$ at $\boldsymbol{P}$} over the original field.

\end{enumerate}

From a \textbf{practical point of view}, we face the problem that in the
approach outlined above, we need to determine the $c_{i}$ \emph{a priori}.
Moreover, the computation of the B\'ezout coefficients $b_{i}$ via the
extended Euclidean algorithm is very time consuming. To remedy these issues,
relying on Proposition \ref{prop local int basis modified} below, we will
replace the $b_{i}$ and $c_{i}$ in steps \eqref{step ci} and \eqref{step bi} by
 easier to construct polynomials $\beta_{i}\in K[X,Y]$ and
appropriate vanishing orders, respectively.

We refer to the following subsections for more details.

\subsection{Puiseux Expansions}
\label{sect:PuiseuxBlock}

As already pointed out in Remark \ref{rem:branches-and-PS}, the factors
$f_{i}$ appearing in the decomposition
\[
f=f_{0}\widetilde{f}=f_{0} f_{1}\cdots f_{r}%
\]
of $f$ as in Equation \eqref{equation:branches} of the introduction can be found
by computing the Puiseux expansions of $f$ (up to a given degree). Since this is
expensive, however, we propose a different approach which, via
Hensel's lemma, makes considerably less use of the Newton-Puiseux algorithm.

In describing the new approach, we use the following terminology.
If $g\in K[[X]][Y]$ is any square-free monic polynomial of degree $\geq 1$ in $Y$,
we partition the set of all Puiseux expansions of $g$ into \emph{Puiseux
blocks}. The Puiseux block represented by an expansion $\gamma$ with
$\gamma\left(  0\right)  =0$ is obtained by collecting all expansions whose
rational part agrees with that of $\gamma$ and whose first non-rational term
is conjugate to that of $\gamma$ over $K((X))$. The \emph{Puiseux segment}
represented by an expansion $\gamma$ with $\gamma\left(  0\right)  =0$
is defined to be the union of all blocks whose expansions have  the same
initial exponent as $\gamma$. That is, we have one Puiseux segment for each face of the
Newton polygon of $g$. In addition, all Puiseux expansions $\gamma$ of $g$
with $\gamma\left(  0\right)  \not =0$ are grouped together to a single
Puiseux block of an extra Puiseux segment. In this way, the Puiseux expansions
of $g$ are divided into Puiseux segments, each segment consists of Puiseux
blocks, and each block is the union of classes of conjugate expansions.

\begin{example}
\label{examplePuiseux2} Suppose that the Puiseux expansions of our given
polynomial $f$ are
\begin{equation*}
\begin{aligned}[t]
\gamma_1 &= 1 + X^2 + \dots, \\
\gamma_2 &= -1 + 3X + \dots, \\
\gamma_3 &= a_1 X^{3/2} + 2X^2 + \dots, \\
\gamma_4 &= a_2 X^{3/2} + 2X^2 + \dots, \\
\gamma_5 &= X + 3 X^2 + \dots, \\
\end{aligned}
\quad\quad
\begin{aligned}[t]
\gamma_6 &= X + b_1 X^{5/2} + X^3 + \dots, \\
\gamma_7 &= X + b_2 X^{5/2} + X^3 + \dots, \\
\gamma_8 &= X + b_1 X^{5/2} + X^{4} + \dots, \\
\gamma_9 &= X + b_2 X^{5/2} + X^{4} + \dots,
\end{aligned}
\end{equation*}

\noindent where $\{\gamma_{3}, \gamma_{4}\}$, $\{\gamma_{6}, \gamma_{7}\}$ and
$\{\gamma_{8}, \gamma_{9}\}$ are pairs of conjugate Puiseux series. Then
$\{\gamma_{1}, \gamma_{2}\}$ is the segment of expansions $\gamma$ of $f$
with $\gamma\left(  0\right)  \not =0$. Another segment is $\{\gamma_{3},
\gamma_{4}\}$ which consists of one block containing a single class of
conjugate expansions. All the other expansions form a single segment,
consisting of the blocks $\{\gamma_{5}\}$ and $\{\gamma_{6}, \gamma_{7},
\gamma_{8}, \gamma_{9}\}$. The last block contains two classes of conjugate
expansions, namely $\{\gamma_{6}, \gamma_{7}\}$ and $\{\gamma_{8}, \gamma
_{9}\}$.
\end{example}

\subsection{Hensel's Lemma\label{sec Hensel}}

We will use Hensel's lemma in the following form:

\begin{lemma}[Hensel's Lemma]
\label{lemma:hensel} Let $F \in K[[X]][Y]$ be a monic polynomial in $Y$.
Assume that $F(0, Y) = g_{0} h_{0}$, with monic polynomials $g_{0}, h_{0} \in
K[Y]$ such that $\langle g_{0}, h_{0}\rangle= K[Y]$. Then there exist unique
monic polynomials $G, H \in K[[X]][Y]$ such that

\begin{enumerate}
\item $F= GH$,

\item $G(0, Y) = g_{0}$, $H(0, Y) = h_{0}$.
\end{enumerate}

\noindent
In fact, for each $k \in{\mathbb{N}}$, there exist unique $g_{k}, h_{k} \in
K[X,Y]$ of $X$-degree $\leq k$ such that

\begin{enumerate}
[resume]

\item $F \equiv g_{k} h_{k}$ in $(K[[X]] / \langle X^{k+1} \rangle)[Y]$,

\item $g_{k} \equiv g_{i}$, $h_{k} \equiv h_{i}$ in $(K[[X]]/ \langle X^{i+1}
\rangle)[Y]$, $i = 0, \dots, k-1$.
\end{enumerate}
\end{lemma}

\begin{proof}
See, for example, \citet{Abhyankar}.
\end{proof}

Conditions (3) and (4) imply that the polynomials $g_{k}$ and $h_{k}$ can be
computed inductively along the $X$-degree, solving for each $k$ a system of
$\ell$ linear equations in $\ell$ variables, where $\ell$ is the $Y$-degree of $F$: For
each $0 \leq i \leq \ell-1$, we get an equation by comparing the coefficients of
$X^{k}Y^{i}$ in $F$ with those in $g_{k}h_{k}$. For further reference in this paper,
we state the resulting procedure as Algorithm \ref{alg:Hensel}, \texttt{HenselLift}, omitting the
actual computation steps.

\begin{algorithm}     
\caption{\texttt{HenselLift}}          
\label{alg:Hensel}
\begin{algorithmic}[1]
\REQUIRE $\degbound\in\N$; $F \in K[X,Y]$ monic in $Y$;  $g_0, h_0 \in K[Y]$ monic with $F(0, Y) = g_0 h_0$,
$\langle g_0, h_0 \rangle = K[Y]$.
\ENSURE $G, H \in K[[X]][Y]$ developed up to $X$-degree $\degbound$, with $G(0, Y) = g_0$,
$H(0, Y) = h_0$, and $F \equiv GH\mod X^{\degbound+1}$.
\end{algorithmic}
\end{algorithm}

Our first application of $\HenselLift$ is to
address the Puiseux segment consisting of all Puiseux expansions $\gamma$ of
$f$ with $\gamma\left(  0\right)  \not =0$. That is, we decompose $f$ as
$f=f_{0} \widetilde{f}$ as in Equation \eqref{equation:branches} of the
introduction, separating the
unit $f_{0}$ from the component $\widetilde{f}$ vanishing at the origin (we
develop $f_{0}$ and $\widetilde{f}$ up to the desired $X$-degree). This is summarized
in Algorithm \ref{alg:Splitting-off-the-unit}, \texttt{SeparateUnit}.

\begin{algorithm}[H]                      
\caption{\texttt{SeparateUnit}}          
\label{alg:Splitting-off-the-unit}
\begin{algorithmic}[1]
\REQUIRE $\degbound\in\N$; $f \in K[X,Y]$ irreducible and monic in $Y$, with $f(0,0)=0$.
\ENSURE $f_0, \widetilde{f} \in K[[X]][Y]$ as in Equation \eqref{equation:branches} of the introduction, developed
up to $X$-degree $\degbound$.
\STATE compute monic $g_0, h_0 \in K[Y]$ with $Y \nmid g_0$, $h_0 = Y^k$ for some $k \in \N_{\geq 1}$, and $f(0, Y) = g_0 h_0$
\RETURN $\HenselLift(\degbound,f,g_0, h_0)$
\end{algorithmic}
\end{algorithm}

\begin{example}
Let $f = (Y-X)(Y+X)(Y+2X) + Y^{7}\in\Q[X,Y]$. Then there are
four Puiseux expansions of $f$ with $\gamma(0)\not =0$ and three expansions with
$\gamma(0)=0$ (note that $f(0,Y) = Y^3+Y^7=Y^3(1+Y^4)$). We write
$\gamma_{1}, \dots, \gamma_{4}$ for the former expansions and $\gamma_{5} = X
+ \dots, \gamma_{6} = -X + \dots, \gamma_{7} = -2X + \dots$ for the latter
ones. We apply Algorithm \ref{alg:Splitting-off-the-unit} to develop the products
$f_{0}=\gamma_{1} \cdots\gamma_{4}$ and $\widetilde{f} =\gamma_{5} \gamma_6
\gamma_{7}$ up to $X$-degree 2, calling $\HenselLift(2, f,g_{0}%
,h_{0})$ with $g_{0} = 1 + Y^{4}$ and $h_{0} = Y^{3}$.
The output is $g_{2} = 5X^{2}Y^{2}-2XY^{3}+Y^{4}+1$, $h_{2} = Y^{3} +
2XY^{2} -2X^{2}Y$.
\end{example}

An alternative way would be to decompose $f=f_{0} \widetilde{f}$ by means of the
Weierstrass Division Theorem. However, applying Hensel's lemma allows for
more generality since it does not require to have one factor vanishing at the origin.
This will be useful in  Section \ref{sec Hensel loc} below, where we will study a local
version of Hensel's lemma. Furthermore, in cases where the singularity under
consideration has no $K$-rational coordinates, we may use Hensel's lemma  to modify
our algorithms so that there is no need to move the singularity to the origin.
As a consequence, no field extension is required at this point. For brevity
of the presentation, we do not give the details of this strategy.

\subsection{A Local Version of Hensel's Lemma\label{sec Hensel loc}}

Being able to decompose $f=f_{0} \widetilde{f}$ as discussed above, we now aim
at factorizing $\widetilde{f}\in K[[X]][Y]$ into the branches $f_1,\dots, f_r$ of $f$.

We begin by separating the different
Puiseux segments of $\widetilde{f}$. To describe how, let $g\in K[[X]][Y]$ be any
square-free monic polynomial of degree $m \geq 1$ in $Y$, and such that
$\gamma(0)=0$ for each Puiseux expansion $\gamma$ of $g$. Then,
since all factors of $g$ vanish at the origin, we cannot apply
Hensel's lemma directly: no matter how we choose $g_{0}, h_{0}$ with
$g=g_0h_0$, the condition $\langle g_{0}, h_{0}\rangle= K[Y]$ will not be
satisfied (consider, for example, the products
$(Y-\gamma_{1})(Y-\gamma_{2})(Y-\gamma_{3})$ and $(Y-\gamma_{4})(Y-\gamma_{5}) $
in Example \ref{exampleTwoBranches-int}).

To overcome this problem, we transform $g$ as explained in what
follows. Write
\begin{align*}
\gamma_{1}  &  = a_{1}^{(1)} X^{t_{1}^{(1)}} + a_{2}^{(1)} X^{t_{2}^{(1)}} + \dots,\\
\gamma_{2}  &  = a_{1}^{(2)} X^{t_{1}^{(2)}} + a_{2}^{(2)} X^{t_{2}^{(2)}} + \dots,\\
&  \vdots\\
\gamma_{m}  &  = a_{1}^{(m)} X^{t_{1}^{(m)}} + a_{2}^{(m)} X^{t_{2}^{(m)}} + \dots
\end{align*}
for the Puiseux expansions of $g$, where all $ a_{1}^{(i)}$ are non-zero. Suppose
for simplicity that $t := t_{1}^{(1)} = \min_{1 \leq i \leq m}t_{1}^{(i)}$.
Naively,  to separate the Puiseux segment corresponding to $t$ from
the rest, we are tempted to substitute $X^{t}Y$ for $Y$ in $g=
(Y-\gamma_{1}) \cdots(Y-\gamma_{m})\in K[[X]][Y]$ and cancel out $X^{t}$ in
all factors. However, this would introduce fractional exponents and force
us, thus, to leave $K[[X]][Y]$. We therefore proceed in a different way: Write
$t = u/v$, with $u,v \in\mathbb{N}_{\geq1}$ coprime, and set
\begin{align*}
F(X, Y)  &  = g(X^{v}, X^{u}Y) / X^{m u}\\
&  = (Y - (a_{1}^{(1)} + a_{2}^{(1)} X^{\tilde t_{2}^{(1)}} + \dots))\cdots(Y -
(a_{1}^{(m)} X^{\tilde t_{1}^{(m)}} + \dots)) \in K[[X]][Y].
\end{align*}

Then  $F$ has factors not vanishing at the origin, and these correspond to the
Puiseux expansions of $f$ forming the Puiseux segment with smallest initial
exponent $t$. Applying Hensel's lemma, reversing the transformation, and
iterating the process yields Algorithm \ref{alg:SegmentSplitting}.

\begin{algorithm}                      
\caption{\texttt{SegmentSplitting}}          
\label{alg:SegmentSplitting}
\begin{algorithmic}[1]
\REQUIRE        $\degbound \in \N$; $g\in K[[X]][Y]$ monic in $Y$, developed up to $X$-degree $\degbound$; we suppose
                that $\gamma(0)=0$ for each Puiseux expansion $\gamma$ of $g$.
\ENSURE Weierstrass polynomials $g_1, \dots, g_\ell \in K[[X]][Y]$, developed up to $X$-degree $\degbound$,
               with $g\equiv g_1 \cdots g_\ell\mod X^{\degbound+1}$, and each $g_i$ corresponding to
               precisely one Puiseux segment of $g$ as outlined above.
\STATE from the Newton polygon of $g$, read off the pairwise different initial exponents $t_1, \dots, t_\ell$
of the Puiseux expansions of $g$
\IF{$\ell = 1$} \RETURN $\{g\}$
\ENDIF
\STATE $t = u/v = \min\{t_1, \dots, t_\ell\}$, with $u,v \in\mathbb{N}_{\geq1}$ coprime
\STATE $m = Y$-degree of $g$
\STATE $F = g(X^v, X^uY) / X^{mu}$
\STATE compute monic $g_0, h_0 \in K[Y]$ with $Y \nmid g_0$, $h_0 = Y^k$
for some $k \in \N_{\geq 1}$, and $F(0, Y) = g_0 h_0$
\STATE $G, H= \HenselLift(v\degbound, F, g_0, h_0)$
\STATE  $\widetilde{G} = X^{(\deg_Y{G}) u}G$, $\widetilde H = X^{(\deg_Y{H}) u}H$
\STATE $g_1 = \widetilde G(X^{1/v}, Y/X^{u/v})$, $h = \widetilde{H}(X^{1/v}, Y/X^{u/v})$
\RETURN $\{g_1\} \cup \SegmentSplitting(\degbound,h)$
\end{algorithmic}
\end{algorithm}

\begin{remark}
See \citet[Theorem 5.1.17]{JP} for an alternative approach which extends the
Weierstrass Division Theorem.
\end{remark}

\begin{example}
\label{example separate unit}
Let $f=(Y^{2}+2X^{3})((Y+2X^{2})^{2}+X^{5}) + Y^{6}\in\Q[X,Y]$. Evaluating $f$ at $X =
0$, we get $f(0, Y) = (Y^{2}+1)Y^{4}$. Applying $\SeparateUnit(8, f)$ gives the
(truncated) factors
\begin{tiny}
\begin{align*}
f_0 &= -48X^8Y-210X^8-8X^7Y+56X^7+32X^6Y-4X^6-8X^5Y-X^5+12X^4-2X^3-4X^2Y+Y^2+1, \\
\widetilde{f} &= -46X^8Y^2+16X^8Y+8X^7Y^2-32X^6Y^3+2X^8+4X^6Y^2+8X^5Y^3+8X^7+X^5Y^2\\
&+8X^5Y+4X^4Y^2+2X^3Y^2+4X^2Y^3+Y^4.
\end{align*}
\end{tiny}

\vspace{-0.4cm}\noindent The Puiseux expansions of $\widetilde{f}$ are
\begin{tiny}
\begin{align*}
\gamma_{1,2}  &= a_{1,2} X^{3/2}  +a_{1,2} X^{9/2} -4X^5  -6 a_{1,2} X^{11/2}  +16 X^6 +41/2 a_{1,2} X^{13/2} -52 X^{7} +\dots, \\
\gamma_{3,4}  &= -2X^2 +b_{1,2} X^{5/2}  + 16 b_{1,2} X^{13/2} + 48 X^7  +    \dots,
\end{align*}
\end{tiny}

\vspace{-0.4cm}\noindent
with roots $a_{1,2}$ of $Z^{2}+2$ and $b_{1,2}$ of $Z^{2}+1$.

\noindent
The smallest initial exponent $t$ of these expansions is $t=u/v=3/2$. We compute

\begin{tiny}
\begin{align*}
F(X,Y)  &= \widetilde{f}(X^{2},X^{3}Y) / X^{12}= -46 X^{10}Y^2-32 X^{9}Y^3-64X^9Y+8 X^{8}Y^2+8 X^{7}Y^3+16X^8+16 X^{7}Y \\
& +4X^{6}Y^2+X^{4}Y^2+2X^{4}+4X^{2}Y^2+4XY^3+Y^4+8X^{2}+8XY+2Y^2.
\end{align*}
\end{tiny}

\vspace{-0.4cm}\noindent
Now note that $F(0,Y)=(Y^{2}+2)Y^{2}$. Applying Hensel's lemma
to the factors $Y^{2}+2$ and $Y^{2}$, we obtain
\begin{align*}
G(X,Y) &= \ldots + 116 X^{10}-48 X^9Y-16X^8+8X^7Y+4X^6+Y^2+2, \\
H(X,Y) &= \ldots + 30 X^{10}+16 X^9 Y-8 X^8+X^4+4X^2+4XY+Y^2.
\end{align*}
So if we set $\widetilde{G} (X, Y) = X^6 G(X,Y)$, $\widetilde{H} (X, Y) = X^6 H(X,Y)$ and apply the inverse transformations,
we get
\begin{scriptsize}
\begin{align*}
g_{1}  &= \widetilde{G} (X^{1/2}, Y/X^{3/2}) = \ldots + 116X^8-48 X^6Y-16X^7+8X^5Y+4X^6+Y^2+2X^3,\\
h  &=  \widetilde{H} (X^{1/2}, Y/X^{3/2}) = \ldots + 30 X^8+16X^6Y-8X^7+X^5+4X^4+4X^2Y+Y^2.
\end{align*}
\end{scriptsize}

\vspace{-0.4cm}\noindent
Since there are only two conjugacy classes of Puiseux expansions of $\widetilde{f}$,
we may conclude that $f_1=g_1$ and $f_2=h$ are the branches of $f$.

\end{example}

The next step is to split the Puiseux segments of $\widetilde{f}$
into their Puiseux blocks. Fix such a segment $\Gamma=\{\gamma_1,\dots,
\gamma_m\}$. Then the $\gamma_i$ satisfy $\gamma_i(0)=0$.  Moreover,
by the very definition of a segment, the $\gamma_i$ must have the
same initial exponent, say, $t=u/v$. Write $g$ for the factor of $\widetilde{f}$
corresponding to $\Gamma$, and $\eta$ for the largest common initial part
of the rational parts of the $\gamma_i$.

If $\eta =0$, set $F(X, Y)  = g(X^{v}, X^{u}Y) / X^{mu}$ as
above. Then $F(0,Y)$ will have different factors, all not vanishing at the origin,
and corresponding to the individual blocks. Hence, we can separate these blocks
iteratively, using Hensel's lemma as above.

If $\eta \not=0$, set
$$\widetilde{g}(X,Y) = g(X, Y + \eta).$$
Then the  Puiseux expansions of $\widetilde g$ coincide with those of $g$, except
that we omit the common initial term $\eta$. We can then use
segment splitting combined with block splitting to separate the blocks. Substituting
$Y - \eta$ for $Y$ to reverse the transformation, we get the desired factors.

We summarize this strategy in Algorithm \ref{alg:BlockSplitting}, \texttt{BlockSplitting}.
In Line \ref{linearFactor} of the algorithm, the presence of a power of a linear
factor implies that the relevant expansions share a common non-zero rational part.
It is, then, possible to decompose the corresponding factor further.

\begin{algorithm}                      
\caption{\texttt{BlockSplitting}}          
\label{alg:BlockSplitting}
\begin{algorithmic}[1]
\REQUIRE        $\degbound \in \N$; $g\in K[[X]][Y]$ monic in $Y$, developed up to $X$-degree $\degbound$; we suppose
                that $\gamma(0)=0$ for each Puiseux expansion $\gamma$ of $g$, and that these
                expansions form a single Puiseux segment.
\ENSURE               Weierstrass polynomials $g_1, \dots, g_k \in K[[X]][Y]$, developed up to $X$-degree $\degbound$,
               with $g \equiv g_1 \cdots g_k\mod X^{\degbound+1}$, and each $g_i$ corresponding to
               precisely one Puiseux block of the given Puiseux segment.

\STATE $L = \emptyset$
\STATE $\eta = $ the common rational part of all Puiseux expansions of $g$
\IF{$\eta = 0$}
  \STATE $t = u/v$, with $u,v \in\mathbb{N}_{\geq1}$ coprime, the initial exponent of the Puiseux expansions of
   $g$ (which is the same for all expansions by assumption and is obtained from the Newton polygon of $g$)
  \STATE $m = Y$-degree of $g$
    \STATE $F = g(X^v, X^uY) / X^{mu}$
  \STATE compute $g_0, h_0 \in K[Y]$ with $g_0 \neq 1$ irreducible or a power of an
               irreducible polynomial, $g_0$, $h_0$ coprime, and $F(0, Y)= g_0 h_0$.
  \IF{$h_0 \neq 1$}
    \STATE $G, H = \HenselLift(v\degbound, F, g_0, h_0)$
    \STATE $g_1 = G(X^{1/v}, Y/X^{u/v})$, $h = H(X^{1/v}, Y/X^{u/v})$
    \IF{$g_0$ is not a power of a linear factor in $Y$} \label{linearFactor}
      \RETURN $\{g_1\} \cup \BlockSplitting(\degbound, h)$
    \ELSE
      \RETURN $\BlockSplitting(\degbound, g_1)\cup \BlockSplitting(\degbound, h)$
    \ENDIF
  \ELSE
    \RETURN $\{g\}$
  \ENDIF
\ELSE
  \STATE $\widetilde g = g(X, Y + \eta)$
  \STATE \{$g_1, \dots, g_\ell\} = \SegmentSplitting(\degbound, \widetilde g)$
  \FOR{$1 \le i \le \ell$}
    \STATE \{$h_1, \dots, h_s\} = \BlockSplitting(\degbound, g_i)$
    \STATE $L = L \cup \{h_1(X, Y - \eta), \dots, h_s(X, Y - \eta)\}$
  \ENDFOR
  \RETURN $L$
\ENDIF
\end{algorithmic}
\end{algorithm}

\begin{remark}
We expect that the  ideas from \citet[Theorem 5.1.20]{JP} can in some cases also be used for
our purposes. However, as stated in \citet{JP}, the theorem is not as general as we require.
\end{remark}

The final step on our way to find the  branches $f_i$ is to separate the factors corresponding
to different conjugacy classes of Puiseux expansions within each given block.
For this, all algorithms known to us require that we  extend our base
field $K$ -- a factor $(Y-{\gamma}_{1})\cdots(Y-{\gamma}_{s})$ is obtained by computing the
individual  $\gamma_{i}$ up to the desired $X$-degree via the Newton-Puiseux algorithm,
and expanding the product. Of course, this last step is only needed if there is a Puiseux block
containing more than one conjugacy class of Puiseux expansions.

In Algorithm \ref{alg:Splitting}, we sum up the discussion above, arriving at
a general \texttt{Splitting} algorithm.

\begin{algorithm}                      
\caption{\texttt{Splitting}}         
\label{alg:Splitting}
\begin{algorithmic}[1]
\REQUIRE $\degbound \in \N$; $f \in K[X][Y]$ irreducible and monic in $Y$ of degree $n$.
\ENSURE $f_0, f_1, \dots, f_r \in K[[X]][Y]$ with
$f = f_0 f_1 \cdots f_r$ as in Equation \eqref{equation:branches} of the
introduction, all developed up to $X$-degree $\degbound$.
\STATE $f_0, \widetilde{f} = \SeparateUnit(\degbound, f)$
\STATE $L = \{f_0\}$
\STATE $\{g_1, \dots, g_{\ell}\} = \SegmentSplitting(\degbound, \widetilde{f})$, the factors corresponding to the different Puiseux
segments of $\widetilde{f}$
\FOR{$i = 1, \dots, \ell$}
\STATE $\{h_1, \dots, h_{s}\} = \BlockSplitting(\degbound, g_i)$
\FOR{$j = 1, \dots, s$}
\STATE $\Delta_1, \dots, \Delta_m = $ sets of singular parts of the Puiseux expansions of $h_j$, grouped into conjugacy classes
\IF{$m > 1$}
\FOR{$k = 1, \dots, m$}
\STATE $\gamma_{1}, \dots, \gamma_{s} = $ Puiseux expansions associated to $\Delta_k$, developed up to $X$-degree $\degbound$
\STATE $p = (Y- \gamma_{1}) \cdots (Y - \gamma_{s})$, developed up to $X$-degree $\degbound$
\STATE $L = L \cup \{p\}$
\ENDFOR
\ELSE
\STATE $L = L \cup \{h_j\}$
\ENDIF
\ENDFOR
\ENDFOR
\RETURN $L$
\end{algorithmic}
\end{algorithm}

\subsection{Integral Bases for the Branches}
\label{section:lowdegree}

We now explain how to find integral bases for the branches of $f$. More generally,
let  $g\in K[[X]][Y]$ be an irreducible Weierstrass polynomial of degree $m$.
Then the Puiseux expansions of $g$ form a complete conjugacy class
$\Gamma=\{\gamma_1,\dots, \gamma_m\}$.
Let $K\subset K(\alpha)$ be a finite field extension
over which all $\gamma_i$ are defined.

Taking Proposition \ref{prop loc int bas shape} and Corollary \ref{cor:mon-triang-basis}
into account, our goal is an algorithm, Algorithm \ref{alg:TruncateGeneral},  which constructs for each $1\leq d \leq m-1$ a monic polynomial $p_d\in K[X][Y]$
of degree $d$ in $Y$ whose integrality exponent satisfies $e_g(p_d) =\lfloor v_{g}(p_d)\rfloor$.
We begin by introducing what we call the extended characteristic exponents of $g$.
We use the following notation:
\begin{notation}
Given $\gamma\in L\{\{X\}\}$ and $t \in \mathbb{Q}_{\geq0}$, we write
$\overline{\gamma}^{(t)}$ (respectively $\overline{\gamma}^{(<t)}$)
for the truncation of $\gamma$ to degree $t$ (respectively $<t$).
\end{notation}

Consider the factorization
\begin{equation}
\label{eq factorization}
g=g_{1}\cdots g_{s}%
\end{equation}
of $g$ into absolutely irreducible Weierstrass polynomials $g_{i}\in K(\alpha)[[X]][Y]$,
and write $k := \deg(g_1) = \ldots = \deg(g_s)$. Factorize $g_1$ as
$$
g_1 = \prod (Y-\zeta (\omega^\ell X^{1/k})),
$$
where $\zeta\in K(\alpha)[[T]]$, and $\omega$  is a primitive $k$th root of unity (see Section
\ref{subsect: Puiseux-expansions}). Consider the action of $\operatorname*{Gal}(K(\alpha)/K)$
on $K(\alpha)[[T]]$, let $G= \operatorname*{Gal}(K(\alpha)/K)/\operatorname{Stab}(\zeta)$,
and set $H=G\times\mathbb{Z}/k\mathbb{Z}$. Let $\eta(X)=\zeta (X^{1/k})$.
Given $\sigma=(\overline{\varphi}, \ell)\in H$, write $\sigma\eta(X)=\varphi\zeta (\omega^\ell X^{1/k})$.
Then
\[
g=%
{\displaystyle\prod\nolimits_{\sigma\in H}}
(Y-\sigma\eta(X)).
\]

Now we truncate: If $t \in \mathbb{Q}_{\geq0}$, let
\[
l(t)=\left\vert H\overline{\eta}^{(t)}\right\vert
\]
be the orbit length under the action of $H$ after truncation. Moreover, write
\[
\left\{  l_{0}%
,\ldots,l_{\nu}\right\}:=\left\{  l(t)\mid t\in\mathbb{Q}_{\geq0}\right\},
\]
where the $l_i$ are sorted such that  $1=l_{0}<\ldots<l_{\nu}=m$. Then
\mbox{$l_i | l_{i+1}$}, for $0 \le i < \nu$. The corresponding minimal
truncation degrees achieving these orbit lengths are
\[
t_{i}:=\min\left\{  t\in\mathbb{Q}_{\geq0}\mid l(t)=l_{i}\right\}  \text{.}%
\]
We call $t_0, t_{1},\ldots,t_{\nu}$ the \emph{extended characteristic exponents} of
$g$.

Let $M=H\overline{\eta}^{<t_{\nu}}$ be the orbit of the truncation to
degree $<t_{\nu}$, and set
\[
\overline{m}:=\left\vert M\right\vert =l_{\nu-1}.
\]
Then
\begin{equation}
\label{equation g bar}
\overline{g}:={\displaystyle\prod\nolimits_{\rho(X)\in M}} (Y-\rho(X))\in K[X,Y]
\end{equation}
is an irreducible Weierstrass polynomial with extended characteristic exponents $t_{0},\ldots,t_{\nu-1}$.

\begin{remark}
\label{rem characteristic orders}
Characteristic exponents as introduced, for example, in \citet[Section 5.2]{JP}
are classical invariants of irreducible complex plane curve singularities. With notation
as above, if $g$ is absolutely irreducible (then $s=1$ in \eqref{eq factorization} and
$k=m$), and if we assume for simplicity that $\upsilon(\gamma)>1$, the
\emph{characteristic exponents} of $g$ are defined recursively as
\begin{align*}
k_0  &  := m,\\
k_{1}  &  :=\min\{j \mid b_{j}\neq0\text{ and }k_0\nmid j\},\\
k_{i}  &  :=\min\{j \mid b_{j}\neq0\text{, }\gcd(k_0, k_{1},\dots,k_{i-1})\nmid
j\}\text{ for } i > 1\text{, if this set is non-empty},%
\end{align*}
where $\gamma=\sum_{j > m}b_{j}X^{j/m}$. 
There are only finitely many such numbers $k_0 < k_{1} < \dots < k_{\nu}$, we have $\gcd(k_0, k_1, \dots, k_\nu) = 1$,
and the $t_{i}$ and $k_i$ are related by the following equalities:
\[
t_{i}=\frac{k_{i}}{m}, \;\text{ for }\;1 \le i \le \nu.
\]

Furthermore,  \citet[Theorem 5.2.16]{JP} yields the  valuation formula below (see also Section \ref{section:adhoc}):
\[
\upsilon_{g}(\overline{g})=\frac{k_{\nu}}{m}+\sum_{j=1}^{\nu-1}\frac{\gcd(k_0,k_1,\ldots,k_{j-1}%
)-\gcd(k_0,k_1,\ldots,k_{j})}{\gcd(k_0,k_1,\ldots,k_{\nu-1})}\frac{k_{j}}{m}.
\]

\end{remark}

We will use the following results to show the correctness of Algorithm \ref{alg:TruncateGeneral}.

\begin{lemma}
\label{lemma:pq}
Let $p, q\in K[X][Y]$ be two monic polynomials of the same degree $d$
in $Y$, where $1\leq d\leq m-1$. Suppose
that the Puiseux expansions of both $p$ and $q$ are truncations of Puiseux expansions
of $g$, and that there is a bijection between the sets of  expansions
of $p$ and $q$ such that each expansion of $p$ is a truncation of the corresponding expansion
of $q$. Then $\upsilon_g(p)  \leq \upsilon_g(q)$.
\end{lemma}

\begin{proof}
Let $p = \prod_{1 \le \ell \le d} (Y - \overline{\gamma}_{i_\ell}^{(s_\ell)})$ and $q = \prod_{1 \le \ell \le d} (Y - \overline{\gamma}_{i_\ell}^{(t_\ell)})$,
with Puiseux expansions $\gamma_{i_\ell}$ of $g$.  Let $\gamma\in\Gamma$ be an arbitrary Puiseux expansion
of $g$. Then $\upsilon(\gamma - \overline{\gamma}_{i_\ell}^{(s_\ell)}) \le \upsilon(\gamma - \overline{\gamma}_{i_\ell}^{(t_\ell)})$,
for each $1 \le \ell \le d$. Indeed, no cancellation in $\gamma - \overline{\gamma}_{i_\ell}^{(s_\ell)}$ can occur
that does not occur in $\gamma - \overline{\gamma}_{i_\ell}^{(t_\ell)}$ as well. Hence, by the valuation formula in Section
\ref{sect:max-expo}, $\upsilon_g(p) = \min_{\gamma \in \Gamma} \sum_{1 \le \ell \le d} \upsilon(\gamma - \overline{\gamma}_{i_\ell}^{(s_\ell)})
\le \min_{\gamma \in \Gamma}  \sum_{1 \le \ell \le d} \upsilon(\gamma - \overline{\gamma}_{i_\ell}^{(t_\ell)})= \upsilon_g(q)$, as claimed.
\end{proof}

\begin{remark}
\label{rem val mult}
Recall that all Puiseux expansions of $g$ are conjugate by our assumptions on $g$. Hence, if
$p \in K[X][Y]$, then $\upsilon_g(p) = \upsilon_\gamma(p)$ for each Puiseux expansion $\gamma$
of $g$.  We conclude that $\upsilon_g$ is additive in our setting here:  If $p, q \in K[X][Y]$, then
$\upsilon_g(pq) = \upsilon_g(p) + \upsilon_g(q)$ (this is not true in general since it may happen
that the valuations of $g$ at $p$ and $q$ are obtained as the valuations at expansions of $g$
in different orbits).
\end{remark}

\begin{lemma}
\label{lem poly by truncation}
For each $d = 1, \dots, m-1$, there is a monic polynomial $p_{d}\in K[X][Y]$ of degree
$d$ in $Y$ whose Puiseux expansions are all truncations of Puiseux expansions of $g$
and whose valuation at $g$ is the maximal valuation $\upsilon_{g}(q)$, for
$q\in K[[X]][Y]$ monic of degree $d$ in $Y$.
\end{lemma}

\begin{proof}
By Lemmas \ref{lemma:intA} and \ref{same exp}, for each given $d$, there is a polynomial $q_{d}\in K[X][Y]$ of degree
$d$ in $Y$ such that $\upsilon_{g}({q}_d)$ is the maximal valuation $\upsilon_{g}(q)$, for
$q\in K[[X]][Y]$ monic of degree $d$ in $Y$.  Arguing as in Section \ref{subsect: Puiseux-expansions},
since $q_{d}\in K[X][Y]$, we may group the Puiseux expansions  of $q_{d}$ into conjugacy classes
over $K(X)$ which correspond to the irreducible factors of $q_{d}$ in $K[X,Y]$.

Let $\{\chi_{1},\dots,\chi_{s}\}$ be any of these classes, and let $u = (Y-{\chi_{1}})\cdots(Y-{\chi_{s}})$
be the corresponding factor of $q_d$. Then no $\chi_{i}$ coincides with a Puiseux expansion of $g$ since otherwise
all expansions of $g$ would arise as expansions of $q_d$, a contradiction to $m>s$. The maximum degree
$t\in\mathbb{Q}_{\geq0}$ of a term of $\chi_{i}$ for which $\chi_{i}^{(t)}$ is a truncation of a Puiseux expansion of $g$ is independent
of the choice of $i$. We have $w = (Y-\overline {\chi}^{(t)}_{1})\cdots(Y-\overline{\chi}^{(t)}_{s}) \in K[X,Y]$ since the truncated
expansions are conjugate over $K(X)$ as well. Moreover, by construction, $\upsilon_g(w) \ge \upsilon_g(u)$.

Replacing $u$ by $w$ and proceeding in the same way with the other conjugacy classes of expansions of $q_d$,
we obtain a polynomial $p_{d}\in K[X][Y]$ with $\upsilon_{g}(p_d) = \upsilon_{g}(q_d)$, and such that each
Puiseux expansion of $p_d$ is a truncation of an expansion of $g$.
\end{proof}

\begin{example}
\label{example poly by truncation}
Consider the polynomial
$$g = X^{10}-2X^9-2X^8-4X^7Y-2 X^5 Y^2+X^4 Y^2+X^3 Y^2+2 X^2Y^3+Y^4\in\mathbb Q[X,Y],$$
and fix the degree $d=3$. The  Puiseux expansions of $g$ are
\begin{align*}
\gamma_{1,2} &= a_{1,2} X^{3/2} - X^2 - \frac{1}{2} a_{1,2} X^{11/2} + \dots, \\
\gamma_{3,4} &= b_{1,2} X^{5/2} - \frac{1}{4} b_{1,2} X^{9/2} + \dots,
\end{align*}
with roots $a_{1,2}$ of $Z^2 +  1$ and $b_{1,2}$ of $Z^2 - 2$.
From Lemma \ref{lemma:intA} we see that $\widetilde{p}_3 = (Y - \gamma_1)(Y - \gamma_2)(Y - \gamma_3)\in\overline{\Q} \{\{X\}\}[Y]$
is a polynomial of degree $3$ in $Y$  such that $\upsilon_{g}(\widetilde{p}_3) = 11/2$ is the maximal valuation $\upsilon_{g}(q)$, for
$q\in K[[X]][Y]$ monic of degree $3$ in $Y$.
Now note that the polynomial $q_3 = (X^6+2 X^5+X^4+2 X^3 Y+X^3+2 X^2 Y+Y^2)(Y-X^4)\in\mathbb Q[X][Y]$, whose Puiseux expansions
are $\chi_{1,2} = a_{1,2}X^{3/2}-X^2 - X^3$ and  $\chi_{3} = X^4$, satisfies $\upsilon_g(q_3) = 11/2$.
Proceeding as in the previous proof, we truncate $\chi_{1,2}$ to $\overline{\chi}_{1,2} = a_{1,2} X^{3/2}-X^2$.
Then, since no Puiseux  expansion of $g$ starts with the term $X^4$, we truncate $\chi_{3}$ to $\overline{\chi}_{3} = 0$.
In sum, $p_3 = (Y - \overline{\chi}_1) (Y - \overline{\chi}_2) (Y - \overline{\chi}_3)  = X^4 Y+X^3 Y+2 X^2 Y^2+Y^3$
is a polynomial as in Lemma \ref{lem poly by truncation}.

\end{example}

The result below is the key lemma for the recursion step of Algorithm \ref{alg:TruncateGeneral}:
\begin{lemma}
\label{lem gbar max val}The polynomial $\overline{g}\in K[X][Y]$ of degree $\overline{m}$ in $Y$ defined
in (\ref{equation g bar}) is such that $\upsilon_{g}(\overline{g})$ is the maximal valuation $\upsilon_{g}(q)$,
for $q\in K[[X]][Y]$ monic of degree $\overline{m}$ in $Y$.
\end{lemma}

\begin{proof}
Choose a monic polynomial $p=p_{\overline{m}}\in K[X][Y]$ of degree
${\overline{m}}$ in $Y$ as in Lemma ~\ref{lem poly by truncation}.
That is, the Puiseux expansions of $p$ are all truncations of Puiseux expansions of $g$,
and $\upsilon_{g}({p})$ is the maximal valuation $\upsilon_{g}(q)$, for
$q\in K[[X]][Y]$ monic of degree ${\overline{m}}$ in $Y$. We may suppose:
\begin{equation}
\label{equation:cond-proof-key-lemma}
\begin{aligned}
\text{From among all polynomials of degree $\overline{m}$ as in Lemma ~\ref{lem poly by truncation},}\\
\text{$p$ has the least number of irreducible factors in $K[X,Y]$.}  
\end{aligned}
\end{equation}
We show that $p$ is irreducible. Suppose the contrary, and let $p = p_1 \cdots p_s$ be the decomposition
of $p$ into irreducible factors $p_i\in K[X][Y]$ which are monic in $Y$, say of degrees $d_i$, where we suppose that  $d_1 \le d_2 \le \ldots \le d_s$. 
Then $\{d_1,\ldots,d_s\}\subset \{l_1,\ldots,l_{\nu-1}\}$, we have \mbox{$d_i | d_{i+1}$}, for $1 \le i < s$,  and \mbox{$d_s | \overline{m}$}.
Let $s'$ be maximal such that $d_1 = d_2 = \ldots = d_{s'}$. 
Then $s'>1$. Indeed, since both $d_{2}+\ldots +d_{s}$ and $d_1+d_{2}+\ldots +d_{s}=\overline{m}$  are divisible by $d_{2}$, also $d_1$ is divisible by $d_{2}$, so that $d_1=d_2$.
If $s'<s$, 
the same argument shows that $d_1 + \ldots + d_{s'}$ is divisible by $d_{s'+1}$.
Hence,  since $d_1 = d_2 = \ldots = d_{s'}$, there exists $1 < t \le s'$ such that $d_1  + \ldots + d_{t} = d_{s'+1}$.
If $s' = s$, let $t=s$. Then $d_1 +  \ldots + d_{t} = \overline{m}$, and we set $p_{s'+1}=\overline{g}$.
In any case, considering that the Puiseux expansions of the $p_i$ are finite, we may suppose that the largest exponent appearing in the expansions of $p_t$
is greater than or equal to the respective exponent of  each $p_i$,  $1\leq i \leq t-1$. Then the expansions of $p_1,\ldots,p_{t-1}$ arise as truncations of those of $p_t$.

Therefore, by Lemma \ref{lemma:pq},
$\upsilon_g(p_1 \cdots p_t) \leq \upsilon_g(p_t^t) \leq \upsilon_g(p_{s'+1})$ (for the second inequality note that if we write
$p_{s'+1}$ in terms of its Puiseux expansions, and truncate these expansions to the degree of the expansions of $p_t$, we get $p_t^{t}$).
It follows that  $\upsilon_g(p)\leq \upsilon_g(p_{s'+1}\cdot p_{t+1}\cdots p_{s})$, a contradiction to \eqref{equation:cond-proof-key-lemma}.
Hence $p$ is irreducible, as claimed.

Since $p$ is irreducible, and by the very definition of $\overline{g}$, the polynomials $p$ and $q=\overline{g}$ satisfy
the assumptions of Lemma \ref{lemma:pq}. So $\upsilon_g(p)\leq \upsilon_g(\overline{g})$,
which concludes the proof.
\end{proof}

We are now ready to formulate  Algorithm \ref{alg:TruncateGeneral} and proof its correctness.

\begin{algorithm}
\caption{\texttt{IntegralBasisElement}}
\label{alg:TruncateGeneral}
\begin{algorithmic}[1]
\REQUIRE $\Delta=\left\{  \delta_{1},\ldots,\delta_{m}\right\}  $, the set of singular parts of the Puiseux expansions of an irreducible
Weierstrass polynomial $g\in K[[X]][Y]$ of degree $m$; an integer $d$ with $1 \leq d \leq  m-1$.

\ENSURE $p_d\in K[X][Y]$ monic of degree $d$ in $Y$ such that $\upsilon_{g}(p)$ is
the maximal valuation $\upsilon_{g}(q)$, for $q\in K[X][Y]$ monic of degree $d$ in $Y$.

\STATE let $t=t_{\nu}$ be the maximal extended characteristic exponent of $g$

\STATE let $\rho_{1},\ldots, \rho_{\overline{m}}$ be the pairwise different
elements of $\left\{  \overline{\delta}_{1}^{<t},\ldots,\overline{\delta}%
_{m}^{<t}\right\}  $

\STATE
\label{step:u}
$u=\left\lfloor \frac{d}{\overline{m}}\right\rfloor $

\STATE
\label{step:dbar}
$\overline{d}=d-u\cdot\overline{m}$

\STATE $q=1$, $p=1$

\IF{$\overline{d}>0$}
\STATE
\label{step:recursion}
$q=\TruncatedFactor(\left\{  \rho_{1},\ldots
, \rho_{\overline{m}}\right\}  ,\overline{d})$
\ENDIF

\IF {$u>0$}
\STATE
\label{step:p}
$p={\displaystyle\prod\nolimits_{i=1}^{\overline{m}}}(Y-\rho_{i}(X))$
\ENDIF

\RETURN $p_d=p^{u}q$
\end{algorithmic}
\end{algorithm}

\begin{theorem}
Algorithm \ref{alg:TruncateGeneral} works correctly as specified.
\end{theorem}

\begin{proof}
We have to show that the polynomial $p_d$ returned by the algorithm has the desired maximal valuation at $g$ in degree $d$.
For this, we retain the notation introduced in the previous discussion. 
We write $p=\overline{g}={\displaystyle\prod\nolimits_{i=1}^{_{\overline{m}}}}(Y-\rho_{i}(X))$ and distinguish two cases.

\emph{Case 1}: Let $\overline{d}=0$. Then necessarily $u>1$ and $p_{d} = p^{u}$. By Lemma \ref{lem poly by truncation},
there is  a monic polynomial $p'_d\in K[X][Y]$ of degree $d$ in $Y$ whose Puiseux expansions are all truncations of Puiseux
expansions of $g$ and whose valuation at $g$ is maximal in degree $d$.
We have to show that $\upsilon_g(p'_d) \le \upsilon_g(p_d)$. Supposing the contrary, we may
write $p'_d$ as a product $p'_d = p^{u'}q'$, with $0 \le u' < u$ and $q' \in K[X][Y]$ monic in $Y$. Then
$\deg(q') \geq \deg(p) =\overline{m}$. We may assume that $p^{u'}$ is the maximal power of $p$ appearing as a factor of a
monic polynomial  $q_d\in K[X][Y]$ of degree $d$ in $Y$ with $\upsilon_g(q_d)>\upsilon_g(p_d)$.
To get a contradiction, it suffices to show that $q'$ has a monic factor $p' \in K[X][Y]$ of degree
$\deg(p)$ in $Y$. Indeed, Lemma \ref{lem gbar max val} then gives $\upsilon_g(p) \ge \upsilon_g(p')$
and, thus,  $\upsilon_g(p^{u'+1} (q'/p')) \ge \upsilon_g(p_d')$,  a contradiction to the maximality assumption on $u'$.

To show that a factor $p'$ of $q'$ as desired exists, let $q' = q_1 \cdots q_s$ be the decomposition of $q'$ into
irreducible factors $q_i\in K[X][Y]$, all monic in $Y$, say of degrees $d_i$, where we suppose that  $d_1 \le d_2 \le \ldots \le d_s$.
Due to our assumption on the Puiseux expansions of $p'_d$, it follows as in the proof of Lemma \ref{lem gbar max val}
that \mbox{$d_i | d_{i+1}$}, for $1 \le i < s$,  and that \mbox{$d_s | \deg(p)$}.
Then \mbox{$d_s | (\deg(p) - d_s)$} and \mbox{$d_j  | (\deg(p) - (d_{j+1} + \ldots + d_{s-1} +d_s ))$}, for $1 \le j < s$.
Hence, if $\deg(p) - (d_{j+1} + \ldots + d_{s-1} +d_s ) > 0$, then $d_j \le \deg(p) - (d_{j+1} + \ldots + d_{s-1} +d_s )$.
Equivalently, $d_{j+1} + \ldots + d_{s-1} +d_s  < \deg(p)$ implies $d_j +d_{j+1} + \ldots + d_{s-1} +d_s \le \deg(p)$.
Therefore, since $d_{1} + \ldots + d_{s-1} +d_s  =\deg(q')\ge \deg(p)$, there exists $1 \le j \le s$ such that $\deg(p) =
d_j +d_{j+1} + \ldots + d_{s-1} +d_s$,  and we may take $p'={\displaystyle\prod\nolimits_{i=j}^s q_i}$.

\emph{Case 2}: Let $\overline{d}>0$. Then the algorithm executes the recursive call in step \ref{step:recursion}. With each
such call, the number $\nu$ of extended characteristic exponents of $g$ decreases by one. The recursion
continues until $\overline{d}$ is zero, so it stops at latest when $\nu$ is 1. Taking case 1 and
Lemma \ref{lem:help-alg6}   below into account,  we may, thus,  assume that the valuation of
the resulting polynomial
$$
q=\TruncatedFactor(\left\{  \rho_{1},\ldots,\rho_{\overline{m}}\right\}  ,\overline{d}).
$$
at $g$ is maximal in degree $\overline d$.  To conclude, we consider two cases.
If $u=0$, then $d=\overline{d}$ and $p_d=q$, so we are done.
Let $u>0$. Since $p_d=p^uq$, the same argument as in case 1 shows that there is a polynomial of type
$p'_d=p^u q'$, with $q' \in K[X][Y]$ monic in $Y$ of degree $\deg(q')=\deg(q)$, and such that the valuation
of $p'_d$ at $g$ is maximal in degree $d$. But then $\upsilon_g(q') \le \upsilon_g(q)$, hence
$\upsilon_g(p_d')=\upsilon_g(p^u q') \le \upsilon_g(p^u q)=\upsilon_g(p_d)$ by Remark
\ref{rem val mult}. \mbox{So we are done again.}
\end{proof}

\begin{lemma}
\label{lem:help-alg6}
With ${\overline{d}}$ as in step \ref{step:dbar} of Algorithm \ref{alg:TruncateGeneral}, let $q \in K[X][Y]$ be a monic
polynomial of degree ${\overline{d}}$ in $Y$ such that $\upsilon_{\overline{g}} (q)$ is the maximal valuation
at $\overline{g}$ in degree $\overline{d}$. Then $\upsilon_{g} (q)$ is
the maximal valuation at $g$ in degree $\overline{d}$.
\end{lemma}

\begin{proof} 
Let $q' \in K[X][Y]$ be any monic polynomial of degree $\overline{d}$ in $Y$. Then no Puiseux
expansion of $\overline{g}$ coincides with the initial part of a Puiseux expansion of $q'$
since $\overline{d} \le \overline{m} - 1$. Hence, taking into account that each expansion
of $\overline{g}$ is the truncation of an expansion of $g$, it easily follows from the valuation
formula in Section \ref{sect:max-expo} that $\upsilon_{\overline{g}}(q') = \upsilon_{g}(q')$
(in particular, $\upsilon_{\overline{g}}(q) = \upsilon_{g}(q)$). We conclude that
$\upsilon_{g}(q') = \upsilon_{\overline{g}}(q') \le \upsilon_{\overline{g}}(q) = \upsilon_{g}(q)$,
as desired.
\end{proof}

\begin{example}
\label{exa:conj-classes}
The polynomial $g$ considered in Example \ref{examplePuiseux} has degree $m=8$ in $Y$.
We apply Algorithm~\ref{alg:TruncateGeneral}  to the set of Puiseux expansions of $g$,
where $d = m-1 = 7$ is chosen maximal.
The singular parts of the Puiseux expansions of $g$ are
\begin{align*}
\delta_{1}  &  =iX^{3/2}+(-1/2i-1/2)X^{7/4}+1/4iX^{2},\\
\delta_{2}  &  =iX^{3/2}+(-1/2i-1/2)X^{7/4}-1/4iX^{2},\\
\delta_{3}  &  =iX^{3/2}+(1/2i+1/2)X^{7/4}+1/4iX^{2},\\
\delta_{4}  &  =iX^{3/2}+(1/2i+1/2)X^{7/4}-1/4iX^{2},\\
\delta_{5}  &  =-iX^{3/2}+(1/2i-1/2)X^{7/4}+1/4iX^{2},\\
\delta_{6}  &  =-iX^{3/2}+(1/2i-1/2)X^{7/4}-1/4iX^{2},\\
\delta_{7}  &  =-iX^{3/2}+(-1/2i+1/2)X^{7/4}+1/4iX^{2},\\
\delta_{8}  &  =-iX^{3/2}+(-1/2i+1/2)X^{7/4}-1/4iX^{2},%
\end{align*}
where $i^{2}=-1$. Truncating the $\delta_{i}$ to degree $<t=t_3=2$,  we obtain
\begin{align*}
\overline{\delta}_{1} &  = \overline{\delta}_{2}=iX^{3/2}%
+(-1/2i-1/2)X^{7/4},\\
\overline{\delta}_{3}  &  = \overline{\delta}_{4}=iX^{3/2}%
+(1/2i+1/2)X^{7/4},\\
\overline{\delta}_{5} &  = \overline{\delta}_{6}=-iX^{3/2}%
+(1/2i-1/2)X^{7/4},\\
\overline{\delta}_{7}  &  = \overline{\delta}_{8}=-iX^{3/2}%
+(-1/2i+1/2)X^{7/4}.%
\end{align*}
Hence, in step \ref{step:u} of Algorithm \ref{alg:TruncateGeneral}, we get
$u_{1}=u=1$.  Denoting the polynomial generated in step \ref{step:p} of the
algorithm by $p_{\overline{m}}$, we have
\begin{align*}
p_{4}  &  =(Y-\overline{\delta}_{1})(Y-\overline{\delta}_{3})(Y-\overline{\delta}_{5})(Y-\overline{\delta}_{7})\\
&  =Y^{4}+2X^{3}Y^{2}+2X^{5}Y+X^{6}+1/4X^{7}.
\end{align*}
Applying the whole procedure recursively gives $p_{2}
=Y^{2}+X^{3}$, $u_{2} = 1$ and $p_{1} = Y$, $u_{3} = 1$. Combining all
factors, we get
\[
p_7=p_{4}^{u_{1}}p_{2}^{u_{2}}p_{1}^{u_{3}}=\left(  Y^{4}+2X^{3}%
Y^{2}+2X^{5}Y+X^{6}+\frac{1}{4}X^{7}\right)  (Y^{2}+X^{3})Y.
\]

\end{example}

If $p_{d} = \TruncatedFactor(\Gamma, d)$ is a polynomial returned by Algorithm
\ref{alg:TruncateGeneral}, we may find the exponent of $x$ in the denominator
of the corresponding integral basis element via the formula
\[
o(\Gamma, d) = \upsilon_g(p_d)=\sum_{\eta\in N_d} \upsilon(\gamma- \eta).
\]
Here, $N_d = \{\eta_{1}, \dots\eta_{d}\}$ is the set of Puiseux expansions of $p_{d}$
and $\gamma\in\Gamma$ is any of the Puiseux expansions of $g$ (recall that
the latter expansions are all conjugate).

\begin{example}
Continuing Example \ref{exa:conj-classes}, we now compute all
elements of the integral basis.

We first use Algorithm \ref{alg:TruncateGeneral} to find all numerators $p_d$. We already
know that $p_{7} = p_{4} p_{2} p_{1}$. Computing the $p_d$ for $1 \le d \le6$ then amounts
to the following: Start with the largest possible power of $p_{4}$ whose degree in $Y$ is
$\leq d$. Then successively address powers of $p_{2}$ and $p_{1}$ in the obvious way.  This
yields $p_{6} = p_{4} p_{2}$, $p_{5} = p_{4} p_{1}$, $p_4$, $p_{3} = p_{2}p_{1}$, $p_{2}$, and $p_{1}$.

We now compute the powers of $x$ in the denominators via the formula for the $o(\Gamma, d)$
above. By construction,  for any $\gamma\in\Gamma$, we have $\sum_{\eta\in N_{p_{4}}} \upsilon(\gamma- \eta) =
27/4$, $\sum_{\eta\in N_{p_{2}}} \upsilon(\gamma- \eta) = 13/4$, and $\sum_{\eta\in
N_{p_{1}}} \upsilon(\gamma- \eta) = 3/2$. We conclude that
$o(\Gamma, 1) = 3/2$, $o(\Gamma, 2) = 13/4$, $o(\Gamma, 3) = 13/4 +
3/2 = 19/4$, $o(\Gamma, 4) = 27/4$, $o(\Gamma, 5) = 27/4 + 3/2 = 33/4$,
$o(\Gamma, 6) = 27/4 + 13/4 = 10$ and $o(\Gamma, 7) = 27/4 + 13/4 + 3/2
= 23/2$. Hence, the desired integral basis for $g$ is
\[
\left\{ 1=p_0,  \frac{p_{1}}{x}, \frac{p_{2}}{x^{3}}, \frac{p_{2} p_{1}}{x^{4}},
\frac{p_{4}}{x^{6}}, \frac{p_{4} p_{1}}{x^{8}}, \frac{p_{4} p_{2}}{x^{10}},
\frac{p_{4} p_{2} p_{1}}{x^{11}}\right\}  .
\]

\end{example}

\subsection{Merging the Integral Bases for the Branches}

We now know how to find integral bases for the branches of $f$. The next step is to
combine the individual bases to an integral basis for the product $\widetilde{f}
= f_1 \cdots f_r$ of the branches. In principle, this could be done by applying
the splitting of normalization as in Proposition \ref{coro complete split}.
However, as already discussed in Section \ref{section:summary},
this would involve the use of the extended Euclidean algorithm for finding
the respective B\'ezout coefficients and is not very practical.
Our next result, which extends Proposition \ref{coro complete split}, provides
a different strategy, replacing the B\'ezout coefficients
by polynomials in $K[X, Y]$ which are both simpler and easier to calculate.

\begin{proposition}
\label{prop local int basis modified} Write $f=f_{0} \widetilde{f} =f_{0} f_{1}\cdots f_{r}$ as before.
For $i=1,\ldots,r$, set $h_{i}=\prod_{j=1\text{, }j\neq i}^{r}f_{j}$, and let
\[
{\mathcal {B}}^{(i)}=\left\{  1=p_{0}^{(i)},\frac{p_{1}^{(i)}}{X^{e_{1}^{(i)}}}%
,\ldots,\frac{p_{m_{i}-1}^{(i)}}{X^{e_{m_{i}-1}^{(i)}}}\right\}
\]
represent an integral basis for $f_{i}$ as in Proposition \ref{prop loc int bas shape}. Furthermore,
for each $i$, let $\beta_{i}\in K[X,Y]$ be a polynomial such that $\upsilon_{f_i}(\beta_{i}h_{i})$ is
an integer $c_i \ge 0$, and set
\[
\widetilde{\mathcal {B}}^{(i)}=\left\{  \frac{\beta_{i}h_{i}}{X^{c_{i}}},\frac{\beta_{i}h_{i}%
p_{1}^{(i)}}{X^{c_{i}+e_{1}^{(i)}}},\ldots,\frac{\beta_{i}h_{i}p_{m_{i}%
-1}^{(i)}}{X^{c_{i}+e_{m_{i}-1}^{(i)}}}\right\}.
\]
Then $\widetilde{\mathcal{B}}^{(1)}\cup\ldots\cup \widetilde{\mathcal{B}}^{(r)}$ represents an integral basis
for $\widetilde{f} =f_{1}\cdots f_{r}$.
\end{proposition}

\begin{proof}  For each integer $i$ with $1\leq i\leq r$, $f_{i}$ and $\beta_{i} h_{i}$
are coprime in $K((X))[Y]$ since otherwise we would have $\upsilon_{f_i}(\beta_{i}h_{i})=\infty$.
Hence, there are polynomials $a_i, b_i \in K[[X]][Y]$ and an integer $\widetilde{e}_{i}\in\mathbb N$
which fit into a B\'ezout identity  of type $a_i f_{i} + b_i\beta_{i} h_{i} = X^{\widetilde{e}_{i}}$.
Then $\widetilde{e}_{i} = {e}_{i} + c_{i}$, where ${e}_{i} := \upsilon_{f_i}(b_{i})\in\N$.
We conclude from Theorem \ref{theorem:integrality local} that both  $g_i:=b_i/ X^{{e}_{i}}$ and $\beta_{i} h_{i} / X^{c_{i}}$
represent elements which are integral over $K[[X]][Y]/\langle f_{i}\rangle$. Moreover, $g_{i}\frac{\beta_{i}h_{i}}{X^{c_{i}}}\equiv
\delta_{ij} \!\mod f_{j}$, for $1 \le j \le r$.
Hence, as in Proposition \ref{prop split}, the well-defined map of $K[[X]]$-modules
\[
(t_{1}\!\!\!\mod f_{1},\dots,t_{r}\!\!\!\mod f_{r})\mapsto\sum_{i=1}^{r}%
g_{i}\frac{\beta_{i}h_{i}}{X^{c_{i}}}t_{i}\!\!\!\mod f_{1} \dots f_{r}
\]
maps $\bigoplus_{i=1}^{r}\overline{K[[X]][Y]/\!\left\langle f_{i}\right\rangle
}$ isomorphically onto $\overline{K[[X]][Y]/\!\left\langle f_{1}\cdots
f_{r}\right\rangle }$. So if we set
\[
\mathcal{C}^{(i)}=\left\{  g_{i}\frac{\beta_{i}h_{i}}{X^{c_{i}}},g_{i}\frac{\beta
_{i}h_{i}p_{1}^{(i)}}{X^{c_{i}+e_{1}^{(i)}}},\ldots,g_{i}\frac{\beta_{i}%
h_{i}p_{m_{i}-1}^{(i)}}{X^{c_{i}+e_{m_{i}-1}^{(i)}}}\right\}
\]
for $1 \le i \le r$, then $\mathcal{C}^{(1)}\cup\ldots\cup \mathcal{C}^{(r)}$ represents an integral basis
for $f_{1}\cdots f_{r}$.

Now, for each $i$, it is easy to see that $\mathcal{C}^{(i)}$ and $\mathcal{B}^{(i)}$, as well as
$\mathcal {B}^{(i)}$ and $\widetilde{\mathcal {B}}^{(i)}$, represent the same $K[[x]]$-submodule
of $K[[x]][y]=K[[X]][Y]/\langle f_i \rangle$: Use that $g_{i}\frac{\beta_{i}h_{i}}{X^{c_{i}}}
\equiv 1 \!\mod f_{i}$ and that $g_i(x,y)$ is integral over $K[[x]][y]$.
Since in addition $\mathcal{C}^{(i)}$ and $\widetilde{\mathcal {B}}^{(i)}$ reduce to zero modulo $f_j$
for $j\neq i$, we see that $\mathcal{C}^{(i)}$ and $\widetilde{\mathcal {B}}^{(i)}$
represent the same $K[[x]]$-submodule of $K[[x]][y]=K[[X]][Y]/\langle
f_1\cdots f_r \rangle$, which concludes the proof.
\end{proof}

To find pairs $(\beta_i, c_i)$ as in the proposition, it is enough to compute
the singular parts of the Puiseux expansions of $f$. In fact,
if $\upsilon_{f_i}(h_{i})\in\mathbb N$, set $\beta_{i} =1$. Otherwise,
set $\overline{h}_{i} = \prod_{\delta\in\Delta^{(i)}}(Y-\delta)$, where $\Delta^{(i)}$ is the set
of singular parts of the Puiseux expansions of $h_{i}$. Then an appropriate power
of $\overline{h}_{i}$ will do. Indeed, $\upsilon_{f_i}({h}_{i} ^{\ell}) = \ell\cdot
\upsilon_{f_i} ({h}_{i} )$ and $\upsilon_{f_i}({h}_{i})=\upsilon_{f_i}(\overline{h}_{i})$.
Typically, however, it is more efficient to  apply Algorithm \ref{alg:mergecoef}
below. Examples \ref{exampleTwoBranches-2} and \ref{example:lin-congr-equ}
show the algorithm at work.

\begin{algorithm}[t]                      
\caption{\texttt{CoefficientsForMerging}}          
\label{alg:mergecoef}
\begin{algorithmic}[1]

\REQUIRE $\Delta_{1}, \dots, \Delta_{r}$, the sets of singular parts of the Puiseux expansions
of the branches $f_{1}, \dots, f_{r}$ of an irreducible polynomial $f \in K[X][Y]$ which is monic in $Y$.
\ENSURE  A set $\{(\beta_{i}, c_i)\}_{1 \le i \le r}$ of pairs $(\beta_{i}, c_i)\in K[X,Y]\times \mathbb N$
with $\upsilon_{f_i}(\beta_{i}h_{i})=c_i$, where  $h_i = \prod_{j \neq i} f_j$.
\FOR{$i = 1, \dots, r$}
\STATE from the given singular parts, compute $\upsilon_{f_i}(h_{i})$
\IF{$\upsilon_{f_i}(h_{i}) \in \mathbb N$}
\STATE $\beta_{i} = 1$, $c_i = \upsilon_{f_i}(h_{i})$
\ELSE
\STATE
\label{alg7:step5}
for $1 \le j \le r$, $j \ne i$, and $1 \le k < \deg_Y(f_{j})$, set
$
f_{j,k} = \TruncatedFactor(\Delta_{j}, k),
$
\\
$
f_{j, \deg_Y(f_{j})} = \prod_{\delta\in\Delta_{j}} (Y-\delta)$
\STATE
\label{alg7:step7}
for each prime divisor $a$ of the denominator of $\upsilon_{f_i}(h_{i})$, choose a
polynomial\footnotemark \ from the $f_{j,k}$ whose valuation at ${f_i}$ has a multiple of $a$ as its
denominator, and whose $Y$-degree is minimal among  the $f_{j,k}$ with this
property
\STATE
set up a linear congruence equation to find for each $f_{j,k}$ selected in step \ref{alg7:step7}
an exponent $\ell_{j,k}$ such that the product $\beta_{i}$ of the powers $f_{j,k}^{\ell_{j,k}}$
satisfies $\upsilon_{f_i}(\beta_{i})  \in \mathbb N$; choose a solution of the linear congruence
equation which minimizes the $Y$-degree of $\beta_i$
\STATE $c_i = \upsilon_{f_i}(\beta_i h_i)$
\ENDIF
\ENDFOR
\RETURN $\{(\beta_{i}, c_i)\}_{1 \le i \le r}$
\end{algorithmic}
\end{algorithm}

\footnotetext{Polynomials $f_{j,k}$ as desired exist since the polynomial
$\overline{h}_{i}=\prod_{\delta\in\Delta^{(i)}}(Y-\delta)$
considered in the text is a factor of the product of all the $f_{j,k}$.}

\begin{example}
\label{exampleTwoBranches-2} Let $f = (Y^{3}+X^{2})(Y^{2}-X^{3})+Y^{6}\in\mathbb Q[X,Y]$ be as in Examples
\ref{exampleTwoBranches-int} and \ref{exampleTwoBranches}, with branches
$f_{1} = (Y -\gamma_{1})(Y - \gamma_{2})(Y - \gamma_{3})$ and $f_{2} = (Y - \gamma_{4})(Y -\gamma_{5})$.
We apply Algorithm~\ref{alg:mergecoef} to compute the polynomial $\beta_1$ and the set
$\widetilde{\mathcal{B}}^{(1)}$ from Proposition \ref{prop local int basis modified}.

In Example \ref{exampleTwoBranches}, we found
the integral basis $\mathcal{B}^{(1)}=\{1,y,\frac{y^{2}}{x}\}$ for $f_{1}$. Considering $h_1=f_2$, we
see from the initial terms of the $\gamma_i$ as written in Example \ref{exampleTwoBranches-int}
that $\upsilon_{f_{1}}(h_{1}) = 4/3$. In step \ref{alg7:step5} of the algorithm, we get
\[
f_{2,1}=\TruncatedFactor(h_{1}, 1) = Y.
\]
Since $\upsilon_{f_{1}}(Y) = 2/3$, we have
$\upsilon_{f_{1}}(Yf_1) = 2\in\mathbb N$. So we can take $\beta_{1} = Y$ and, thus,
\[
\widetilde{\mathcal{B}}^{(1)} = \left\{ \frac{Y h_{1}}{X^{2}}, \frac{Y h_{1}
Y}{X^{2}}, \frac{Y h_{1} Y^{2}}{X^{3}} \right\}.
\]
Note that this set is simpler than the set $\widetilde{\mathcal{B}}^{(1)} $ obtained in Example \ref{exampleTwoBranches}.
\end{example}

Here is a slightly more complicated example:

\begin{example}
\label{example:lin-congr-equ}
Let $f(X, Y) = (Y^{6} - 6X^{3}Y^{4} - 2X^{7}Y^{3} + 12X^{6}Y^{2} - 12X^{10}Y -
8X^{9})(Y^{2}-2YX^{3}-2X^{3})(Y^{2}+X^{7}) + X^{30}\in\mathbb Q[X,Y]$. The Puiseux expansions of $f$ are
\begin{equation*}
 \begin{aligned}[t]
\gamma_1 &= a_1 X^{3/2} + X^{7/3} + \dots, \\
\gamma_2 &= a_1 X^{3/2} + b_1 X^{7/3} + \dots, \\
\gamma_3 &= a_1 X^{3/2} + b_2 X^{7/3} + \dots, \\
\gamma_4 &= a_2 X^{3/2} + X^{7/3} + \dots, \\
\gamma_5 &= a_2 X^{3/2} + b_1 X^{7/3} + \dots,  \\
\end{aligned}
\quad\quad
\begin{aligned}[t]
\gamma_6 &= a_2 X^{3/2} + b_2 X^{7/3} + \dots, \\
\gamma_7 &= a_1 X^{3/2} + X^3 + \dots, \\
\gamma_8 &= a_2 X^{3/2} + X^3 + \dots, \\
\gamma_9 &= c_1 X^{7/2} + \dots, \\
\gamma_{10} &= c_2 X^{7/2} + \dots,
\end{aligned}
\end{equation*}

\noindent with roots $a_{1}, a_{2}$ of $Z^{2} - 2$, $b_{1}, b_{2}$ of $Z^{2} + Z + 1$, and $c_{1}, c_{2}$ of $Z^{2} + 1$.
Corresponding to the conjugacy classes $\Delta_{1} = \{\gamma_{1}, \gamma_{2}\, \gamma_{3}, \gamma_{4},
\gamma_{5}, \gamma_{6}\}$, $\Delta_{2} = \{\gamma_{7}, \gamma_{8}\}$, and $\Delta_{3} =\{\gamma_{9}, \gamma_{10}\}$,
there are three branches $f_{1},f_{2}$, and $f_{3}$ of $f$. We show how to compute the polynomial $\beta_{1}$
from Proposition \ref{prop local int basis modified}.

Let $h_{1} = f_{2} f_{3}$. Then $\upsilon_{f_{1}}(h_{1})=41/6$. In step \ref{alg7:step5} of Algorithm \ref{alg:mergecoef},
we get $f_{2,1}= Y$, $f_{2,2}= Y^{2}-2YX^{3}-2X^{3}$, and $f_{3,1}= Y$, $f_{3,2}= Y^{2}+X^{7}$. Furthermore, we have
$\upsilon_{f_{1}}(y) = 3/2$, $\upsilon_{f_{1}}(y^{2}-2yx^{3}-2x^{3}) = 23/6$, and $\upsilon_{f_{1}}(y^{2}+x^{7}) = 3$.
So we can take a polynomial of type $\beta_{1} = y^{\ell_{1}}(y^{2} -2yx^{3} - 2x^{3})^{\ell_{2}}$,  with exponents
$\ell_{1}, \ell_{2} \in{\mathbb{N}}$ such that $\ell_{1} \frac{3}{2} + \ell_{2} \frac{23}{6} + \frac{41}{6} \in{\mathbb{Z}}$.
The corresponding linear congruence equation is $9 \ell_1 + 23 \ell_2 + 41 \equiv 0 \mod 6$. Choosing the solution
which minimizes the $Y$-degree of $\beta_1$, we get $\ell_{1} = 1$, $\ell_{2} = 2$ and, thus,
$\beta_{1} = y (y^{2}-2yx^{3}-2x^{3})^{2}$.
\end{example}
Merging the integral bases for the branches using Proposition \ref{prop local int basis modified}
requires that we know the precision $t$ up to which all power series in $X$ appearing in the
process must be developed. Taking Remark \ref{ypowers} into account, we can use $t=e_c + E(f)$,
where $e_c = \max_{1 \le i \le r} c_i$. Indeed, the integrality exponent of any polynomial appearing
in the construction will be at most this number. We obtain Algorithm \ref{alg:loclocbasis} below.

\begin{algorithm}                      
\caption{\texttt{MergingIntegralBases}}        
\label{alg:loclocbasis}
\begin{algorithmic}[1]
\REQUIRE  Lists of pairs $\{(\Delta_i, {f}_i)\}_{1 \le i \le r}$, $\{(\beta_{i}, c_i)\}_{1 \le i \le r}$,
where
\begin{itemize}
\item $\Delta_{1}, \dots, \Delta_{r}$ are the sets of singular parts of the Puiseux expansions
of the branches $f_{1}, \dots, f_{r}$ of an irreducible polynomial $f \in K[X][Y]$ which is monic in $Y$,
\item \{$(\beta_{i}, c_i)\}_{1 \le i \le r}$, the output of $\MergeCoefficients(\Delta_{1}, \dots, \Delta_{r})$
\item the $f_i$ are developed up to $X$-degree $e_c + E(f)$, where $e_c = \max_{1 \le i \le r} c_i$.
\end{itemize}
\ENSURE A list $\{(p_{i}, e_i)\}_{1 \le i \le r}$ of pairs $(p_i, e_i)\in K[X,Y]\times \mathbb N$  such that
$\left\{1=p_0, \frac{p_{1}}{x^{e_{1}}}\dots, \frac{p_{m-1}}{x^{e_{m-1}}}\right\}$ is an integral basis for $\widetilde{f}=f_1 \cdots f_r$
of monic triangular type.
\STATE $m = \deg_Y(f_1 \cdots f_r)$
\FOR{$i = 1, \dots, r$}
\STATE $h_i = \prod_{j \neq i} f_j$
\FOR{$d = 0, \dots, m_i-1$}
\STATE $q_d = \TruncatedFactor(\Delta_i, d)$
\STATE $p_{d} = b_i h_i q_d$
\STATE $e_{d} = c_i + e_{f_i}(q_d)$
\ENDFOR
\STATE $\mathcal B^{(i)} = \left\{ \frac{p_{0}}{x^{e_{0}}}, \frac{p_{1}}{x^{e_{1}}},
\dots, \frac{p_{m_i-1}}{x^{e_{m_i-1}}}\right\}$
\ENDFOR
\STATE \label{step int bas} applying the recipe from Remark \ref{rem: HN-mts} to
$\mathcal{B}^
{(1)} \cup \ldots \cup \mathcal{B}^{(r)}$, compute an integral basis
$\left\{1=p_0, \frac{p_{1}}{x^{e_{1}}}\dots, \frac{p_{m-1}}{x^{e_{m-1}}}\right\}$
of monic triangular type for $f_1 \cdots f_r$
\RETURN $\{(p_{i}, e_i)\}_{1 \le i \le r}$
\end{algorithmic}
\end{algorithm}

\begin{remark}
Note that if $e=\upsilon_{\widetilde{f}}\;\!(y)$, then the fractions $y^{i} / x^{\lfloor ei \rfloor}$, $0 \le i < m = \deg_{Y}(\widetilde{f})$,
are integral over $K[[x]][y] = K[[X]][Y]/\langle \widetilde{f}\rangle$.
To enhance the performance of Algorithm \ref{alg:loclocbasis}, we propose to add these elements to
$\mathcal{B}^{(1)} \cup \dots \cup \mathcal{B}^{(r)}$ before applying Remark \ref{rem: HN-mts}.
\end{remark}

\subsection{Ad-hoc Formulas for Absolutely Irreducible Weierstrass Polynomials}
\label{section:adhoc}
It turns out that in the case where $g\in K[[X]][Y]$ is an absolutely irreducible Weierstrass
polynomial, we can get simple formulas for computing valuations and, thus,
maximal integrality exponents. In particular, if $g$ is absolutely irreducible and
there is only one characteristic exponent, then these formulas allow us
to write down an integral basis for $g$ directly, without using Algorithm
\ref{alg:TruncateGeneral}, $\TruncatedFactor$. We need:

\begin{lemma}
\label{lemma:valuation-special-case}
Let $g\in K[[X]][Y]$ be an absolutely irreducible Weierstrass polynomial of degree
$m$ in $Y$. Factorize $g$ as
$$
g=\prod (Y-\gamma_{\ell}(X)) := \prod (Y-\eta (\omega^\ell X^{1/m}))
$$
as in Section \ref{subsect: Puiseux-expansions}, where $\eta\in L[[T]]$, and $\omega$
is a primitive $m$th root of unity. Set $\gamma=\gamma_m$. As in Remark~\ref{rem characteristic orders},
assume that $\upsilon(\gamma)>1$, and let $k_{0}, \ldots,k_{\nu}$ be the characteristic exponents of $g$.
Then, with notation as in Section \ref{sect:max-expo}, we have:

\begin{enumerate}
\item For $j=1,\ldots,\nu$, denote by $N_{j}$ the set of all $i\in\{1,\ldots,m-1\}$
such that%
\[
\frac{k_{0}}{\gcd(k_{0},\ldots,k_{j-1})}\mid i\text{\hspace{5mm}%
and\hspace{5mm}}\frac{k_{0}}{\gcd(k_{0},\ldots,k_{j})}\nmid i\text{.}%
\]
Then%
\[
\upsilon(\gamma-\gamma_i)=\frac{k_{j}}{m} \ \text{ for all } \ i \in N_{j}.%
\]
In particular, if $\nu=1$, then for $i \neq m$,%
\[
\upsilon(\gamma-\gamma_i)=\frac{k_{1}}{m}.%
\]
\item For all $i$, we have
\begin{align*}
\upsilon_\gamma\left(\frac{\partial g}{\partial Y}\right) &  =\sum
_{j=1}^{\nu}\left(  \gcd(k_{0},\ldots,k_{j-1})-\gcd(k_{0},\ldots,k_{j})\right)
\frac{k_{j}}{m}\\
&  =\Int_i 
\end{align*}
\end{enumerate}
\end{lemma}
\begin{proof}
See \citet[Lemma 5.2.18(1)]{JP} and its proof.
\end{proof}

\begin{proposition}
\label{prop integral} With notation and assumptions as in Lemma \ref{lemma:valuation-special-case}, set
$$
e_d=\left\lfloor \upsilon_\gamma\left(\frac{\partial^{m-d}g}{\partial Y^{m-d}%
}\right)\right\rfloor, \ \text{ for } \ d=1,\dots, m-1.
$$
Then we have:
\begin{enumerate}
[leftmargin=0.9cm]

\item
\label{step integral e m-1}
The element $\frac{\frac{\partial g}{\partial Y}%
}{X^{e_{m-1}}}$ is integral over $K[[X]][Y]/\langle g \rangle$, and
$e_{m-1}$ is the maximal integrality exponent with respect
to $g$ in degree $m-1$.

\item
\label{step integral e 1}
The element $\frac{\frac{\partial^{m-1}g}{\partial Y^{m-1}}}{X^{e_1}}$ is integral over
$K[[X]][Y]/\langle g \rangle$, $e_1=\left\lfloor \frac{k_{1}}{m}\right\rfloor $, and
this number is the maximal integrality exponent with respect to $g$ in degree $1$.

\item
\label{step integral nu 1}
If $\nu=1$, then $e_d = \left\lfloor \frac{dk_1}{m}\right\rfloor$, $1 \le d \le m-1$, and
\[
\left\{
1,\frac{\frac{\partial^{m-1}g}{\partial Y^{m-1}}}{X^{e_{1}}},\ldots
,\frac{\frac{\partial g}{\partial Y}}{X^{e_{m-1}}}%
\right\}
\]
represents an integral basis for $g$.
\end{enumerate}
\end{proposition}

\begin{proof}
As in Lemma \ref{lemma:intA},  for each $1 \le d \le m-1$, choose a subset
$\widetilde{\mathcal{A}}\subseteq\{1,\ldots, m-1\}$ of cardinality $d$  with
$\Int(\widetilde{\mathcal{A}})$ maximal among all $\Int({\mathcal{A}})$,
${\mathcal{A}}\subseteq\{1,\ldots, m-1\}$ of cardinality $d$, and set
\[
\widetilde{p}_d:=%
{\displaystyle\prod\limits_{j\in\widetilde{\mathcal{A}}}}
\left(  Y-\gamma_j(X)\right) \in{\mathcal{P}_{X}}[Y].
\]
Then $\upsilon_{g}(\widetilde{p}_d)=\Int({\widetilde{\mathcal{A}}})$, and this
number is the maximal valuation $\upsilon_{g}(q)$, for $q\in L\{\{X\}\}[Y]$
monic of degree $d$ in $Y$.
Hence, items (\ref{step integral e m-1}), (\ref{step integral e 1}) and the first part of (\ref{step integral nu 1}) follow from the previous lemma,
computing the derivatives of $g=\prod_{i=1}^m (Y-\gamma_{i}(X))$ by the product rule and evaluating the result at $Y = \gamma(X)$.
The second part of (\ref{step integral nu 1}) follows then from Proposition \ref{prop loc int bas shape}.
\end{proof}

\begin{example}
The Weierstrass polynomial $g=Y^{4}-2X^{3}Y^{2}-4X^{11}Y+X^{6}-X^{19} \in \CC[[X]][Y]$ is (absolutely)
irreducible. It factors as
$$
g= \prod (Y-\eta (\omega^\ell X^{1/4})),
$$
where $\eta(T)=T^{6}+T^{19}\in \CC[[T]]$, and $\omega=e^{\pi i/2}$. We compute that
$\upsilon_{g}\left(\frac{\partial^{2}g}{\partial Y^{2}}\right)=3$, but for
\[
\widetilde{p}_2=\left(  Y-\eta(X^{1/4})\right)  \left(  Y-\eta(i X^{1/4})\right),
\]
we have $\upsilon_{g}(\widetilde{p}_2) =\frac{25}{4}$. So here $\nu=2$, and
the  assertion of Proposition \ref{prop integral}, $(3)$ does not hold.
\end{example}

\subsection{Computing  Minimal Local Contributions}
\label{subsect:local-contr}

Summing up, we now present our algorithm for finding minimal local contributions.
For this, as we already know, we may restrict ourselves to the case of a $K$-rational
singularity, which may be chosen to be the origin. That is, we consider the prime
ideal $P=\langle x,y\rangle \in\Sing(A)$. If we assume in addition
that the origin is the only singularity at $X = 0$, we may apply  the recipe given
in Proposition \ref{completion basis}, which allows us to compute the minimal
contribution from an integral basis of monic triangular type for $f_1 \cdots f_r$
(take Proposition \ref{prop completion to localization} into account). We describe
the resulting procedure in Algorithm \ref{alg:HenselBasis}.

\begin{algorithm}                      
\caption{\texttt{MinimalLocalContribution}}        
\label{alg:HenselBasis}
\begin{algorithmic}[1]
\REQUIRE $f \in K[X][Y]$ irreducible and monic in $Y$ of degree $n$, with
$P = \langle X, Y \rangle\in\Sing(A)$, where $ A=K[x,y] =K[X,Y]/\langle f(X,Y)\rangle$,
and such that $P$ is the only singularity of $A$ at $X=0$.
\ENSURE A set of $K[x]$-module generators for the
minimal local contribution to $\overline{A}$ at $P$.
\STATE $\Delta_0 = $  set of singular parts of the Puiseux expansions of $f$ not vanishing at the origin,
            $\Delta_{1}, \dots, \Delta_{r} = $ sets of singular parts of the Puiseux expansions
            of $f$ vanishing at the origin, grouped into conjugacy classes.
\STATE compute the maximal integrality exponent $E(f)$ as indicated in Section \ref{sect:max-expo}
\STATE  \{$(\beta_{i}, c_i)\}_{1 \le i \le r} = \MergeCoefficients(\Delta_{1}, \dots, \Delta_{r})$
\STATE $e_c = \max_{1 \le i \le r} c_i$
\STATE \{$f_0, f_1, \dots, f_r\} = \Splitting(E(f)+e_c, f)$, where $f_0$ corresponds to $\Delta_0$ and
            $f_1, \dots, f_r$ correspond to $\Delta_1, \dots, \Delta_r$
\STATE $m_0 = \deg(f_0)$, $m = n - m_0$
\STATE $\{(p'_{i}, e'_i)\}_{1 \le i \le r}= \IntegralBasisForWeierstrassPolynomial(\{(\Delta_i, {f}_i)\}_{1 \le i \le r}, \{(\beta_{i}, c_i)\}_{1 \le i \le r}\})$
\FOR{$i = 0, \dots, m_0-1$}
\STATE $p_i = y^i$, $e_i = 0$
\ENDFOR
\FOR{$i = 0, \dots, m-1$}
\STATE $p_{m_0 + i} = f_0 \cdot p'_i$, $e_{m_0 + i} = e'_i$
\ENDFOR
\RETURN $\left\{1=p_0, \frac{p_{1}}{x^{e_{1}}}\dots, \frac{p_{m-1}}{x^{e_{m-1}}}\right\}$
\end{algorithmic}
\end{algorithm}

\begin{remark}
\label{rmk conjugated}In the presence of conjugate singularities, we get a
better performance by handling groups of conjugate singularities simultaneously
(see also \citet[Section 2.3]{vanHoeij94}). Let $P\in\Sing(A)$ correspond to such
a group of singularities, where we  assume that no two of the singularities have
the same $X$-coordinate. Then we can find polynomials $q_{1},q_{2}\in K[X]$ such
that $P=\langle q_{1}(X),Y-q_{2}(X)\rangle$. We take $\alpha$ to be a root of
$q_{1}(X)$ and translate the singularity $(\alpha,q_{2}(\alpha))$ to the
origin. We compute the local contribution to the integral basis at the origin and
apply the inverse translation to the output. The least common denominator
of the resulting generators will be a power of $x-\alpha$. We rewrite all
generators using this denominator. Then we regard  $\alpha$ as a variable
and eliminate it from all numerators by successively  reducing each numerator
with respect to the numerators of smaller degree, using an elimination
ordering for which $\alpha > y > x$ (this works since we can always find
an integral basis which is defined over the original base field). Finally, we
replace $(x-\alpha)$ by $q_{1}(x)$ in all denominators.
\end{remark}

\begin{example}
\label{example conjugated singularities}
Let $f(X, Y) = Y^{3} - (X^{2}-2)^{2}\in\Q[X,Y]$. Then $\Sing(A)$ consists of the single prime
ideal $P = \langle X^{2} - 2, Y\rangle$, that is, the conjugate points $(\sqrt{2}, 0)$ and
$(-\sqrt{2}, 0)$ are the only singularities. Taking $\alpha = \sqrt{2}$, the recipe above
yields  the set $\left\{  1, y, \frac{y^{2}}{x-\alpha}\right\}$  of $K[x]$-module generators
for the minimal local contribution at $(\alpha,0)$. Hence,
$\left\{  1, y, \frac{y^{2}}{x^{2}-2}\right\}$ is an integral basis
for $\overline{A}$ over $K[x]$. So in this simple
case, we did not need to eliminate $\alpha$ from the numerators.
\end{example}

\begin{example}
\label{example conjugated singularities with reduction}
Let $f(X, Y) = (Y-X)^{3} - (X^{2}-2)^{2}\in\Q[X,Y]$. Now the singular locus
consists of the single prime ideal $P = \langle X^{2} - 2, Y-X \rangle$,
that is, the conjugate points $(-\sqrt{2}, -\sqrt{2})$ and $(\sqrt{2},\sqrt{2})$ are the
only singularities. Taking $\alpha= \sqrt{2}$ and computing the minimal
local contribution at $(\alpha, \alpha)$, we get
\[
\left\{  1, y, \frac{y^{2} - 2\alpha y + 2}{x-\alpha}\right\}  = \left\{  \frac{x-\alpha}{x-\alpha}, \frac{y(x-\alpha
)}{x-\alpha}, \frac{y^{2} - 2\alpha y + 2}{x-\alpha}\right\}.
\]
Reducing $y^{2} - 2\alpha y + 2$ with respect to $x-\alpha$ and $y(x-\alpha)$ as described above,
we get $y^{2} - 2x y + 2$ and, thus, the new set of $K[x]$-module generators $\left\{  1, y, \frac{y^{2} - 2x y +
2}{x-\alpha}\right\}  $. So in this example, the final result of our algorithm is
$\left\{  1, y, \frac{y^{2} - 2x y + 2}{x^{2}-2}\right\}$.
\end{example}

\begin{remark}
\label{remark rotation}
If there are two singularities with the same $X$-coordinate, we may apply a linear change of coordinates
of type $X \rightarrow X + a Y$, $a \in K$, to remedy the situation. After the integral basis has been
computed, we apply the inverse transformation $X \rightarrow X - a Y$ to the elements obtained.
This will give us a representation for the normalization of type $\overline{A} = \frac{1}{d_1}U_1$,
where the denominator $d_1$ may not depend on $x$ alone.
By choosing an element $d_2 \in K[x]$ of the conductor of $A$ (consider  the Jacobian ideal of $A$),
and computing the ideal quotient $U_2 = (d_2 U_1) : d_1$, we  arrive at a representation $\overline{A} = \frac{1}{d_2} U_2$
whose denominator does not depend on $y$. From this, we obtain an integral basis for $A$ over $K[x]$
following the recipe given in Remark \ref{rem:spec-int-basis-II}.

\end{remark}
\begin{example}
\label{example rotation}
Let $f(X, Y) = (Y^{2}-2)^{2} + X^{5} = Y^4-4Y^2+X^5+4 \in\Q[X,Y]$. Then $\Sing(A)$ consists of the single prime
ideal $P = \langle Y^{2} - 2, X\rangle$. That is, the conjugate points $(0, \sqrt{2})$ and
$(0, -\sqrt{2})$, which have the same $X$-coordinate,  are the only singularities.

Implementing the recipe from Remark \ref{remark rotation}, we apply the coordinate change $X \rightarrow X+Y$ which yields the
polynomial $g(X,Y) = f(X+Y, Y) = Y^5+5Y^4X+Y^4+10Y^3X^2+10Y^2X^3-4Y^2+5YX^4+X^5+4$ of $Y$-degree 5. The singular
locus of the coordinate ring of the plane curve defined by $g$ consists of the single prime ideal
$Q=\langle y^2-2, x+y\rangle$. We extend the base field from $\Q$ to $\Q(\sqrt{2})$, and consider the prime
ideals $Q_1 = \langle y+\sqrt{2}, x+y\rangle$ and $Q_2 = \langle y-\sqrt{2}, x+y\rangle$.
Then we apply the translation $X \rightarrow X + \sqrt{2}$, $Y \rightarrow Y - \sqrt{2}$ to move the point
$(-\sqrt{2}, \sqrt{2})$ corresponding to $Q_1$ to the origin. This yields the polynomial
$h(X,Y)=g(X + \sqrt{2}, Y - \sqrt{2}) = X^5+5X^4Y+10X^3Y^2+10X^2Y^3+5XY^4+Y^5+Y^4-4\sqrt{2}Y^3+8Y^2$. The decomposition of $h$
given by the Weierstrass preparation theorem is $h = h_0 h_1\in \Q(\sqrt{2})[[X]][Y]$, where the unit  $h_0\in \Q(\sqrt{2})[[X, Y]]$
is a factor of $Y$-degree 3,
and where $h_1$, which has $Y$-degree 2, is the single branch of $h$. The two Puiseux expansions of $h_1$ are
$\gamma_{1,2} = \pm \frac{\sqrt{2}}{4} i X^{5/2} + ...\ $. Hence, by Proposition \ref{prop loc int bas shape},
$\left\{1, \frac{y}{x^2}\right\}$is an integral basis for $h_1$. Incorporating the truncation $\overline{h}_0 =
Y^3+Y^2+5XY^2+(-4\sqrt{2})Y+8$  of $h_0$ to $X$-degree $1$ as in Proposition \ref{completion basis},
we get the local contribution
\[
\left\{1, y, y^2, y^3, \frac{\overline{h}_0(x,y) \cdot y}{x^2}\right\}
\]
to $\overline{\Q(\sqrt{2})[X, Y]/\langle h\rangle}$ at the origin. Translating the singularity back to its original location via
$X \rightarrow X - \sqrt{2}$, $Y \rightarrow Y + \sqrt{2}$ gives us the local contribution
\[
\left\{1, y-\sqrt{2}, (y-\sqrt{2})^2, (y-\sqrt{2})^3, \frac{\overline{h}_0(x-\sqrt{2}, y+\sqrt{2})\cdot (y + \sqrt{2})}{(x-\sqrt{2})^2}\right\}
\]

to $\overline{\Q(\sqrt{2})[X, Y]/\langle g\rangle}$ at $Q_1$, where the first four elements generate the same $\Q(\sqrt{2})[x]$-module as $\{1, y, y^2, y^3\}$.
To get the local contribution to $\overline{\Q[X, Y]/\langle g\rangle}$ at $Q$, we reduce the numerator of the last element
$\overline{h}_0(x-\sqrt{2}, y+\sqrt{2})\cdot (y + \sqrt{2})$ modulo $(x-\sqrt{2})^2$ to eliminate $\sqrt{2}$ from the numerator,
and then replace the denominator $(x-\sqrt{2})^2$ by $(x^2-2)^2$. We obtain the set
\[
\overline{\mathcal{B}}_t = \left\{1, y, y^2, y^3, \frac{q_4}{(x^2-2)^2}\right\},
\]
where
{\footnotesize
$$
q_4=y^4+\left(\frac{1}{4}x^3+\frac{7}{2}x+1\right)y^3+\left (\frac{1}{4}x^3+\frac{15}{2}x^2-\frac{3}{2}x-3\right)y^2+\left (\frac{11}{2}x^3-3x-2\right)y-\frac{1}{2}x^3+5x^2+3x-6.
$$
}

\vskip0.01cm\noindent
Applying the inverse transformation $X \rightarrow X - Y$ to $\overline{\mathcal{B}}_t$ yields the local contribution to  $\overline{A}$
at $P$. In fact, since $P$ is the only  prime ideal in $\Sing(A)$, we get a representation for all of $\overline{A}$. This
is of type  $\overline{A }= U / d_1$, with denominator $d_1=((x-y)^2-2)^2$. To change the denominator to a polynomial which
depends on $x$ alone, we follow the recipe given in Remark \ref{remark rotation}. Inspecting the  Jacobian ideal
$\langle x^4, 4y^3-8y \rangle$ of $A$, we see that $d_2=x^4$ is an element of  the conductor of $A$. Computing
the ideal quotient $U_2 = (d_2 U_1) : d_1$, we  arrive at the representation $\overline{A} = \frac{1}{x^2} \langle y^4-4y^2+4,
x^2y^2-2x^2, x^2 \rangle$. From this, proceeding as  in Remark \ref{rem:spec-int-basis-II}, we get the integral basis
\[
\overline{\mathcal{B}} = \left\{ 1, y, \frac{y^2-2}{x^2}, \frac{y^3-2y}{x^2} \right\}
\]
for $\overline{A}$ over $\Q[x]$.
\end{example}

\begin{remark}
\label{rem no rotation}
As indicated by the last example, applying a coordinate change to separate the $X$-coordinates of the singularities
requires extra computations which may be expensive, in particular in the case where the $X$-degree
of $f$ is considerably larger than its $Y$-degree. We sketch an alternative approach which avoids such
a coordinate change.

If we face more than one singularity on the line $X = \alpha$, we consider the translation $X\rightarrow X+\alpha$
in order to move the singularities to the line $X=0$. We get the polynomial $f_\alpha(X,Y) = f(X+\alpha, Y)$ and the
curve $C_\alpha$ defined by $f_\alpha$. We then decompose $f_\alpha$ as  $f_\alpha= g_0 g_{1}\cdots g_{s}$, where
$g_1, \dots, g_s$ are the irreducible factors of $f_\alpha$ in $K[[X]][Y]$ whose zeros on the line $X=0$ are singular
points of $C_\alpha$, and where $g_0$ is the product of the remaining factors.
Different from our convention so far, we now refer to $g_{1}, \dots, g_{s}$ as the \emph{branches} of $f_\alpha$.
Adapted versions of Propositions \ref{completion basis} and \ref{prop local int basis modified}
allow us to compute an integral basis for $K[[x]][y]=K[[X]][Y] /\langle f_\alpha\rangle$ over $K[[x]]$ in a way similar to that of the
previous discussion. In the case where $\alpha$ is $K$-rational, we may then apply the inverse translation $X\rightarrow X-\alpha$ to get an
integral basis for $K[[x-\alpha]][y]=K[[X-\alpha]][Y] /\langle f_\alpha\rangle$ over $K[[x-\alpha]] $. In the case of conjugate singularities, we may
handle these singularities simultaneously following a recipe similar to that of Remark \ref{rmk conjugated}.
\end{remark}

\begin{example}
\label{ex no rotation}
The polynomial $f(X, Y) = (Y^{2}-2)^{2} + X^{5} = Y^4-4Y^2+4+X^5 \in\Q[X,Y]$ from Example \ref{example rotation}
has the conjugate points $(0, \sqrt{2})$ and $(0, -\sqrt{2})$ as its only singularities.

We show how to handle
the singularities simultaneously, without applying a coordinate change first. The four Puiseux expansions
of $f$ are $\gamma_{1,2,3,4} = a \pm \frac{1}{4} a X^{5/2} + \frac{1}{32} a X^5 + \dots$, where $a$ is a root of $Z^2-2$.
Since the expansions are conjugate over $\Q((X))$, we see that $f$ is irreducible in $\Q[[X]][Y]$. That is, $f$
consists of precisely one branch.

We construct the integral basis for the branch by truncating the Pusieux expansions as in Algorithm \ref{alg:TruncateGeneral}.
We get the numerators $p_0 = 1$, $p_1 = y$, $p_2 = (y-\sqrt{2})(y+\sqrt{2})=y^2-2$, with $\upsilon_f(p_2)= 2$, and $p_3 = (y^2-2)y$,
with $\upsilon_f(p_3)= 2$. Since there is only one branch, we conclude as in Example \ref{example rotation} that the resulting
integral basis
\[
\overline{\mathcal{B}} = \left\{1, y, \frac{y^2-2}{x^2}, \frac{y^3-2y}{x^2}\right\}
\]
for the branch is already an integral basis for $\overline{A}$ over $\Q[x]$.
\end{example}

\section{Timings}

\label{sec timings}

We present timings to compare the computation of integral bases via
\begin{itemize}
\item the \textsc{Singular}{} implementation of our integral basis
algorithm\footnote{Column \textsc{Singular intbas} in the tables;},
\item  the \textsc{Singular}{} implementation of the local
normalization algorithm\footnote{column \textsc{Singular normal} in the
tables;} from Section \ref{sect:noem-via-loc},
\item the \textsc{Maple}{} \citep{maple}  implementation
of van Hoeij's algorithm\footnote{column \textsc{Maple} in the tables;},
and
\item the \textsc{Magma} \citep{Magma,ford1994implementing} implementation
of the Round $2$ algorithm\footnote{column \textsc{Magma} in the tables.}.
\end{itemize}

\noindent
We apply the algorithms to rings of type $A=\mathbb{Q}[X,Y]/\langle f\rangle$,
with polynomials $f$ as specified. All timings are in seconds, taken on an Intel Xeon CPU E5-2643
 with $24$ cores, $3.4$GHz, and $384$GB of RAM running a
Linux operating system. An asterisk ($\ast$) indicates that the computation did not finish
within 3600 seconds. At current state, parallel computations are used only for
the decomposition of the singular locus. A systematic parallelization of the
integral basis algorithm and a modular approach following the strategy of
\citet{MR3378860} is subject to ongoing work. Recall that for obtaining the
integral bases, singularities at infinity of the curve $\{f=0\}$ do not matter.

\vskip0.2cm
\subsection{One Singularity of Type \texorpdfstring{${\boldsymbol{A_{k}}}$}{A\_k}}
\label{timingsAk}

A plane curve with defining polynomial
$$
f(X,Y)=Y^{2}+X^{k+1}+Y^{d}, \ k\geq 1, \ d\geq 3,
$$
has exactly one singularity at the origin. This singularity is of type $A_{k}$.%

\[%
\begin{tabular}
[c]{|r|r|r|r|r|r|}\hline
\multicolumn{1}{|c|}{\multirow{2}{*}{$k$}} & \multicolumn{1}{|c|}{\multirow{2}{*}{$d$}}  & \multicolumn{2}{|c|}{\textsc{Singular}} & \multirow{2}{*}{\textsc{Maple}} &
\multirow{2}{*}{\textsc{Magma}}\\
&  & \textsc{intbas} & \textsc{normal} &  & \\\hline
$5$   & $10$  & $0$ & $0$ & $0$   & $0$\\
$5$   & $100$ & $0$ & $0$ & $1$   & $168$\\
$5$   & $500$ & $0$ & $1$ & $46$  & $\ast$\\
$50$  & $60$  & $0$ & $0$ & $1$   & $294$\\
$50$  & $100$ & $0$ & $1$ & $2$   & $10751$\\
$50$  & $500$ & $0$ & $0$ & $65$  & $\ast$\\
$90$  & $100$ & $0$ & $1$ & $3$   & $\ast$\\
$90$  & $500$ & $0$ & $1$ & $76$ & $\ast$\\
$400$ & $500$ & $0$ & $3$ & $237$ & $\ast$\\\hline
\end{tabular}
\ \
\]

\vskip0.2cm
\subsection{One Singularity of Type \texorpdfstring{${\boldsymbol{D_{k+1}}}$}{D\_(k+1)}}

\label{timingsDk}

A plane curve with defining polynomial
$$
f(X,Y)=X(X^{k-1}+Y^{2})+Y^{d}, \ k\geq 3, \ d\geq 3,
$$
has exactly one singularity at the origin. This singularity is of type $D_{k+1}$.

\[%
\begin{tabular}
[c]{|r|r|r|r|r|r|}\hline
\multicolumn{1}{|c|}{\multirow{2}{*}{$k$}} & \multicolumn{1}{|c|}{\multirow{2}{*}{$d$}}  & \multicolumn{2}{|c|}{\textsc{Singular}} & \multirow{2}{*}{\textsc{Maple}} &
\multirow{2}{*}{\textsc{Magma}}\\
&  & \textsc{intbas} & \textsc{normal} &  & \\\hline
$5$   & $10$  & $1$ & $0$ & $0$    & $0$\\
$5$   & $100$ & $1$ & $1$ & $0$    & $1683$\\
$5$   & $500$ & $3$ & $0$ & $29$   & $\ast$\\
$50$  & $60$  & $1$ & $1$ & $3$    & $312$\\
$50$  & $100$ & $1$ & $1$ & $8$    & $3480$\\
$50$  & $500$ & $3$ & $1$ & $490$  & $\ast$\\
$90$  & $100$ & $1$ & $1$ & $27$   & $\ast$\\
$90$  & $500$ & $3$ & $1$ & $1441$ & $\ast$\\
$400$ & $500$ & $4$ & $4$ & $\ast$ & $\ast$\\\hline
\end{tabular}
\]

\vskip0.2cm
\subsection{Ordinary Multiple Points}

\label{timingsOMP}

We consider random curves of degree $d$ with an ordinary $k$-fold point at the
origin and no other singularities. The defining polynomials were generated by the function
\texttt{polyDK} from the \textsc{Singular}{} library
\texttt{integralbasis.lib} (with random seed $1231$).

\[%
\begin{tabular}
[c]{|r|r|r|r|r|r|}\hline
\multicolumn{1}{|c|}{\multirow{2}{*}{$k$}} & \multicolumn{1}{|c|}{\multirow{2}{*}{$d$}}  & \multicolumn{2}{|c|}{\textsc{Singular}} & \multirow{2}{*}{\textsc{Maple}} &
\multirow{2}{*}{\textsc{Magma}}\\
&  & \textsc{intbas} & \textsc{normal} &  & \\\hline
$5$  & $10$ & $0$ & $2$    & $0$  & $0$\\
$15$ & $20$ & $0$ & $7784$ & $1$  & $4$\\
$15$ & $30$ & $1$ & $\ast$ & $21$ & $124$\\
$20$ & $25$ & $1$ & $\ast$ & $2$  & $18$\\
$20$ & $30$ & $2$ & $\ast$ & $17$ & $42$\\\hline
\end{tabular}
\]

\vskip0.2cm
\subsection{Curves With Many Singularities of Type \texorpdfstring{${\boldsymbol{A_{k-1}}}$}{A\_(k-1)}}
\label{timingsManyAk}

If $k$ is odd, the projective plane curves with defining polynomials
\[
X^{2k}+Y^{2k}+Z^{2k}+2(X^{k}Z^{k}-X^{k}Y^{k}+Y^{k}Z^{k})
\]
have precisely $3k$ singularities in $\mathbb P^2(\CC)$, all of type $A_{k-1}$  (see
\citet{Cogolludo1999}). We consider the affine parts of these curves obtained
by substituting $Z=X-2Y+1$ (these parts contain all singularities).

\[%
\begin{tabular}
[c]{|r|r|r|r|r|}\hline
\multicolumn{1}{|c|}{\multirow{2}{*}{$k$}} & \multicolumn{2}{|c|}{\textsc{Singular}} & \multirow{2}{*}{\textsc{Maple}} &
\multirow{2}{*}{\textsc{Magma}}\\
& \textsc{intbas} & \textsc{normal} &  & \\\hline
$5$  & $0$   & $1249$ & $1$   & $1$\\
$7$  & $1$   & $\ast$ & $8$   & $8$\\
$9$  & $20$  & $\ast$ & $35$  & $59$\\
$11$ & $102$ & $\ast$ & $297$ & $251$\\\hline
\end{tabular}
\
\]

\vskip0.2cm
\subsection{More General Singularities}
\label{timingsGeneral}

We now consider some examples of curves which have singularities of a type
other than $ADE$ or ordinary multiple points:

\begin{enumerate}

\item \label{example:pfister1} { $f=-X^{15}+21X^{14}-8X^{13}%
Y+6X^{13}+16X^{12}Y-20X^{11}Y^{2}+X^{12}-8X^{11}Y+36X^{10}Y^{2}-24X^{9}%
Y^{3}-4X^{9}Y^{2}+16X^{8}Y^{3}-26X^{7}Y^{4}+6X^{6}Y^{4}-8X^{5}Y^{5}%
-4X^{3}Y^{6}+Y^{8}$: one singularity at the origin with multiplicity $m=8$ and
delta invariant $\delta=42$, one node, and one set of $6$ conjugate nodes. }

\item \label{example:pfister2} { $f=(Y^{4}+2X^{3}Y^{2}+X^{6}%
+X^{5}Y)^{3}+X^{11}Y^{11}$: one singularity at the origin with $m=12$ and
$\delta=133$.}

\item \label{example:3Branches} {$f=(Y^{5}+Y^{4}X^{7}%
+2X^{8})(Y^{3}+7X^{4})(Y^{7}+2X^{12})(Y^{11}+2X^{18})+Y^{30}$: one singularity
at the origin with $m=26$ and $\delta=523$.}

\item \label{example:2BranchesDeg36} { $f=(Y^{15}+2X^{38}%
)(Y^{19}+7X^{52})+Y^{36}$: one singularity at the origin with $m=34$ and
$\delta=1440$.}

\item \label{example:2BranchesDeg100}{ $f=(Y^{15}+2X^{38}%
)(Y^{19}+7X^{52})+Y^{100}$: same type of singularity as (4), but of higher degree.}

\item \label{example:vanHoeijDeg40} { \label{exam:Hoeij1}
$f=Y^{40}+XY^{13}+X^{4}Y^{5}+X^{5}+2X^{4}+X^{3}$: one double point with
$\delta=2$ and one triple point with $\delta=19$ (see \citet[Section
6.1]{vanHoeij94}).}

\item \label{example:vanHoeijDeg100}{ $f=Y^{200}+XY^{13}+X^{4}%
Y^{5}+X^{5}+2X^{4}+X^{3}$: same type of singularities as (6), but higher degree.}

\item \label{example:4Branches} { $f=(Y^{35}+Y^{34}X^{7}%
+2X^{38})(Y^{33}+7X^{44})(Y^{37}+2X^{52})+Y^{110}$: one singularity at the
origin with $m=105$ and $\delta=6528$.}
\end{enumerate}
\[%
\begin{tabular}
[c]{|r|r|r|r|r|r|}\hline
\multicolumn{1}{|c|}{\multirow{2}{*}{No.}} & \multicolumn{1}{|c|}{\multirow{2}{*}{$Y$-degree}} & \multicolumn{2}{|c|}{\textsc{Singular}} & \multirow{2}{*}{\textsc{Maple}} &
\multirow{2}{*}{\textsc{Magma}}\\
&  & \textsc{intbas} & \textsc{normal} &  & \\\hline
(\ref{example:pfister1})        & $8$   & $1$    & $\ast$ & $0$    & $0$\\
(\ref{example:pfister2})        & $12$  & $8$    & $\ast$ & $1$    & $1$\\
(\ref{example:3Branches})       & $30$  & $16$   & $\ast$ & $4$    & $31$\\
(\ref{example:2BranchesDeg36})  & $36$  & $1$    & $\ast$ & $4$    & $59$\\
(\ref{example:2BranchesDeg100}) & $100$ & $1$    & $\ast$ & $28$   & $\ast$\\
(\ref{example:vanHoeijDeg40})   & $40$  & $0$    & $2$    & $0$    & $9$\\
(\ref{example:vanHoeijDeg100})  & $200$ & $2$    & $3$    & $10$   & $\ast$\\
(\ref{example:4Branches})       & $110$ & $\ast$ & $\ast$ & $\ast$ & $\ast$\\\hline
\end{tabular}
\
\]

\vskip0.2cm
In Example (\ref{example:4Branches}), \textsc{Singular} and \textsc{Maple} do
not finish due to the computation of the decomposition of the singular locus
of the curve (note that the criterion given in Proposition \ref{prop:one-singular-point}
does not apply since the curve has a singularity at infinity). See Section \ref{section:examplesDetails} for the timings of the
computation of the local contribution to the integral basis (normalization) at the origin.

\subsection{Summary}
\label{timingsSummary}

We note that in most cases,  our new algorithm outperforms the other algorithms by far.

\subsection{A More Detailed Analysis of Some of the Examples}
\label{section:examplesDetails}

The computation of an integral basis with our algorithm has two major
components. First, we decompose the singular locus into associated primes;
second, we compute the local contribution to the integral basis at each
prime and combine the results.
In \textsc{Maple}, a similar strategy is followed in that  all the
$X$\hyp{}coordinates of the singular points are computed first. With both
approaches, the first step may be time consuming. To
analyze the difference between the two approaches in more detail,
we provide timings for the computation of the integral basis at the
origin in examples where it is known to us that the origin is the only
singularity. This fact can be specified in \textsc{Singular} and
\textsc{Maple} by appropriate input options.%

\[%
\begin{tabular}
[c]{|c|r|r|r|r|}\hline
\multirow{2}{*}{Example} & \multicolumn{1}{|c|}{\multirow{2}{*}{$k$}} & \multicolumn{1}{|c|}{\multirow{2}{*}{$d$}} & \multicolumn{1}{|c|}{\textsc{Singular}} & \multirow{2}{*}{\textsc{Maple}}\\
&  &  & \multicolumn{1}{|c|}{\textsc{intbas}} & \\\hline
\ref{timingsAk}                                        & $5$   & $500$ & $0$  &   $44$\\
\ref{timingsAk}                                        & $50$  & $500$ & $0$  &   $58$\\
\ref{timingsAk}                                        & $400$ & $500$ & $0$  &   $235$\\
\ref{timingsDk}                                        & $5$   & $500$ & $2$  &   $24$\\
\ref{timingsDk}                                        & $50$  & $500$ & $2$  &   $251$\\
\ref{timingsDk}                                        & $400$ & $500$ & $2$  &   $\ast$\\
\ref{timingsOMP}                                       & $15$  & $30$  & $0$  &   $19$\\
\ref{timingsOMP}                                       & $20$  & $25$  & $0$  &   $2$\\
\ref{timingsOMP}                                       & $20$  & $30$  & $0$  &   $14$\\
\ref{timingsGeneral}\! (\ref{example:pfister2})        &       &       & $8$  &   $1$\\
\ref{timingsGeneral}\! (\ref{example:3Branches})       &       &       & $2$  &   $2$\\
\ref{timingsGeneral}\! (\ref{example:2BranchesDeg100}) &       &       & $1$  &   $13$\\
\ref{timingsGeneral}\! (\ref{example:4Branches})       &       &       & $15$ &   $1343$\\\hline
\end{tabular}
\
\]

\vskip0.2cm
\noindent
We compare the timings above with those in the previous tables and
observe first that for the examples in \ref{timingsAk}, the time required
for decomposing the singular locus can be neglected. For the examples in
\ref{timingsDk}, \textsc{Maple} spends plenty of time for the decomposition,
but the part consuming most of the time is nevertheless the computation of the
integral basis at the origin. For the examples in \ref{timingsOMP}, our algorithm
uses most of the time for decomposing the singular locus. The computation of
the integral basis at the origin is significantly faster than that in \textsc{Maple}.
From among the examples in \ref{timingsGeneral}, Example
\ref{timingsGeneral}\! (\ref{example:pfister2}) sticks out: while the time for the initial
decomposition is not significant, the computation of the integral basis at the origin
when running our algorithm in \textsc{Singular}{} is slower than
that using van Hoeij's algorithm in \textsc{Maple}. In this example, the
algorithm runs into an algebraic field extension of high degree. At current state, the
handling of such extensions in \textsc{Singular}{} is not optimal.

\bibliographystyle{elsarticle-harv}
\bibliography{mybib}

\end{document}